\documentclass[11pt]{article}
\usepackage[utf8]{inputenc}
\usepackage{amssymb}
\usepackage{amsmath}
\usepackage{amsthm}
\usepackage{hyperref}
\usepackage{a4wide}
\usepackage{breqn}
\usepackage[table,xcdraw]{xcolor}
\usepackage{xcolor}
\usepackage{mathtools}
\usepackage{bm}
\usepackage{enumitem}
\usepackage{mathrsfs}
\usepackage{graphicx}
\usepackage{caption}
\usepackage{url}
\usepackage{euscript}
\usepackage{placeins}
\usepackage{mathdots}
\usepackage{multirow}
\usepackage{arydshln} 
\usepackage[english]{babel}
\usepackage{array}
\usepackage{subcaption}
\usepackage{booktabs}
\usepackage[numbers]{natbib}
\usepackage{cases}
\usepackage{graphics}
\usepackage{graphicx}
\usepackage{algorithm}
\usepackage{algpseudocode}

\usepackage{graphics}
\usepackage{natbib}
\numberwithin{equation}{section}
\newcommand{\dm}[1]{{\displaystyle{#1}}}

\newtheorem{theorem}{\bf Theorem}[section]
\newtheorem{definition}{Definition}[section]
\newtheorem{corollary}{Corollary}[section]
\newtheorem{proposition}{Proposition}[section]

\newtheorem{lemma}{Lemma}[section]

\newtheorem{remark}{Remark}[section]

\theoremstyle{remark}
\newtheorem{exam}{\bf Example}
\newcommand{\vone}{\vskip 2ex}

\def \R{{\mathbb R}}
\def \P{{\mathscr{P}}}

\def\bmatrix#1{\left[\begin{matrix}
		#1
	\end{matrix}\right]}
\def \diag{\mathrm{diag}}
\def \A{\mathcal A}

\def \d{{\bf d}}

\def \g{{\bf g}}

\newcommand{\iu}{{i\mkern1mu}}

\def\bmatrix#1{\left[ \begin{matrix} #1 \end{matrix} \right]}
\def \noin{\noindent}

\def \R{{\mathbb R}}

\date{}

\title{A robust parameterized enhanced shift-splitting preconditioner for three-by-three block saddle point problems}
\author{Sk. Safique Ahmad\footnotemark[2] \footnotemark[1] \and Pinki Khatun \footnotemark[3] }
\begin{document}

\maketitle
 
	\begin{abstract}
		This paper proposes a new parameterized enhanced shift-splitting {\it (PESS)} preconditioner to solve the three-by-three block saddle point problem  ({\it SPP}). Additionally, we introduce a local \textit{PESS} (\textit{LPESS}) preconditioner by relaxing the \textit{PESS} preconditioner. 
  Necessary and sufficient criteria are established for the convergence of the proposed \textit{PESS} iterative process for any initial guess. Furthermore, we meticulously investigate the spectral bounds of the \textit{PESS} and \textit{LPESS} preconditioned matrices. Moreover, empirical investigations have been performed for the sensitivity analysis of the proposed \textit{PESS} preconditioner, which unveils its robustness. Numerical experiments are carried out to demonstrate the enhanced efficiency and robustness of the proposed  \textit{PESS} and \textit{LPESS} preconditioners  {compared to the existing state-of-the-art preconditioners.}
	\end{abstract}
		   
		

\footnotetext[2]{
	Department of Mathematics, Indian Institute of Technology Indore, Khandwa Road, Indore, 453552, Madhya Pradesh, India. \texttt{Email:\,safique@iiti.ac.in}, \texttt{safique@gmail.com}}
	\footnotetext[3]{
		Research Scholar, Department of Mathematics, Indian Institute of Technology Indore, Khandwa Road, Indore, 453552, Madhya Pradesh, India. \texttt{Email:\,phd2001141004@iiti.ac.in}, \texttt{pinki996.pk@gmail.com}}
\footnotetext[1]{Corresponding author.}
 	\section{Introduction}
	Three-by-three block saddle point problems {\it (SPPs)} have recently received considerable attention due to their wide applicability in various computational research and technological applications. These domains include  the constraint optimization problems \cite{OPTMIZATION1},  
	computational fluid dynamics \cite{fluid1991, Matlab2007}, the constrained least square estimations \cite{LSP}, 
	optimal control \cite{OPTIMAL}, 
	solving  Maxwell equations \cite{MaxwellEQ}, magma-mantle dynamics \cite{magma-mantle}   
	and so on. The three-by-three block {\it SPP} is represented as follows
	\cite{HuangNA}:
	\begin{align}\label{eq11}
		\mathcal{A}{\bf u}= {\bf d}, 
	\end{align}
	where $\mathcal{A}=\bmatrix{A &B^T & 0\\-B &0& -C^T\\0 & C & 0},\,  {\bf u}=\bmatrix{x\\y\\z} , \, {\bf d}=\bmatrix{f\\g\\h},$ $A\in \R^{n\times n}$ is a symmetric positive definite {\it (SPD)} matrix, $B\in \R^{m\times n}$ and $C\in \R^{p\times m}$ are the full row rank matrices.  Here, $B^T$  {and $C^T$ stand for the transpose of  $B$ and $C$, respectively;}  $f\in \R^n, \, g\in\R^m$ and $h\in \R^p$ are known vectors. Under the stated assumptions on block matrices A, B and C,  $\mathcal{A}$ is nonsingular, which ensures the existence of a unique solution for the system \eqref{eq11}.
	\vspace{0.2 cm}
	
	Due to the growing popularity of three-by-three block {\it SPPs} in a variety of applications, it is essential to determine the efficient and robust numerical solution of three-by-three block {\it SPPs} \eqref{eq11}. However,  the indefiniteness and poor spectral properties of  $\mathcal{A}$ makes challenging to determine the solution  of the system \eqref{eq11}. Furthermore,  $\A$ is generally sparse and large; therefore, iterative processes become superior to direct methods. The Krylov subspace iterative process is preferable among other iterative processes investigated widely in the literature due to their minimal storage requirement and feasible implementation \cite{wathen2015}. 
	Slow convergence of the Krylov subspace iterative process is a major drawback for  system of linear equations of large dimensions. Moreover, the saddle point matrix $\A$ can be very sensitive, which leads to a significant
	slowdown in the solution algorithm. Therefore, it is important to develop novel, efficient and robust preconditioners for the fast convergence of the Krylov subspace iterative process, which can handle the sensitivity of the problem \eqref{eq11} when the matrix $\mathcal{A}$ is ill-conditioned.  
 	
	
	Consider the following partitioning of  $\A:$ 
	\begin{align}\label{EQ13}
		\left[ \begin{array}{cc|c}
			A &B^T  & 0 \\
			-B & 0 & -C^T  \\ 
			\hline
			0 & C & \multicolumn{1}{c}{%
				0} \\
		\end{array} \right] \quad \mbox{and}\quad \left[ \begin{array}{c|cc}
			A &B^T  & 0 \\
			\hline
			-B & 0 & -C^T  \\ 
			0 & C & 0   \\
		\end{array} \right].
	\end{align}
	Then the linear system \eqref{eq11} can be identified as the standard two-by-two block {\it SPP}  {\cite{MBenzi2005}}. 
	In the last decades, colossal attention has been given to the development of numerical solution methods for standard two-by-two block {\it SPP}, such as Uzawa methods  {\cite{ZZBaiUzawa, Uzawa2014}, Hermitian and skew-Hermitian splitting-type (\textit{HSS}-type) methods \cite{HSS, HSS2019}, successive overrelaxation-type (\textit{SOR}-type) methods \cite{SOR2001, GSOR2005} and} null space methods \cite{nullspace2001, nullspace1998}.  {For a more recent comprehensive survey see \cite{ZZBai2021}.}  However, the properties of the submatrices in \eqref{EQ13} are different from the standard two-by-two block {\it SPP}. Notice that the leading block is not \textit{SPD} in the first partitioning. The second partitioning highlights that the $(1, 2)$ block is rank deficient and the $(2, 2)$ block is skew-symmetric. In contrast, in the standard two-by-two block {\it SPP}, $(1,2)$  block is full row rank, and the $(2, 2)$ block is generally symmetric positive semidefinite. Therefore, the existing literature for solving the two-by-two block {\it SPPs}  (see \cite{MBenzi2005}) may not be applied to solve \eqref{eq11} directly and it is essential to develop new preconditioners for three-by-three block {\it SPPs} that exploit the specific structure of the saddle point matrix $\A$, which is sensitive in nature. 
 
	Recently, various effective preconditioners have been developed in the literature for solving three-by-three block {\it SPPs}.   For instance, \citet{HuangNA}  have studied the exact block diagonal {\it (BD)} and inexact {\it BD (IBD)} preconditioners $\P_{BD}$ and 
	$\P_{IBD}$ in the following forms:
	\begin{align}
		\P_{BD}=\bmatrix{A& 0 &0\\0& S& 0\\0 &0&CS^{-1}C^T}\quad \mbox{and} \quad \P_{IBD}=\bmatrix{\widehat{A}& 0 &0\\0&\widehat{S}& 0\\0 &0&C\widehat{S}^{-1}C^T}, 
	\end{align}
	where $S= BA^{-1}B^T,$ $\widehat{A}$ and $\widehat{S}$ are {\it SPD} approximations of $A$ and $S,$ respectively. Although the spectrum of the preconditioned matrices  $\P^{-1}_{BD}\A$ and $\P^{-1}_{IBD}\A$ have good clustering properties, these preconditioners have certain shortcomings, such as they are time-consuming, expensive, require excessive number of iterations and condition numbers are very large.  For more details, see the paper \cite{HuangNA}.  {Inspired by the \textit{HSS} iteration method \cite{HSS}, \citet{APPS2021} split the coefficient matrix \(\A\) and  presented the  alternating positive semidefinite splitting (\textit{APSS}) iteration method. They proved the iteration method is unconditional convergence and proposed the corresponding \textit{APSS} preconditioner to solve the system \eqref{eq11}. Moreover, to improve the effectiveness of the \textit{APSS} preconditioner, many relaxed and modified versions of \textit{APSS} have been designed; see \cite{MAPSS2023, SRAPSS2023, RAPSS2024}. For instance, by relaxing the $(1,1)$ block and introducing a new parameter $\beta>0$  in the \textit{APSS} preconditioner, \citet{MAPSS2023} proposed the following modified \textit{APSS} (\textit{MAPSS}) preconditioner: 
     \begin{equation}
     \P_{MAPSS}=\bmatrix{A & B^T & -\frac{1}{\alpha}B^TC^T\\ -B & \alpha I &-C^T\\ 0 &C & \beta I},
 \end{equation}
 where $\alpha>0$ and $I$  stands for the identity matrix of the appropriate dimension. 
Motivated by the work in \cite{SPLITTING2011} and by incorporating the ideas of shifting and the \textit{BD} preconditioner, the authors in \cite{Twoblock23} proposed two preconditioners along with their inexact variants. One of these preconditioners is given by:
 \begin{equation}
    \P_{SL}=\bmatrix{A & B^T & 0\\ -B & C^TC & 0\\ 0 &C & I}.
 \end{equation} 
%
\noindent{Three block preconditioners are developed by \citet{ANOTE2020} for the system \eqref{eq11} with the equivalent symmetric coefficient matrix.}
It is also demonstrated in \cite{ANOTE2020} that the preconditioned matrices possess no more than three distinct eigenvalues. \citet{NaHuangVUZ} proposed a variant of the Uzawa iterative method for the theree-by-three block \textit{SPP} by introducing two variable parameters. Additionally, \citet{NaHuangUZ19} generalized the well-known Uzawa method to solve the linear system \eqref{eq11} and propose the inexact Uzawa method. In addition to the preconditioners discussed above, recent literature \cite{BDMaryam, PBSPD2022, NLAA2024} have introduced several other preconditioners to solve the three-by-three block \textit{SPP} \eqref{eq11}.}

 {The shift-splitting {\it (SS)} preconditioners were initially developed for a non-Hermitian system of linear equations by \citet{BaiSS} and later for two-by-two block {\it SPPs} \cite{CaoSS2014, GSS2015, CaoSS2017}}.
 Recently, \citet{CAOSS19} enhanced this idea and introduced the following {\it SS} preconditioner $\P_{SS}$ and  relaxed {\it SS} ({\it RSS}) preconditioner $\P_{RSS}$ for the three-by-three block {\it SPP} \eqref{eq11}:
		\begin{equation}\label{eq14}
			\P_{SS}=\frac{1}{2}\bmatrix{\alpha I+A& B^T &0\\-B&\alpha I& -C^T\\0 &C&\alpha I} \,\, \mbox{and}\, \, \P_{RSS}=\frac{1}{2}\bmatrix{A& B^T &0\\-B&\alpha I& -C^T\\0 &C&\alpha I},
		\end{equation}
		where $\alpha>0$ and verified unconditionally convergence of the associated {\it SS} iterative process.  \citet{wang2019GSS} generalized the {\it SS} preconditioner by introducing a new parameter $\beta >0$ in the $(3, 3)$ block.    {By merging the lopsided and shift technique, a lopsided {\it SS} preconditioner is presented by \citet{LSS22}}. 
  Further, \citet{EGSS23}  developed  {the following} extensive generalized {\it SS (EGSS)} preconditioner: 
		\begin{equation}\label{eq15}
			\P_{EGSS}=\frac{1}{2}\bmatrix{\alpha P+A& B^T &0\\-B&\beta Q& -C^T\\0 &C&\gamma W}
		\end{equation}
		for the three-by-three block {\it SPP} \eqref{eq11}, where $\alpha, \beta, \gamma> 0$ and $P, Q, W$ are {\it SPD} matrices and investigated its convergent properties.  { By relaxing the $(1,1)$ block and eliminating  the prefactor $1/2,$  \citet{IEGSS2024} have proposed relaxed and preconditioned generalized shift-splitting 
(\textit{RPGSS}) preconditioner 
\begin{eqnarray}
    \P_{RPGSS}=\bmatrix{A & B^T & 0 \\ -B & \beta Q& -C^T \\ 0 & C & \gamma  W},
\end{eqnarray}
 where $Q\in \R^{m\times m}$ and $ W\in \R^{p\times p}$ are \textit{SPD} and $ \beta, \gamma>0$. }
	
	
	Despite exhibiting favorable performance of  {\it SS, RSS} and {\it EGSS} preconditioners,  they do not outperform {\it BD} and {\it IBD} preconditioners. Therefore, to improve the convergence speed and efficiency of the preconditioners $\P_{SS},\P_{RSS}$ and $\P_{EGSS},$  this paper presents a parameterized enhanced shift-splitting (\textit{PESS}) preconditioner by introducing a parameter $s> 0$ for there-by-three block {\it SPPs.} 
	As per the direct correlation between the rate of convergence in Krylov subspace iterative processes and the spectral properties of the preconditioned matrix, we perform in-depth spectral distribution of the \textit{PESS} preconditioned matrix.
 We  summarize  the main contributions of this paper as follows:
	
	\begin{itemize}
		\item We propose the \textit{PESS} iterative process along with its associated \textit{PESS} preconditioner and its relaxed version known as the local \textit{PESS} (\textit{LPESS}) preconditioner to solve the three-by-three block {\it SPP} \eqref{eq11}. The proposed preconditioners generalize the existing {\it SS} preconditioners, for example,  \cite{CAOSS19, wang2019GSS, EGSS23}. 
		
		\item General framework is given on necessary and sufficient criteria for the convergence of the \textit{PESS} iterative process. These investigations also encompass the unconditional convergence of other exiting  {\it SS} preconditioners in the literature \cite{CAOSS19, EGSS23}.
		\item We have conducted a comprehensive analysis of the spectral distribution for the \textit{PESS}  {and \textit{LPESS} preconditioned matrices}. Our framework allows us to derive the spectral distribution for the  {existing} {\it SS} and {\it EGSS} preconditioned matrices.
		\item We empirically show that our proposed preconditioner significantly reduces the condition number of the ill-conditioned saddle point matrix $\A$. Thereby establishing an efficiently solvable, well-conditioned system.
		
		\item Numerical experiments show that the proposed \textit{PESS} and \textit{LPESS} preconditioners outperform all the compared baseline preconditioners. Moreover, we perform sensitivity analysis by perturbing the ill-conditioned matrix $\A$ in \eqref{eq11}, demonstrating the robustness of the proposed preconditioner.
	\end{itemize}
	\vspace{0.2 cm}
	The structure of this article is as follows. In Section \ref{sec2}, we propose the \textit{PESS} iterative process and the associated \textit{PESS} preconditioner. Section \ref{sec3} is devoted to investigating the convergence of the \textit{PESS} iterative process. In Section \ref{sec4}, we investigate the spectral distribution of the \textit{PESS} preconditioned matrix. In Section \ref{sec:LPESS}, we present the \textit{LPESS} preconditioner.  {In Section \ref{Sec:parameter}, we discuss strategies for selecting the 
 appropriate parameters for the proposed preconditioners.}  In Section \ref{sec5}, we conduct numerical experiments to illustrate the computational efficiency and robustness of the proposed preconditioner. In the end, a brief concluding remark is provided in Section \ref{sec6}.
	
	\vspace{0.2 cm}
	
	\noin{\bf Notation.}  The following notations are adopted throughout the paper. For any matrix $A\in \R^{n\times n},$ $A^T, \sigma(A)$ and $\vartheta(A)$ denote  transpose,  spectrum and spectral radius of $A,$ respectively. For any $x\in \mathbb{C}^{n},$ $x^*$ denotes its conjugate transpose. The symbol $\diag(A, B)$ represents a block diagonal matrix with diagonal blocks $A$ and $B.$ The real and imaginary parts of any complex number $z$ are represented by $\Re(z)$ and $\Im(z)$, respectively.  For any matrix $A\in \R^{n\times n}$ having real eigenvalues,  $\lambda_{\max}(A)$ and $\lambda_{\min}(A)$ denote the maximum and minimum eigenvalue of $A,$ respectively. The notation $A> 0$ signifies that $A$ is {\it SPD} and $A>B$ identifies $A-B>0.$ \\\noin{\bf Abbreviations used in the paper are as follows:}\\
	{\it SPP}: saddle point problem, {\it SPD}: symmetric positive definite, {\it BD}: block diagonal, {\it IBD}: inexact block diagonal,  {\textit{MAPSS}: modified alternating positive semidefinite splitting},   {\it SS}: shift-splitting, {\it RSS}: relaxed shift-splitting, {\it EGSS}: extensive generalized shift-splitting,  {{\it RPGSS}: relaxed and preconditioned generalized shift-splitting},  \textit{PESS}: parameterized enhanced shift-splitting, \textit{LPESS}: local parameterized enhanced shift-splitting, {\it GMRES}: generalized minimum residual process.

 \section{The parameterized enhanced shift-splitting  (\textit{PESS}) iterative process and the  preconditioner}\label{sec2} 
In this section, we proposed a parameterized enhanced shift-splitting (\textit{PESS}) iterative process and the corresponding \textit{PESS} preconditioner.
Let $s>0$ be a real number. Then, we split the  matrix $\A$ in the form 
	\begin{equation}\label{eq21}
		\A=(\Sigma+s\A)-(\Sigma-(1-s)\A)=: \P_{PESS}-\mathcal{Q}_{PESS}, \, \mbox{where} 
	\end{equation}
	\[\P_{PESS} =\bmatrix{\Lambda_1+sA & sB^T & 0\\ -sB &\Lambda_2&-sC^T\\ 0& sC &\Lambda_3},
	\mathcal{Q}_{PESS}= \bmatrix{\Lambda_1-(1-s)A & -(1-s)B^T & 0\\ (1-s)B &\Lambda_2&(1-s)C^T\\ 0& -(1-s)C &\Lambda_3},\]
	$\Sigma=\diag(\Lambda_1,\Lambda_2,\Lambda_3)$ 	and $\Lambda_1,\Lambda_2,\Lambda_3$ are {\it SPD} matrices. The matrix $\P_{PESS}$ is nonsingular for $s>0.$ The special matrix splitting \eqref{eq21} leads us to the subsequent iterative process for solving the three-by-three block {\it SPP} given in \eqref{eq11}.
	
	{\bf The \textit{PESS} iterative process.} {\it Let $s$ be a positive constant and let $\Lambda_1\in \R^{n\times n},\Lambda_2\in \R^{m\times m} {,}$  $\Lambda_3\in \R^{p\times p}$ be {\it SPD} matrices. For any initial guess ${\bf u}_0\in \R^{n+m+p}$ and $k=0,1,2\ldots,$ until  a specified termination criterion is fulfilled, we compute 
		\begin{align}\label{eq22}
			{\bf u}_{k+1}=\mathcal{T}{\bf u}_k+{\bf c},
		\end{align}	    
			where   $ {\bf u}_{k}=[x_{k}^T, y_{k}^T, z_{k}^T]^T, {\bf c}=\P_{PESS}^{-1}\d$
			and  $\mathcal{T}=\P_{PESS}^{-1}\mathcal{Q}_{PESS}$ is called the iteration matrix for the PESS iterative process. }
		
	Moreover, the matrix splitting \eqref{eq22} introduces a new preconditioner $\P_{PESS}$, identified as {\textit{PESS}} preconditioner. 	 {This preconditioner generalizes the previous work in the literature;  for example, refer to} \cite{CAOSS19,wang2019GSS, EGSS23} for specific choices of $\Lambda_1,\Lambda_2,\Lambda_3$ and $s,$ which are listed in Table 1.
		\begin{table}[ht!]
			\centering
			\caption{$\P_{PESS}$ as a generalization of the above {\it SS} preconditioners for different choices of $\Lambda_1,\Lambda_2,\Lambda_3$ and $s.$}\label{tab1}
			\begin{tabular}{c|c|c|c|c|c}
				 \hline
				$\P_{PESS}$  & $\Lambda_1$& $\Lambda_2$&$\Lambda_3$ &$s$& memo\\ [1ex]
				\hline
				$\P_{SS}$ \cite{CAOSS19}& $\frac{1}{2}\alpha I$&$\frac{1}{2}\alpha I$ &$\frac{1}{2}\alpha I$ &$s=\frac{1}{2}$&$\alpha>0$\\  [1ex]
				\hline
				$\P_{GSS}$ \cite{wang2019GSS}& $\frac{1}{2}\alpha I$& $\frac{1}{2}\alpha I$&$\frac{1}{2}\beta I$ &$s=\frac{1}{2}$&$\alpha,\beta>0$\\ [1ex]
				\hline
				$\P_{EGSS}$ \cite{EGSS23}&$\frac{1}{2}\alpha P$ &$\frac{1}{2}\beta Q$ &$\frac{1}{2}\gamma W$ &$s=\frac{1}{2}$& $P,Q,W$ are {\it SPD}  \\ &&&&&  matrices and $\alpha,\beta,\gamma>0$\\
				   \hline
			\end{tabular}
		\end{table}

		In the implementation of the {\textit{PESS}} iterative process or the {\textit{PESS}} preconditioner to enhance  the rate of convergence of the Krylov subspace iterative process like {\it GMRES}, at each iterative step,  we  solve the following system of linear equations:
		\begin{eqnarray}\label{eq25}
			\P_{PESS}w= r,
		\end{eqnarray} 
		where  $r= [r_1^T, r_2^T, r_3^T]^T\in \R^{n+ m+ p}$ and $w= [w_1^T, w_2^T, w_3^T]^T\in \R^{n+ m+ p}.$ Specifying $\widehat{X}=\Lambda_2+ s^2C^T\Lambda_3^{-1}C$ and $\widetilde{A}=\Lambda_1+ sA+ s^2B^T\widehat{X}^{-1}B,$ then  $\P_{PESS}$ can be written in the following way:
		\begin{eqnarray}\label{eq26}
			\P_{PESS}=\bmatrix{I &sB^T\widehat{X}^{-1}&0\\0&I&-sC^T\Lambda_3^{-1}\\0&0&I}\bmatrix{\widetilde{A}&0&0\\0&\widehat{X}&0\\0 &0 &\Lambda_3}\bmatrix{I&0&0\\-s\widehat{X}^{-1}B &I&0\\0&s\Lambda_3^{-1}C&I}.
		\end{eqnarray}
		The decomposition \eqref{eq26} leads us to the following Algorithm \ref{alg1} to determine the solution of the system \eqref{eq25}.
		\begin{algorithm}
			\caption{Computation of  $w$ from $\P_{PESS}w=r$}\label{alg1}
			\begin{algorithmic}[1]
				\State Solve $\widehat{X}v_1=r_2+sC^T\Lambda_3^{-1}r_3$ to find $v_1;$
				\State Compute $v=r_1-sB^Tv_1;$
				\State Solve $\widetilde{A}w_1=v$ to find $w_1;$
				\State Solve $\widehat{X}v_2=sBw_1$ to find $v_2$; 
				\State Compute $w_2=v_1+v_2;$
				\State Solve $\Lambda_3w_3=r_3-sCw_2$ for $w_3.$
			\end{algorithmic}
		\end{algorithm}
		The implementation of Algorithm \ref{alg1} requires to solve two linear subsystems with coefficient matrices  $\widehat{X}$ 
		and $\widetilde{A}$. Since $\Lambda_1, \Lambda_2, \Lambda_3, \widehat{X}$ and $\widetilde{A}$ are {\it SPD} and $s> 0$,  we can apply the Cholesky factorization to solve 
		these linear subsystems exactly or inexactly by the preconditioned conjugate gradient method. 
\section{Convergence analysis of the \textit{PESS} iterative process}\label{sec3}
		
		In this section, we investigate the convergence of the {\textit{PESS}} iterative process \eqref{eq22}. To achieve this aim, the following definition and lemmas are crucial.
		\begin{definition}
			A	matrix $\A$ is  called positive stable if $\Re(\lambda)>0$  for all $\lambda\in \sigma(\A)$.  
		\end{definition}
  \begin{lemma}\cite{CAOSS19}\label{lm21}
			Let $A\in\R^{n\times n}$ be a {\it SPD} matrix and let $B\in \R^{m\times n}$ and $C\in \R^{p\times m}$  be full row rank matrices. Then, the saddle point matrix $\A$ is positive stable.
		\end{lemma}
		\begin{lemma}\label{lm22}
			Let  $A\in\R^{n\times n}$ be a {\it SPD} matrix and let $B\in \R^{m\times n}$ and $C\in \R^{p\times m}$ be full row rank matrices. Then $\Sigma^{-1}\A$ $($or $\Sigma^{-\frac{1}{2}}\A \Sigma^{-\frac{1}{2}})$ is positive stable.
		\end{lemma}
		\proof 
	Since $\Sigma^{-\frac{1}{2}}=\diag(\Lambda_1^{-\frac{1}{2}}, \Lambda_2^{-\frac{1}{2}}, \Lambda_3^{-\frac{1}{2}}),$	 we obtain
		\begin{eqnarray}
			\Sigma^{-\frac{1}{2}}\A \Sigma^{-\frac{1}{2}}=\bmatrix{\Lambda_1^{-\frac{1}{2}}A\Lambda_1^{-\frac{1}{2}}& \Lambda_1^{-\frac{1}{2}}B^T\Lambda_2^{-\frac{1}{2}}&0\\-\Lambda_{2}^{-\frac{1}{2}}B\Lambda_1^{-\frac{1}{2}}& 0&\Lambda_2^{-\frac{1}{2}}C^T\Lambda_3^{-\frac{1}{2}}\\ 0&\Lambda_3^{-\frac{1}{2}}C\Lambda_2^{-\frac{1}{2}}&0}.
		\end{eqnarray}
		 { Given that the} matrices $\Lambda_1,\Lambda_2$ and $\Lambda_3$ are {\it SPD}, then $\Lambda_1^{-\frac{1}{2}}A\Lambda_1^{-\frac{1}{2}}$ is {\it SPD}.  { Furthermore, the matrices}  $\Lambda_2^{-\frac{1}{2}}B\Lambda_1^{-\frac{1}{2}}$ and $\Lambda_3^{-\frac{1}{2}}C\Lambda_2^{-\frac{1}{2}}$ are of full row rank. Consequently, the block structure of  $\A$ and the matrix $ \Sigma^{-\frac{1}{2}}\A \Sigma^{-\frac{1}{2}}$ are identical. 
		Then by Lemma \ref{lm21}, $\A$ is positive stable implies that the matrix  $ \Sigma^{-\frac{1}{2}}\A \Sigma^{-\frac{1}{2}}$ and  consequently, $\Sigma^{-1}\A$ is positive stable. $\blacksquare$
		
		By  \cite{YSAAD},  the stationary iterative process \eqref{eq22} converges to the exact solution of the three-by-three block {\it SPP} \eqref{eq11} for any initial guess vector if and only if   $|\vartheta(\mathcal{T})| < 1,$ where $\mathcal{T}$ is the iteration matrix.
		Now, we establish an if and only if condition that precisely determines the convergence of the {\textit{PESS}} iterative process.
		\begin{theorem}\label{th21}
			Let  $A\in\R^{n\times n}$ be a {\it SPD} matrix and let $B\in \R^{m\times n}$ and $C\in \R^{p\times m}$ be full row rank matrices and $s>0$.  Then, the PESS iterative process \eqref{eq22} converges to the unique solution of the three-by-three block SPP \eqref{eq11} if and only if 
			\begin{equation}\label{eq211}
				(2s-1)|\mu|^2+2\Re(\mu)>0, \,\,\forall  \, \mu\in \sigma(\Sigma^{-\frac{1}{2}}\A\Sigma^{-\frac{1}{2}}).
			\end{equation}
		\end{theorem}
		\proof
		From \eqref{eq22} we have 
		\begin{equation}
			\mathcal{T}=\P_{PESS}^{-1}\mathcal{Q}_{PESS}=(\Sigma+s\A)^{-1}(\Sigma-(1-s)\A),
		\end{equation}
		where $\Sigma=\diag(\Lambda_1,\Lambda_2,\Lambda_3).$ Given that  $\Lambda_1,\Lambda_2$ and $\Lambda_3$ are {\it SPD}, then 
		\begin{equation}\label{eq213}
			\mathcal{T}=\Sigma^{-\frac{1}{2}}(I+s\Sigma^{-\frac{1}{2}}\A\Sigma^{-\frac{1}{2}})^{-1}(I-(1-s)\Sigma^{-\frac{1}{2}}\A\Sigma^{-\frac{1}{2}})\Sigma^{\frac{1}{2}}.
		\end{equation}
		Thus,  (\ref{eq213}) shows that the iteration matrix $\mathcal{T}$ is similar to $\widetilde{\mathcal{T}},$ where 
		\[\widetilde{\mathcal{T}}=(I+s\Sigma^{-\frac{1}{2}}\A\Sigma^{-\frac{1}{2}})^{-1}(I-(1-s)\Sigma^{-\frac{1}{2}}\A\Sigma^{-\frac{1}{2}}).
		\]
		Let $\theta$ and $\mu$ be the eigenvalues of the matrices $\widetilde{\mathcal{T}}$ and $\Sigma^{-\frac{1}{2}}\A\Sigma^{-\frac{1}{2}},$ respectively. Then, it is easy to show that 
		\[\theta=\frac{1-(1-s)\mu}{1+s\mu}.
		\]
		Since $\Sigma^{-\frac{1}{2}}\A\Sigma^{-\frac{1}{2}} $ is nonsingular then  $\mu\neq 0$ and  we have
		\begin{equation}\label{eq215}
			|\theta|=\frac{|1-(1-s)\mu|}{|1+s\mu|}=\sqrt{\frac{(1-(1-s)\Re(\mu))^2+(1-s)^2\Im(\mu)^2}{(1+s\Re(\mu))^2+s^2\Im(\mu)^2}}.
		\end{equation}
		From (\ref{eq215}), it is clear that  $|\theta|<1$  if and only if 
		\begin{equation*}
			(1-(1-s)\Re(\mu))^2+(1-s)^2\Im(\mu)^2<(1+s\Re(\mu))^2+s^2\Im(\mu)^2.
		\end{equation*}
		This implies $|\theta|<1$  if and only if  $(2s-1)|\mu|^2+2\Re(\mu)>0.$  Since, $\vartheta(\mathcal{T})=\vartheta{\widetilde(\mathcal{T}}),$ the proof is conclusive. $\blacksquare$
		
		Using the condition in Theorem \ref{th21}, next, we present two sufficient conditions that ensure the convergence of the {\textit{PESS}} iterative process. The first one is presented as follows.
		\begin{corollary}\label{coro1}  Let  $A\in\R^{n\times n}$ be a {\it SPD} matrix and let $B\in \R^{m\times n}$ and $C\in \R^{p\times m}$ be full row rank matrices and if $s\geq \frac{1}{2},$ then $\vartheta(\mathcal{T})<1,$ i.e., the \textit{PESS} iterative process always converges to the unique solution of the three-by-three block SPP \eqref{eq11}.
			
		\end{corollary}
		\proof 
		By Lemma \ref{lm22}, $\Sigma^{-\frac{1}{2}}\A\Sigma^{-\frac{1}{2}}$ is positive stable, implies that $\Re(\mu)>0$ and $|\mu|>0.$ Thus, the inequality (\ref{eq211}) holds if $s\geq \frac{1}{2}$. This completes the proof. $\blacksquare$
		
		\begin{remark}
			By applying Corollary \ref{coro1}, the unconditional convergence of the existing preconditioners, namely $\P_{SS}$, $\P_{GSS}$ and $\P_{EGSS}$ can be obtained from  $\P_{PESS}$ by sustituting  $s=1/2$.
		\end{remark}
		Next, we present a stronger sufficient condition to Corollary \ref{coro1} for the convergence of the \textit{PESS} iterative process.
		\begin{corollary}\label{coro2}
			Let  $A\in\R^{n\times n}$ be a {\it SPD} matrix and let $B\in \R^{m\times n}$ and $C\in \R^{p\times m}$ be full row rank matrices, if $$s>\max\left\{\frac{1}{2}\left(1-\frac{\lambda_{\min}(\Sigma^{-\frac{1}{2}}(\A+\A^T)\Sigma^{-\frac{1}{2}})}{\vartheta(\Sigma^{-\frac{1}{2}}\A\Sigma^{-\frac{1}{2}})^2}\right),0\right\},$$ then the \textit{PESS} iterative process is convergent for any initial guess.
		\end{corollary}
		\proof 
		Note that, as  $\Re(\mu)>0,$  the condition in (\ref{eq211}) holds if and only if $\frac{1}{2}-\frac{\Re(\mu)}{|\mu|^2}<s.$  Let ${\bf p}$ be an eigenvector corresponding to $\mu,$ then  we  have 
		\[\Re(\mu)=\frac{{\bf p}^*(\Sigma^{-\frac{1}{2}}(\A+\A^T)\Sigma^{-\frac{1}{2}}){\bf p}}{2 {\bf p}^*{\bf p}}\geq \frac{1}{2}\lambda_{\min}(\Sigma^{-\frac{1}{2}}(\A+\A^T)\Sigma^{-\frac{1}{2}}).
		\]  
   {Since $|\mu|\leq\vartheta(\Sigma^{-\frac{1}{2}}\A\Sigma^{-\frac{1}{2}}),$  we obtain the following:}
		\begin{equation}\label{EQ38}
			\frac{1}{2}-\frac{\lambda_{\min}(\Sigma^{-\frac{1}{2}}(\A+\A^T)\Sigma^{-\frac{1}{2}})}{2\vartheta(\Sigma^{-\frac{1}{2}}\A\Sigma^{-\frac{1}{2}})^2}\geq\frac{1}{2}-\frac{\Re(\mu)}{|\mu|^2}.
		\end{equation}
		If we apply  $s>0$, then we get the desired result from
		\eqref{EQ38}. $\blacksquare$
		
		The sufficient condition on Corollary \ref{coro2} is difficult to find for large $\A$ due to the involvement of computation of $\lambda_{\min}(\Sigma^{-\frac{1}{2}}(\A+\A^T)\Sigma^{-\frac{1}{2}})$ and $\vartheta(\Sigma^{-\frac{1}{2}}\A\Sigma^{-\frac{1}{2}})^2.$  Thus, we consider $s\geq \frac{1}{2}$ for practical implementation.
		
  \section{Spectral distribution of the {\textit{PESS}} preconditioned matrix}\label{sec4}
		Solving the three-by-three block \textit{SPP} \eqref{eq11} is equivalent to solving the preconditioned system of linear equations 
		\begin{align}
			\P_{PESS}^{-1}\A {\bf u}= \P_{PESS}^{-1}\d,
		\end{align}
		where $ \P_{PESS}$ serves as a preconditioner for the preconditioned {\it GMRES} ({\it PGMRES}) process. Since preconditioned matrices with clustered spectrum frequently culminate in rapid convergence for {\it GMRES}, see, for example  \cite{MBenzi2005, YSAAD}, the spectral distribution of the preconditioned matrix $\P_{PESS}^{-1}\A$ requires careful attention. As an immediate consequence of Corollary \ref{coro1}, the following clustering property for the eigenvalues of $\P_{PESS}^{-1}\A$ can be established.
		\begin{theorem}\label{th41}
			Let $A\in \R^{n\times n}$ be a {\it SPD} matrix and let  $B\in \R^{m\times n}$ and $C\in \R^{p\times m}$ are full row rank matrices. Suppose that $\lambda$ is an eigenvalue of the \textit{PESS} preconditioned matrix $\P_{PESS}^{-1}\A$. Then, for $s\geq \frac{1}{2},$ we have
			\begin{equation*}
				|\lambda-1|<1,
			\end{equation*}
			i.e.,  the spectrum of $\P_{PESS}^{-1}\A$ entirely contained in a disk  with centered at $(1,0)$ and radius strictly smaller than $1.$
		\end{theorem}
		\proof 
		Assume that $\lambda$ and $\theta$ are the eigenvalues of the \textit{PESS} preconditioned matrix  $\P_{PESS}^{-1}\A$ and the iteration matrix $\mathcal{T}$, respectively. On the other hand, from the matrix splitting \eqref{eq21}, we have \begin{align*}
			\P_{PESS}^{-1}\A= \P_{PESS}^{-1}(\P_{PESS}- \mathcal{Q}_{PESS})= I- \mathcal{T}.
		\end{align*}
		Then, we get $\lambda= 1- \theta.$ Now,  {from Corollary \ref{coro1}, it holds that $|\theta|< 1$  for $s\geq \frac{1}{2}.$}  Hence, it follows that,  { $|\lambda- 1|<1$  for $s\geq \frac{1}{2},$}
		which completes the proof. $\blacksquare$
		\vspace{0.1 cm}
		
		Let  $(\lambda,{\bf p}=[u^T, v^T, w^T]^T)$ be an eigenpair of the preconditioned matrix $\P_{PESS}^{-1}\A.$ Then, we have $\A{\bf p}=\lambda \P_{PESS}{\bf p},$ which can be written as
		\begin{subnumcases}{}\label{eq42a}
			Au+ B^Tv=\lambda (\Lambda_1+ sA)u+ s\lambda B^Tv, \\ \label{eq42b}
			-Bu- C^Tw= -s\lambda Bu+ \lambda \Lambda_2v- s\lambda C^Tw, \\ \label{eq42c}
			Cv=s\lambda Cv+ \lambda \Lambda_3w.
		\end{subnumcases}
		\begin{remark}\label{rm41}
			Notice that from \eqref{eq42a}-\eqref{eq42c}, we get $\lambda\neq 1/s,$ otherwise , ${\bf p}=[u^T, v^T, w^T]^T={0}\in \R^{n+m+p},$ which is  impossible as ${\bf p}$ is an eigenvector.
		\end{remark}  
		\begin{proposition}\label{prop1}
			Let $A\in \R^{n\times n}$ be a {\it SPD} matrix and let  $B\in \R^{m\times n}$ and $C\in \R^{p\times m}$ be full row rank matrices and $s>0.$ Let $(\lambda,{\bf p}=[u^T, v^T, w^T]^T)$ be an eigenpair of the preconditioned matrix $\P_{PESS}^{-1}\A.$ Then, the following holds:
			\begin{enumerate}
				\item[$(1)$] $u\neq 0.$ \label{pr1}
				\item[$(2)$] When $\lambda$ is real, $\lambda>0.$ \label{pr2}
				\item[$(3)$] $\Re(\mu)>0,$ where $\mu=\lambda/(1-s\lambda).$ \label{pr3}
			\end{enumerate}
		\end{proposition}
		\proof 
	$(1)$	
  Let $u=0$. Then by  \eqref{eq42a},  we get  $(1-s\lambda)B^Tv=0,$ which shows that $v=0,$ as $B$ has full row rank and $\lambda\neq 1/s.$ Furthermore, when combined with  \eqref{eq42c} leads to $w=0,$ as $\Lambda_3$ is {\it SPD}. Combining the above facts, it follows that ${\bf p}=[u^T, v^T, w^T]^T=0,$ which is impossible as ${\bf p}$ is an eigenvector. Hence, $u\neq 0.$
			
\noindent  $ (2)$ Let ${\bf p}=[u^T, v^T, w^T]^T$ be an eigenvector corresponding to the real eigenvalue $\lambda.$ Then, we have $\P_{PESS}^{-1}\A{\bf p}=\lambda {\bf p},$ or $\A{\bf p}=\lambda \P_{PESS}\,{\bf p}.$ Consequently, $\lambda$ can be written as 
			\begin{eqnarray*}
				\lambda &=& \frac{{\bf p}^T  \A{\bf p}+{\bf p}^T  \A^T\bf p}{2({\bf p}^T  \P_{PESS}{\bf p}+{\bf p}^T  \P_{PESS}^T\bf p)}\\
				&=& \frac{u^TAu}{2(su^TAu+u^T\Lambda_1u+v^T\Lambda_2v+w^T\Lambda_3w)}
				>0.
			\end{eqnarray*}
			The last inequality follows from the assumptions that $A,\Lambda_1,\Lambda_2$ and $\Lambda_3$ are {\it SPD} matrices. 
			
$(3)$   The system of linear equations  \eqref{eq42a}-\eqref{eq42c} can be reformulated as
			\begin{equation}\label{eq43}
				\left\{ \begin{array}{lcl}
					Au+B^Tv= (\lambda/(1-s\lambda)) \Lambda_1u, &  \\
					-Bu-C^Tw=(\lambda/(1-s\lambda)) \Lambda_2v,& \\
					Cv= (\lambda/(1-s\lambda)) \Lambda_3w.
				\end{array}\right.
			\end{equation}
    {The system of linear equations in \eqref{eq43} can also be expressed as}
			 $\A{\bf p}=\mu\Sigma{\bf p},$ or    $\Sigma^{-1}\A{\bf p}=\mu{\bf p},$ where $\mu=\lambda/(1-s\lambda).$ 
			By Lemma \eqref{lm22}, we have $\Sigma^{-1} \A$ is positive stable, which implies $\Re(\mu)>0,$ hence the proof is completed. $\blacksquare$
\begin{remark}\label{re42}
			From Proposition \ref{prop1}, it follows that when $\lambda$ is real, then   $\mu>0.$ Furthermore, the values of $\lambda$ lies in $(0, 1/s)$  for all $s>0$. 
		\end{remark}
  In the following theorem, we provide sharper bounds for real eigenvalues of the \textit{PESS} preconditioned matrix. Before that, we introduce the following notation:
  \begin{align}\label{eq:noation1}
    &  \xi_{\max}:=\lambda_{\max}(\Lambda_1^{-1}A), ~ \xi_{\min}:=\lambda_{\min}(\Lambda_1^{-1}A),~ \eta_{\max}:=\lambda_{\max}(\Lambda_2^{-1}B\Lambda_1^{-1}B^T),\\ \label{eq:noation2}
      &~\eta_{\min}:=\lambda_{\min}(\Lambda_2^{-1}B\Lambda_1^{-1}B^T) ~\text{and}~\theta_{\max}:=\lambda_{\max}(\Lambda_2^{-1}C^T\Lambda_3^{-1}C).
  \end{align}
		\begin{theorem}\label{th42}
			Let $A\in \R^{n\times n}$ be a {\it SPD} matrix and let  $B\in \R^{m\times n}$ and $C\in \R^{p\times m}$ be full row rank matrices and $s> 0.$ Suppose $\lambda$ is a real eigenvalue of the preconditioned matrix $\P_{PESS}^{-1}\A.$ Then
		\begin{equation}\label{eq46}
				\lambda \in \left(0, \frac{\xi_{\max}}{1+s\xi_{\max}}\right].
			\end{equation}
		\end{theorem}
		\begin{proof}
			Let ${\bf p}=[u^T, v^T, w^T]^T$ be an eigenvector corresponding to the real eigenvalue $\lambda.$  Then, the system of linear equations in \eqref{eq42a}-\eqref{eq42c} are satisfied. First, we assume that $v=0.$ Subsequently,  \eqref{eq42a} and \eqref{eq42c} are simplified to
			\begin{align}\label{eq47}
				Au=\lambda(\Lambda_1+sA)u~~  {\text{and}}~~ \lambda \Lambda_3w=0,
			\end{align}
			respectively. The second equation  {in} \eqref{eq47} gives  $w=0,$ as $\Lambda_3$ is {\it SPD}.  {Premultiplying by $u^T$ to the first equation of \eqref{eq47}} , we get 
			\begin{equation*}
				\lambda =\frac{u^TAu}{u^T\Lambda_1 u+su^TAu}=\frac{p}{1+sp},
			\end{equation*}
			where $p= \dm{\frac{u^TAu}{u^T\Lambda_1 u}}>0.$
			Now, for any nonzero vector $u\in \R^n,$ we  have 
		\begin{equation}\label{eq410}
				\lambda_{\min}(\Lambda_1^{-\frac{1}{2}}A\Lambda_1^{-\frac{1}{2}})\leq \frac{u^TAu}{u^T\Lambda_1 u}=\frac{(\Lambda_1^{\frac{1}{2}}u)^T\Lambda_1^{-\frac{1}{2}}A\Lambda_1^{-\frac{1}{2}}\Lambda_1^{\frac{1}{2}}u}{(\Lambda_1^{\frac{1}{2}}u)^T\Lambda_1^{\frac{1}{2}} u}\leq \lambda_{\max}(\Lambda_1^{-\frac{1}{2}}A\Lambda_1^{-\frac{1}{2}}).
			\end{equation}
			Since $\Lambda_1^{-\frac{1}{2}}A\Lambda_1^{-\frac{1}{2}}$ is similar to $\Lambda_1^{-1}A$ and the function $f:\R\rightarrow \R$ defined by $f(p):= p/(1+s p),$ where $s>0,$ is monotonically increasing in $ p>0$, we obtain the following bound for $\lambda:$
			\begin{equation}\label{eq411}
				\frac{\xi_{\min}}{1+s\xi_{\min}}\leq   \lambda\leq \frac{\xi_{\max}}{1+s\xi_{\max}}.
			\end{equation}
			Next, consider $v\neq 0$ but $Cv=0.$ Then $\eqref{eq42c}$ implies $w=0$ and from $\eqref{eq42b},$ we get
			\begin{align}\label{eq412}
				v=-\frac{1}{\mu}\Lambda_2^{-1}Bu.
			\end{align}
			Substituting \eqref{eq412} in \eqref{eq42a} yields
		\begin{equation}\label{eq413}
	Au=\lambda(\Lambda_1+sA)u+\frac{(1-s\lambda)}{\mu}B^T\Lambda_2^{-1}Bu.
			\end{equation}
			Premultiplying both sides of \eqref{eq413} by $u^T,$ we obtain
			\begin{equation}\label{eq414}
				\lambda u^TAu=\lambda^2u^T\Lambda_1u+s\lambda^2 u^TAu+(1-s\lambda)^2u^TB^T\Lambda_2^{-1}Bu,
			\end{equation}
			which is  {further} equivalent to
		\begin{equation}\label{eq415}
				\lambda^2-\lambda \frac{2sq+p}{s^2q+ sp+ 1}+ \frac{q}{s^2q+ sp+ 1} =0,
			\end{equation}
			where $q= \dm{\frac{u^TB^T\Lambda_2^{-1}Bu}{u^T\Lambda_1u}}.$
			By solving \eqref{eq415}, we obtain the following real solutions for  $\lambda$: 
			\begin{equation}\label{eq416}
				\lambda^{\pm}= \frac{2sq+ p\pm \sqrt{p^2- 4q}}{2(s^2q+ sp+ 1)},
			\end{equation}
			where $p^2-4q\geq 0.$ Consider  the functions $\Phi_1,\Phi_2:\R\times \R\times \R\longrightarrow \R$ defined by 
			\begin{eqnarray*}
				\Phi_1(p, q, s) &=& \frac{2sq+ p+  \sqrt{p^2- 4q}}{2(s^2q+sp+1)},\\
				\Phi_2(p, q, s) &=& \frac{2sq+ p- \sqrt{p^2- 4q}}{2(s^2q+ sp+ 1)},
			\end{eqnarray*}
			respectively, where $p,s> 0,$ $q\geq 0$ and $p^2 > 4q$. Then $\Phi_1$ and $\Phi_2$  {are} strictly monotonically increasing in the argument $p>0$ and $q\geq 0,$   
			and decreasing in the argument $q\geq0$ and $p>0,$  respectively.  Now, using \eqref{eq410}, \eqref{eq416} and monotonicity of the functions $\Phi_1$ and $\Phi_2,$ we get the following bounds: 
			\begin{align}\label{eq419}
				&\lambda^+=\Phi_1(p,q,s)<\Phi_1(\lambda_{\max}(\Lambda^{-1}_1 A),0,s),\\\label{eq420}
				&\Phi_2(\lambda_{\max}(\Lambda^{-1}_1 A),0,s)< \lambda^{-}=\Phi_2(p,q,s).
			\end{align}
			Combining \eqref{eq419} and \eqref{eq420}, we get the following bounds for $\lambda$:
			\begin{align}\label{eq421}
				0< \lambda^{\pm} < \frac{\xi_{\max}}{1+s\xi_{\max}}.
			\end{align}
			Next, consider the case $Cv\neq 0.$ Then from \eqref{eq42c}, we get $w=\frac{1 }{\mu}\Lambda_3^{-1}Cv.$ Substituting the value of $w$ in \eqref{eq42b}, we obtain
			\begin{equation}\label{eq422}
				(1-s\lambda)Bu+\frac{(1-s\lambda)}{\mu}C^T\Lambda_3^{-1} C v=-\lambda \Lambda_2v.
			\end{equation}
			Now, we assume that 
			\begin{equation}\label{Eq423}
				\lambda > \frac{\xi_{\max}}{1+ s\xi_{\max}}.
			\end{equation}
			  {By} Remark \ref{re42}, we get $1-s\lambda >0$  and  the above  assertion yields  $ \mu > \xi_{\max}.$
			Hence, the matrix $\mu\Lambda_1-A$ is nonsingular and \eqref{eq42a} gives 
			\begin{align}\label{eq424}
				u=\left(\mu\Lambda_1-A\right)^{-1}B^Tv.
			\end{align}
			Substituting \eqref{eq424} into \eqref{eq422}, we obtain
			\begin{equation}\label{eq425}
				(1-s\lambda)B\left(\mu\Lambda_1-A\right)^{-1}B^Tv+\frac{(1-s\lambda )}{\mu}C^T\Lambda_3^{-1} C v=-\lambda \Lambda_2v.
			\end{equation}
			Note that, $v\neq 0.$   {Thus}, premultiplying by $v^T$ on the both sides of \eqref{eq425} leads to the following identity:
			\begin{equation}\label{eq426}
				(1- s\lambda)v^TB\left(\mu\Lambda_1-A\right)^{-1}B^Tv+ \lambda v^T \Lambda_2v= -\frac{(1-s\lambda )}{\mu}v^TC^T\Lambda_3^{-1} C v.
			\end{equation}
			On the other hand,  $\mu\Lambda_1- A>\Lambda_1(\mu-\ \lambda_{\max}(\Lambda^{-1}A))I> 0$ and hence  {is a} {\it SPD} matrix. Consequently, the matrix  $B\left(\mu\Lambda_1-A\right)^{-1}B^T$ is also a {\it SPD} matrix and this implies  $v^TB\left(\mu\Lambda_1- A\right)^{-1}B^Tv> 0$ for all $v\not=0.$ This leads to a contradiction to \eqref{eq426} as $v^TC^T\Lambda_3^{-1} C v\geq 0.$ Hence, the assumption  \eqref{Eq423} is not true and we get
			\begin{equation}\label{eq428}
				\lambda \leq \frac{\xi_{\max}}{1+s\xi_{\max}}.
			\end{equation}
			 {Again, we have $\lambda>0$ from $(1)$ of Proposition \ref{prop1}.}   Combining the above together with \eqref{eq411} and \eqref{eq421}, we obtain the desired bounds in \eqref{eq46} for $\lambda.$ Hence, the proof is concluded.
		\end{proof}
		
		Since the preconditioners $\P_{SS}$ and  $\P_{EGSS}$  are special cases of the \textit{PESS} preconditioner, next, we obtain refined bounds for real eigenvalues of the {\it SS} and {\it EGSS}  preconditioned matrices from Theorem \ref{th42}.
		\begin{corollary}
			Let  $\P_{SS}$ be defined as in \eqref{eq14} with $\alpha>0$ and let  $\lambda$ be an real eigenvalue of the { \it SS} preconditioned matrix $\P_{SS}^{-1}\A.$ Then
			\begin{equation}\label{Eq429}
				\lambda\in \left(0, \frac{2 \kappa_{\max}}{ \alpha+\kappa_{\max}}\right],
			\end{equation}
   where $\kappa_{\max}=\lambda_{\max}(A).$ 
		\end{corollary}
		\proof 
		Since $\P_{SS}$ is a special case of the \textit{PESS} preconditioner $\P_{PESS}$ for $s=1/2$ as discussed in Table \ref{tab1}, the bounds in \eqref{Eq429} are obtained by substituting $\Lambda_1=\frac{1}{2}\alpha I$ and $s=1/2$ in Theorem \ref{th42}. $\blacksquare$
		
		In the next result, we discuss bounds for the real eigenvalues of the {\it EGSS} preconditioned matrix $\P_{EGSS}^{-1}\A$.
		\begin{corollary}
			Let  $\P_{EGSS}$ be defined as in \eqref{eq15} with $\alpha>0$ and let  $\lambda$ be an real eigenvalue of the {\it EGSS} preconditioned matrix $\P_{EGSS}^{-1}\A.$ Then
			\begin{equation}\label{Eq430}
				\lambda\in \left(0, \frac{2  \tilde{\kappa}_{\max}}{ \alpha+ \tilde{\kappa}_{\max}}\right],
			\end{equation}
where  $ \tilde{\kappa}_{\max}=\lambda_{\max}(P^{-1}A).$
		\end{corollary}
		\proof 
		Since $\P_{EGSS}$ is a special case of the \textit{PESS} preconditioner $\P_{PESS}$ for $s= 1/2$ as discussed in Table \ref{tab1} and the bounds   in \eqref{Eq430} are obtained by substituting $\Lambda_1=  \frac{1}{2}\alpha P$ and $s= 1/2$ in Theorem \ref{th42}. $\blacksquare$
		
		The following result shows the bounds when $\lambda$ is a non-real eigenvalue.
		
		\begin{theorem}\label{th43}
			Let $A\in \R^{n\times n}$ be a {\it SPD} matrix and let  $B\in \R^{m\times n}$ and $C\in \R^{p\times m}$ be full row rank matrices. Let $s> 0$ and   $\lambda$ be a non-real eigenvalue of the preconditioned matrix $\P_{PESS}^{-1}\A$ with $\Im(\lambda)\neq 0.$ Suppose  ${\bf p}=[u^T, v^T, w^T]^T$ is the  eigenvector corresponding to $\lambda.$  Then the following holds: 
			 \begin{enumerate}
				\item[$(1)$]   If $Cv=0,$   then  $\lambda$ satisfies
				\begin{eqnarray*}
					\frac{\xi_{\min}}{2+s\xi_{\min}} \leq \,|\lambda|\,\leq \sqrt{\frac{\eta_{\max}}{1+s\xi_{\min}+s^2\eta_{\max}}}.
				\end{eqnarray*}  \label{th431}  \vspace{-4mm}  
				\item[$(2)$]  If $Cv\neq 0,$  then real and imaginary part of $\lambda/(1-s\lambda)$ satisfies
    \begin{eqnarray*}
        \frac{\xi_{\min} \, \eta_{\min}}{2\left(\xi_{\max}^2+\eta_{\max}+\theta_{\max}\right)}\leq \Re\left(\frac{\lambda}{1-s\lambda}\right)\leq \frac{\xi_{\max}}{2}~\text{and}~	\left|\Im\left(\frac{\lambda}{1-s\lambda}\right)\right|\leq \sqrt{\eta_{\max}+\theta_{\max}},
	 \end{eqnarray*}			
    \end{enumerate}
    where $\xi_{\max},\xi_{\min},\eta_{\max},\eta_{\min}$ and $\theta_{\max}$ are defined as in \eqref{eq:noation1} and \eqref{eq:noation2}.
		\end{theorem}
		\proof  $(1)$\,\,  Let $\lambda$ be an eigenvalue of $\P_{PESS}^{-1}\A$ with $\Im(\lambda)\neq 0$ and the corresponding eigenvector is ${\bf p}=[u^T, v^T, w^T]^T.$ Then, the system of linear equations in \eqref{eq42a}-\eqref{eq42c} holds. We assert that, $v\neq 0,$ otherwise from \eqref{eq42a}, we get $Au= \lambda(\Lambda_1+ sA)u,$ which further implies $(\Lambda_1+ sA)^{-\frac{1}{2}}A(\Lambda_1+ sA)^{-\frac{1}{2}}\bar{u}= \lambda \bar{u},$ where $\bar{u}= (\Lambda_1+ sA)^{\frac{1}{2}}u.$   {Since  $(\Lambda_1+ sA)^{-\frac{1}{2}}A(\Lambda_1+ sA)^{-\frac{1}{2}}$ is a {\it SPD} matrix and by Lemma \ref{prop1}   $u\neq 0$  and thus  we have $\Bar{u}\neq 0.$ This implies $\lambda$ is real, leading to  a contradiction to $\Im(\lambda)\neq 0.$} 
		
		Initially, we consider the case  $Cv=0.$ The aforementioned discussion suggests that $Au\neq \lambda (\Lambda_1+sA)u$ for any $u\neq 0.$ Consequently,  $(\mu \Lambda_1 -A)$ is nonsingular.  {Again, since $Cv= 0,$ from \eqref{eq42c} and following a similar process as in \eqref{eq412}-\eqref{eq414}, we obtain  quadratic equation \eqref{eq415} in $\lambda$, which has complex solutions:}
		\begin{equation}\label{eq431}
			\lambda^{\pm i}= \frac{2sq+ p\pm i \sqrt{4q -p^2}}{2(s^2q+ sp+ 1)} 
		\end{equation}
		for $4q>p^2.$ From \eqref{eq431}, we get 
		\begin{equation}\label{eq432}
			|\lambda^{\pm i}|^2= \frac{q}{s^2q+ sp+ 1}. 
		\end{equation}
		Now, consider  the function $\Psi_1:\R\times \R\times \R\longrightarrow \R$ defined by 
		\begin{equation*}
			\Psi_1(p, q, s)=\frac{q}{s^2q+sp+1},
		\end{equation*}
		where $p, q, s>0.$  Then $\Psi_1$ is monotonically increasing in the argument $q>0$ while decreasing in the argument $p>0.$ Due to the fact that $q\leq \lambda_{\max}(\Lambda_1^{-1}B^T\Lambda_2^{-1}B)= \eta_{\max},$ we get the following bounds: 
\begin{eqnarray}\label{eq436}
	|\lambda^{\pm i}|^2 = \Psi_1(p, q, s) &\leq& \Psi_1(\xi_{\min}, \eta_{\max},s)=\frac{\eta_{\max}}{1+s\xi_{\min}+s^2\eta_{\max}}.
\end{eqnarray}
Again, considering the condition $q>(p^2/4),$ \eqref{eq432} adheres to the inequality  $(p/(2+ sp))\leq |\lambda^{\pm i}|.$ Notably, the function $\Psi_2:\R\longrightarrow \R$ defined by $\Psi_2(p)=p/(2+ sp),$ where $p,s>0,$ is monotonically increasing in the argument $p> 0,$ we obtain the following bound:
\begin{align}\label{eq437}
|\lambda^{\pm i}|=\Psi_2(p)\geq \Psi_2(\xi_{\min})= \frac{\xi_{\min}}{2+s\xi_{\min}}.
		\end{align}
		Hence, combining \eqref{eq436} and \eqref{eq437}, proof for the part $(1)$ follows.
		\vone 
		\noin $(2)$\,\,     Next, consider the case $Cv\neq 0,$ then  \eqref{eq42c} yields
		$w=\frac{1}{\mu}\Lambda_3^{-1}Cv,$ where $\mu=\lambda/(1-s\lambda).$ Substituting this in \eqref{eq42b}, we get 
		\begin{equation}\label{Eq438}
			Bu+\frac{1}{\mu} C^T\Lambda_3^{-1}Cv=-\mu \Lambda_2v.
		\end{equation}
		On the other side, nonsingularity of $(\mu \Lambda_1-A)$ and \eqref{eq42a} enables us the following identity:
		\begin{equation}\label{eq438}
			u=(\mu \Lambda_1-A)^{-1}B^Tv.
		\end{equation}
		Putting \eqref{eq438} on \eqref{Eq438}, we get  
		\begin{equation}\label{eq439}
			B(\mu \Lambda_1-A)^{-1}B^Tv+\frac{1}{\mu} C^T\Lambda_3^{-1}Cv=-\mu \Lambda_2 v.
		\end{equation}
		As $v\neq 0,$  premultiplying by $v^*$ on the both sides of \eqref{eq439} yields 
		\begin{equation}\label{eq440}
			v^*B(\mu \Lambda_1-A)^{-1}B^Tv+\frac{1}{\mu} v^*C^T\Lambda_3^{-1}Cv=-\mu v^{*}\Lambda_2 v.
		\end{equation}
		Consider the following eigenvalue decomposition of $ \Lambda_1^{-\frac{1}{2}}A \Lambda_1^{-\frac{1}{2}}$:
		\begin{equation*}
			\Lambda_1^{-\frac{1}{2}}A \Lambda_1^{-\frac{1}{2}}=V\mathcal{D}V^T,
		\end{equation*}
		where $\mathcal{D}=\diag(\theta_i)\in \R^{n\times n}$ and $\theta_i>0,\, i=1,2,\ldots,n,$ are in $\sigma ( \Lambda_1^{-\frac{1}{2}}A \Lambda_1^{-\frac{1}{2}})$ and $V\in \R^{n\times n}$ is an orthonormal matrix. Then, \eqref{eq440} can be written  as
		\begin{equation}\label{eq442}
			v^*B\Lambda_1^{-\frac{1}{2}}V(\mu I-\mathcal{D})^{-1}V^T\Lambda_1^{-\frac{1}{2}}B^Tv+\frac{1}{\mu} v^*C^T\Lambda_3^{-1}Cv= -\mu v^*\Lambda_2 v.
		\end{equation}
		Since $(\mu I-\mathcal{D})^{-1}=\Theta_1- {\bf 
			\iu}\Im(\mu)\Theta_2,$ where ${\bf 
			\iu}$ denotes the imaginary unit and
		
		$$\Theta_1= \dm{\diag\left(\frac{\Re(\mu)- \theta_i}{(\Re(\mu)- \theta_i)^2+ \Im(\mu)^2}\right), ~ \Theta_2= \diag\left(\frac{1}{(\Re(\mu)- \theta_i)^2+ \Im(\mu)^2}\right)},$$ 
		from  \eqref{eq442} and the fact $\Im(\mu)\neq 0,$ we get
		\begin{eqnarray}\label{eq443}
			& v^*B\Lambda_1^{-\frac{1}{2}}V\Theta_1V^T\Lambda_1^{-\frac{1}{2}}B^Tv+\frac{\Re(\mu)}{|\mu|^2} v^*C^T\Lambda_3^{-1}Cv=-\Re(\mu) v^*\Lambda_2 v,\\ \label{eq444}
			&v^*B\Lambda_1^{-\frac{1}{2}}V\Theta_2V^T\Lambda_1^{-\frac{1}{2}}B^Tv+\frac{1}{|\mu|^2} v^*C^T\Lambda_3^{-1}Cv=v^*\Lambda_2 v.
		\end{eqnarray}
		 Observed that, $\Theta_2\leq \frac{1}{\Im(\mu)^2}I,$ then from\eqref{eq444}, we obtain 
		\begin{equation*}
			1\leq \frac{1}{\Im(\mu)^2}\left(\frac{v^*B\Lambda_1^{-1}B^Tv}{v^*\Lambda_2v}+\frac{v^*C^T\Lambda_3^{-1}Cv}{v^*\Lambda_2v}\right).
		\end{equation*}
		Therefore, we get 
		\begin{equation}\label{eq446}
			|\Im(\mu)|\leq    \sqrt{\eta_{\max}+\theta_{\max}}.
		\end{equation}
		Furthermore, notice that  $\Theta_1\geq \left(\Re(\mu)-\xi_{\max}\right)\Theta_2,$ then from \eqref{eq443}, we deduce
		\begin{equation}\label{EQ448}
			-\Re(\mu)\geq \left(\Re(\mu)-\xi_{\max}\right)\frac{v^*B\Lambda_1^{-\frac{1}{2}}V\Theta_2V^T\Lambda_1^{-\frac{1}{2}}B^Tv}{v^*\Lambda_2 v}+\frac{\Re(\mu)}{|\mu|^2}\frac{v^*C^T\Lambda_3^{-1}Cv}{v^*\Lambda_2 v}.
		\end{equation}
		Using \eqref{eq444} to \eqref{EQ448}, we get 
		\begin{equation}\label{eq448}
			2\Re(\mu)\leq \xi_{\max}
			-\frac{\xi_{\max}}{|\mu|^2}\frac{v^*C^T\Lambda_3^{-1}Cv}{v^*\Lambda_2 v}\leq \xi_{\max}.
		\end{equation}
		Hence, \eqref{eq448} yields the following bound:
		\begin{equation}\label{Eq449}
			\Re(\mu)\leq \frac{\xi_{\max}}{2}.
		\end{equation}
		On the other side,  from \eqref{eq443} and \eqref{eq444}, we deduce that
		\begin{equation}\label{eq449}
			2\Re(\mu)=\frac{v^*B\Lambda_1^{-\frac{1}{2}}V(\Re(\mu)\Theta_2-\Theta_1)V^T\Lambda_1^{-\frac{1}{2}}B^Tv}{v^*\Lambda_2 v}\geq \xi_{\min}\, \frac{v^*B\Lambda_1^{-\frac{1}{2}}V\mathcal{X}V^T\Lambda_1^{-\frac{1}{2}}B^Tv}{v^*\Lambda_2 v},
		\end{equation}
		where $\dm{\mathcal{X}=\diag \left(\frac{1}{(\Re(\mu)-\theta_i)^2+\Im(\mu)^2}\right).}$  Note that 
		
		$$\dm{\frac{1}{(\Re(\mu)-\theta_i)^2+\Im(\mu)^2}\geq \frac{1}{\displaystyle{\max_i}(\Re(\mu)-\theta_i)^2+\Im(\mu)^2}.}$$ Now, using Proposition \ref{prop1} and \eqref{Eq449}, we get 
		\[-\theta_i<\Re(\mu)-\theta_i\leq \frac{\xi_{\max}}{2}-\theta_i,
		\]
		which further yields $(\Re(\mu)-\theta_i)^2\leq \max\left\{\left(\frac{\xi_{\max}}{2}-\theta_i\right)^2,\theta_i^2\right\}.$ Hence, we get 
		\begin{eqnarray}\label{eq450}
			\nonumber \displaystyle{\max_{i}}(\Re(\mu)-\theta_i)^2&=&\displaystyle{\max_{i}}\{(\Re(\mu)-\xi_{\min})^2, (\Re(\mu)-\xi_{\max})^2\}\\ 
			&\leq & \displaystyle{\max_{i}} \left\{\left(\frac{\xi_{\max}}{2}-\xi_{\min}\right)^2,\xi^2_{\max}\right\}= \xi^2_{\max}.
		\end{eqnarray}
		Therefore, by \eqref{eq450} and from the bound in \eqref{eq446}, we obtain 
		\begin{equation}\label{eq453}
			\mathcal{X}\geq \frac{I}{\xi^2_{\max}+\eta_{\max}+ \theta_{\max}}.
		\end{equation}
		Combining \eqref{eq449} and \eqref{eq453} leads to the following bounds: 
		\begin{eqnarray}\label{eq455}
			   2\Re(\mu)&\geq&\frac{\xi_{\min}}{\xi^2_{\max}+ \eta_{\max}+ \theta_{\max}} \frac{v^*B\Lambda_1^{-1}B^Tv}{v^*\Lambda_2 v} \geq\frac{\xi_{\min}\eta_{\min}}{\xi^2_{\max}+\eta_{\max}+ \theta_{\max}}.
		\end{eqnarray}
		 Hence, the proof of part $(2)$ follows by merging the inequalities of \eqref{eq446} and \eqref{eq455}. $\blacksquare$
    \begin{remark}
       From the bounds in $(2)$ of Theorem \ref{th43}, we obtain  $${\frac{1}{|\tau|^2}\left(\frac{\xi_{\min}\eta_{\min}}{2(\xi_{\max}^2+\eta_{\max}+\theta_{\max})}+\frac{1}{s}\right)\leq \frac{1}{s}-\Re(\lambda)\leq \frac{1}{|\tau|^2}\left(\frac{\xi_{\max}}{2}+\frac{1}{s}\right)}$$
     and  
 $$|\Im(\lambda)|\leq \frac{1}{s\tau \bar{\tau}}\sqrt{\eta_{\max}+\theta_{\max}},$$ where $\tau=s\mu+1.$ Thus, the real and imaginary parts of the eigenvalues of the preconditioned matrix $\P_{PESS}^{-1}\A$ cluster better as $s$ grows. Therefore, the \textit{PESS} preconditioner can accelerate the rate of convergence of the Krylov subspace process, like \textit{ GMRES}.
  \end{remark}
		
		Utilizing the established bounds in Theorem \ref{th43}, we derive the subsequent estimations for non-real eigenvalues of  {\it SS} and {\it EGSS} preconditioned matrices.
		\begin{corollary}
			Let  $\P_{SS}$ be defined as in \eqref{eq14} with $\alpha>0$ and let  $\lambda$ be a non-real eigenvalue of the  {\it SS} preconditioned matrix $\P_{SS}^{-1}\A$ with $\Im(\lambda)\neq 0$ and ${\bf p}=[u^T, v^T, w^T]^T$ is the corresponding eigenvector. Then,  
		
		\vone
	\noin 			$(1) $\,\, If $Cv=0,$    $\lambda$ satisfies
				\begin{equation*}
					\frac{2\kappa_{\min}}{2\alpha+\kappa_{\min}}     \leq |\lambda|\leq \sqrt{\frac{4\tau_{\max}}{\alpha^2+\alpha\kappa_{\min}+\tau_{\max}}}.
				\end{equation*}  \label{cr431}    
\noin				$(2) $ \,\, If $Cv\neq 0,$ the real and imaginary part of $\lambda/ (2-\lambda)$ satisfies
				\begin{eqnarray*}
					&& \frac{\kappa_{\min} \, \tau_{\min}}{2\alpha\left(\kappa^2_{\max}+\tau_{\max}+ \beta_{\max}\right)}\leq \Re\left(\frac{ \lambda}{2-\lambda}\right)\leq \frac{\kappa_{\max}}{2\alpha}~ \mbox{and}~\left|\Im\left(\frac{\lambda}{2-\lambda}\right)\right|\leq \frac{\sqrt{\tau_{\max}+\beta_{\max}}}{\alpha},
				\end{eqnarray*} \label{cr432}         where $\kappa_{\min}=\lambda_{\min}(A),$ $\tau_{\max}=\lambda_{\max}(BB^T),$ $\tau_{\min}=\lambda_{\min}(BB^T)$ and $\beta_{\max}=\lambda_{\max}(C^TC).$
		\end{corollary}
		\proof 
		Since $\P_{SS}$ is a special case of $\P_{PESS}$ for $s= 1/2$ as discussed in Table \ref{tab1}, then desired  bounds will be obtained by setting $\Lambda_1=\Lambda_2= \Lambda_3= \frac{1}{2}\alpha I$  and $s=1/2$  into Theorem \ref{th43}. $\blacksquare$ 
		\begin{corollary}
			Let  $\P_{EGSS}$ be defined as in \eqref{eq15} and let $P,Q$ and $W$ are  {\it SPD} matrices with $\alpha ,\beta,\gamma>0.$ Let $\lambda$ be a non-real eigenvalue of the {\it EGSS} preconditioned matrix $\P_{EGSS}^{-1}\A$  with $\Im(\lambda)\neq 0$  and ${\bf p}=[u^T, v^T, w^T]^T$ is the corresponding eigenvector. Then
	\vone
		\noin 	$(1)$  If $Cv=0,$  $\lambda$ satisfies
				\begin{equation*}
					\frac{2\tilde{\kappa}_{\min}}{2\alpha+\tilde{\kappa}_{\min}}    \leq |\lambda|\leq \sqrt{\frac{4\tilde{\tau}_{\max}}{\alpha \beta+\beta\tilde{\kappa}_{\min}+\tilde{\tau}_{\max}}}.
				\end{equation*}  \label{cr441}    
			\noin 	$(2)$\,\, If $Cv\neq 0,$   the real and imaginary part of $\lambda/(2-\lambda)$ satisfies
				\begin{eqnarray*}
					&& \frac{\tilde{\kappa}_{\min} \, \tilde{\tau}_{\min}}{2\left(\beta\tilde{\kappa}^2_{\max}+\alpha\tilde{\tau}_{\max}+(\alpha^2/\gamma)\tilde{\beta}_{\max}\right)}\nonumber  \leq  \Re\left(\frac{\lambda}{2-\lambda}\right) \leq  \frac{\tilde{\kappa}_{\max}}{2\alpha}~ \mbox{and}\\
					&&    \left|\Im\left(\frac{\lambda}{2- \lambda}\right)\right| \leq \sqrt{\frac{1}{\alpha \beta}\tilde{\tau}_{\max}+ \frac{1}{\beta \gamma}\tilde{\beta}_{\max}},
				\end{eqnarray*} 
    where $\tilde{\kappa}_{\min}=\lambda_{\min}(P^{-1}A),$ $\tilde{\tau}_{\max}=\lambda_{\max}(Q^{-1}BP^{-1}B^T),$ $\tilde{\tau}_{\min}=\lambda_{\min}(Q^{-1}BP^{-1}B^T)$ and $\tilde{\beta}_{\max}=\lambda_{\max}(C^TC).$  		
		\end{corollary}
		\proof 
		Since $\P_{EGSS}$ is a special case of $\P_{PESS}$ for $s= 1/2$ as discussed in Table \ref{tab1}, then the desired  bounds will be obtained by setting $\Lambda_1= \frac{1}{2}\alpha P,$ $\Lambda_2= \frac{1}{2}\beta Q$, $\Lambda_3= \frac{1}{2}\gamma W$ and $s= 1/2$ in  Theorem \ref{th43}.  $\blacksquare$

  
  \section{Local \textit{PESS} (\textit{LPESS}) preconditioner}\label{sec:LPESS}

  To enhance the efficiency of the \textit{PESS} preconditioner, in this section,  we propose a relaxed version of the \textit{PESS} preconditioner by incorporating the concept of \textit{RSS} preconditioner \cite{CAOSS19}.  By omitting the term $\Lambda_1$ from the $(1,1)$-block of $\P_{PESS}$, we present the local \textit{PESS} (\textit{LPESS}) preconditioner, denoted as $\P_{LPESS}$, defined as follows:
    \begin{eqnarray}
        \P_{LPESS}:=\bmatrix{A & sB^T &0\\ -sB & \Lambda_2&-sC^T\\ 0& sC& \Lambda_3}.
    \end{eqnarray}
    The implementation of the \textit{LPESS} preconditioner is similar to the Algorithm \ref{alg1}. However, there is one modification in step $3,$ i.e., we need to solve a linear subsystem $(sA+ s^2B^T\widehat{X}^{-1}B)w_1=v$ instead of $\widehat{A}w_1=v.$

    To illustrate the efficiency of the \textit{LPESS} preconditioner, we study the spectral distribution of the preconditioned matrix $\P_{LPESS}^{-1}\A.$ To achieve this, we consider the following decomposition of the matrices $\A$ and $\P_{LPESS} $ as follows:
    \begin{eqnarray}\label{sec5:eq52}
\A=\mathfrak{L}\mathfrak{D}\mathfrak{U}~ \text{and}~ \P_{LPESS}=\mathfrak{L}\widetilde{\mathfrak{D}}\mathfrak{U},
    \end{eqnarray}
    where $$\mathfrak{L}=\bmatrix{I & 0& 0\\ -BA^{-1}& I&0\\ 0&0&I},\mathfrak{U}=\bmatrix{I & A^{-1}B^T& 0\\ 0& I&0\\ 0&0&I},\mathfrak{D}=\bmatrix{A&0&0\\0&Q&-C^T\\0&C&0}, \widetilde{\mathfrak{D}}=\bmatrix{sA&0&0\\0&\Lambda_2+sQ&-sC^T\\0&sC&\Lambda_3}$$ and $Q=BA^{-1}B^T.$
    Using the decomposition in \eqref{sec5:eq52}, we have 
    \begin{eqnarray}
\P^{-1}_{LPESS}\A=\mathfrak{U}^{-1}\bmatrix{s^{-1}I&0\\ 0&M^{-1}K}\mathfrak{U},
    \end{eqnarray}
    where $M=\bmatrix{\Lambda_2+sQ& -sC^T\\ sC&\Lambda_3}$ and $K=\bmatrix{Q& -C^T\\C& 0}.$ Since, $\P^{-1}_{LPESS}\A$ is similar to  $\bmatrix{s^{-1}I&0\\ 0&M^{-1}K},$ $\P_{LPESS}^{-1}\A$ has eigenvalue $1/s$ with multiplicity $n$ and the remaining eigenvalues satisfies the generalized eigenvalue problem $Kp=\lambda Mp.$
    \begin{theorem}\label{theorem:RPESS}
    Let $A\in \R^{n\times n}, \Lambda_2\in \R^{m\times m}$ and $\Lambda_3\in \R^{p\times p}$  be  {\it SPD} matrices and let  $B\in \R^{m\times n}$ and $C\in \R^{p\times m}$ be full row rank matrices. Assume that $s> 0,$  then \textit{LPESS} preconditioned matrix $\P_{LPESS}^{-1}\A$ has $n$ repeated eigenvalues equal to $1/s.$ The remaining $m+p$ eigenvalues satisfies the following:
    \begin{enumerate}
        \item[$(1)$] real eigenvalues located in the interval $$\Big[ \min\Big\{\frac{\vartheta_{\min}}{1+s\vartheta_{\min}}, \frac{\tilde{\theta}_{\min}}{\vartheta_{\max}+s\tilde{\theta}_{\min}}\Big\}, \frac{\vartheta_{\max}}{1+s\vartheta_{\max}}\Big].$$
        \item  [$(2)$] If $\lambda$ is any non-real eigenvalue (i.e., $\Im(\lambda)\neq 0$), then 
        \begin{eqnarray*}
            \frac{\vartheta_{\min}}{2+s\vartheta_{\min}}\leq |\lambda|\leq \sqrt{\frac{\tilde{\theta}_{\max}}{1+s\vartheta_{\min}+s^2\tilde{\theta}_{\max}}}~\text{and}~
          \frac{1}{s(1+s\sqrt{\tilde{\theta}_{\max}})}  \leq |\lambda-\frac{1}{s}|\leq \frac{2}{s(2+s\vartheta_{\min})},
        \end{eqnarray*}
        \end{enumerate}
        where $\vartheta_{\max}:=\lambda_{\max}(\Lambda_2^{-1}Q),$ $\vartheta_{\min}:=\lambda_{\min}(\Lambda_2^{-1}Q)$ $\tilde{\theta}_{\max}:=\lambda_{\max}(\Lambda_3^{-1}C\Lambda_1^{-2}C^T)$ and  $\tilde{\theta}_{\min}:=\lambda_{\min}(\Lambda_3^{-1}C\Lambda_1^{-2}C^T).$
    \end{theorem}
    \proof Observe that,  
    \begin{eqnarray}
M^{-1}K=\mathcal{F}^{-1}\widetilde{M}^{-1}\widetilde{K}\mathcal{F},
    \end{eqnarray}
  where $\mathcal{F}=\bmatrix{I &0\\0 &-I},$ $\widetilde{M}=\bmatrix{\Lambda_2+sQ & sC^T\\ -sC& \Lambda_3}$ and $\widetilde{K}=\bmatrix{Q& C^T\\ -C& 0}.$ Hence, $M^{-1}K$  is similar to $\widetilde{M}^{-1}\widetilde{K}.$  Now, $\widetilde{M}^{-1}\widetilde{K}$ is in the form of the preconditioned matrix discussed in \cite{NESS2019}.  Therefore, by applying the Corollary 4.1 and Theorem 4.3 of \cite{NESS2019}, we obtain the desired bounds in $(1)$ and $(2).$
    $\blacksquare$

\section{ {The strategy of parameter selection}}\label{Sec:parameter}
 
  {It is worth noting that the efficiency of the proposed \textit{PESS} preconditioner depends on the selection of involved  \textit{SPD} matrices $\Lambda_1, \Lambda_2, \Lambda_3$ and a positive real parameter $s.$ In this section, motivated by \cite{NESS2019, parameter2014, PSS2018}, we discuss a practical way to choose the parameter $s$ and the \textit{SPD} matrices for the effectiveness of the proposed preconditioners.
 }
 
  {It is well acknowledged that the optimal parameter of \textit{PESS} iteration process is obtained when $\vartheta(\mathcal{T})$ is minimized \cite{parameter2014}. For achieving this, similar to the approach in \cite{PSS2018}, we first define a function  $\varphi(s)=\|\mathcal{Q}_{PESS}\|^2_F$ depending on the parameter $s$. 
 Our aim is to minimize $\varphi(s).$ Then after some straightforward calculations, we obtain
 \begin{align*}
     \varphi(s)=\|\mathcal{Q}_{PESS}\|^2_F&=tr(\mathcal{Q}^T_{PESS}\mathcal{Q}_{PESS})\\
    &=\|\Lambda_1\|^2_F+\|\Lambda_2\|^2_F+\|\Lambda_3\|^2_F+(s-1)^2\|A\|^2_F+2(s-1)tr(\Lambda_1 A)\\
    &\quad+2(s-1)^2\|B\|_F^2)+2(s-1)^2\|C\|_F^2),
  \end{align*}
  where $tr(\cdot)$  and $\|\cdot\|_F$ denote the trace and the Frobenious norm of a matrix, respectively.
  Now we choose the parameter $s$ and \textit{SPD} matrices $\Lambda_1, \Lambda_2$ and $\Lambda_3$  to make $\varphi(s)$ as small  as possible. Since $\|A\|^2_F, \|B\|^2_F, \|C\|^2_F$ and $tr(\Lambda_1A)$ are positive,  we can select $s=1,$ then $ \varphi(s)=\|\Lambda_1\|^2_F+\|\Lambda_2\|^2_F+\|\Lambda_3\|^2_F.$   Thus if we choose $\|\Lambda_1\|_2, \|\Lambda_2\|_2, \|\Lambda_3\|_2\,\rightarrow 0,$ we have $\varphi(s)\rightarrow 0$ and consequently,  $\mathcal{Q}_{PESS}\rightarrow 0.$ } 

 {On the other hand, motivated by \cite{DK2020, YCAO2015}, we discuss another strategy for choosing the parameter $s$. Notice that in Algorithm \ref{alg1}, we need to solve two linear system with coefficient matrices $\widehat{X}=\Lambda_2+ s^2C^T\Lambda_3^{-1} C$ and $\widetilde{A}=\Lambda_1 +s A+s^2B^T \widehat{X}B.$ Similar to   \cite{DK2020, YCAO2015}, we choose  $s$ and $\|\Lambda_2\|_2$ as follows: 
\begin{eqnarray}\label{s_opt}
    s=\sqrt{\frac{\|\Lambda_2\|_2}{\|C^T\Lambda_3^{-1}C\|_2}} ~~\text{and}~~ \|\Lambda_2\|_2=\frac{\|B\|_2^{4}}{{4\|C^T\Lambda_3^{-1}C\|_2\|A\|_2^2}},
\end{eqnarray} 
which balance the matrices $\Lambda_2$ and $C^T\Lambda_3^{-1}C$ in $\widehat{X}$ and  the matrices $A$ and $C^T\Lambda_3^{-1}C$ in $\widetilde{A}$. In this case, we denote $s$ by  $s_{est}$ and  
$\|\Lambda_2\|_2$ by $\beta_{est}.$ Numerical results are presented in Section \ref{sec5} to demonstrate the effectiveness of $\P_{PESS}$ for the above choices of the parameters.}

  {In the sequel, we can rewrite the \textit{PESS} preconditioner as 
$\P_{PESS}=s\widetilde{\P}_{PESS},$
 where $$\widetilde{\P}_{PESS}=\bmatrix{\frac{1}{s}\Lambda_1+A& B^T &0\\ -B& \frac{1}{s}\Lambda_2&-C^T\\ 0&C&\frac{1}{s}\Lambda_3}.$$ Since the prefactor $s$ has not much effect on the performance of \textit{PESS} preconditioner, investigating the optimal parameters of $\P_{PESS}$ and $\widetilde{\P}_{PESS}$ are equivalent. A general criterion for a preconditioner to perform efficiently is that it should closely approximate the coefficient matrix $\A$ \cite{MBenzi2005}. Consequently, the difference $\widetilde{\P}_{PESS}-\A=\dm{\frac{1}{s}\Sigma}$  approaches zero matrix as $s$ tends to positive infinity for fixed $\Lambda_1, \Lambda_2$ and $\Lambda_3.$  Thus, the preconditioner \textit{PESS} shows enhance efficiency for large value of $s.$ However, $s$ can not be too large as  the coefficient matrix $\widetilde{A}$ of the linear subsystem in step $3$ of Algorithm \ref{alg1} becomes very-ill conditioned. Similar investigations also hold for \textit{LPESS} preconditioner. Nevertheless, in Figure \ref{IT_vs_s}, we show the adaptability of the \textit{PESS} and \textit{LPESS} preconditioners by varying the parameter $s.$}
 \section{Numerical  experiments}\label{sec5}
	In this section, we conduct a few numerical experiments to showcase the superiority and efficiency of the proposed \textit{PESS} and \textit{LPESS} preconditioners over the existing preconditioners to enhance the convergence speed of the  Krylov subspace iterative process to solve three-by-three block {\it SPPs.} Our study involves a comparative analysis among {\it GMRES} process and the {\it PGMRES} process employing the proposed   preconditioners $\P_{PESS}$ and $\P_{LPESS}$  and the existing  {baseline} preconditioners $\P_{BD}$,  $\P_{IBD}$,  {$\P_{MAPSS}$, $\P_{SL},$} $\P_{SS},\P_{RSS},$ $\P_{EGSS}$  {and $\P_{RPGSS}$}. The numerical results are reported from the aspect of iteration counts (abbreviated as ``IT'') and elapsed CPU times in seconds (abbreviated as ``CPU'').  Each subsystem involving $\widehat{X}$ and $\widetilde{A}$ featured in \textbf{Algorithm \ref{alg1}} are precisely solved by applying the Cholesky factorization of the coefficient matrices. For all iterative process, the initial guess vector is ${\bf u}_0=0\in \R^{n+m+p}$ and the termination criterion is 
	$${\tt RES}:=\frac{\|\A{\bf u}^{k+1}-{\bf d}\|_2}{\|{\bf d}\|_2}<10^{-6}.$$
	The vector $\d\in \R^{n+m+p}$ is chosen so that the exact solution of the system \eqref{eq11} is ${\bf u}_*=[1, 1 ,\ldots, 1]^T\in \R^{n+m+p}.$ All numerical tests are run in MATLAB (version R2023a) on a Windows 11 operating system 
	with Intel(R) Core(TM) i7-10700 CPU, 2.90GHz, 16 GB memory. 

 
	\begin{exam}\label{ex1}
		\begin{table}[]
					\centering
			\caption{Numerical results of {\it GMRES}, {\it BD}, {\it IBD},  {\textit{MAPSS}, \textit{SL},} {\it SS}, {\it RSS}, {\it EGSS},  {\it RPGSS,} \textit{PESS} and {\it LPESS PGMRES} processes for Example \ref{ex1}.}
			\label{tab2}
			\resizebox{11.5cm}{!}{
				\begin{tabular}{cccccccc}
					\toprule
					Process& $l$	&$16$& $32$ &$48$& $64$ &$80$&  {128}  \\
					\midrule 
					& size($\A$)	& $1024$ &$4096$& $9216$ & $16384$ & $25600$ &  {$65536$} \\
					\midrule
					\multirow{3}{*}{\it GMRES}& IT& $865$& $3094$& $6542$&$\bf{--}$&$\bf{--}$ &  {$\bf{--}$} \\
					& CPU&  $0.6607$ & $101.7659$& $989.0242$& $2116.1245$& $3755.9251$&  {$8694.5762$}\\
					&{\tt RES}&$8.2852e-07$&$9.9189e-07$&$9.8389e-07$&$1.0813e-03$&$1.9862e-03$ &  {$0.0037$}\\	
					\midrule
					\multirow{3}{*}{\it BD}& IT& $4$&$4$&$4$&$4$&$4$ &  {$4$}\\
					& CPU&$ \textbf{0.0611} $&$2.3135$& $17.0971$& $105.4514$ &$337.4322$&  {$1355.6953$} \\
					&{\tt RES}&$1.2728e-13$&$1.7577e-13$ &$6.6653e-13$ &$6.0211e-12$ &$2.6610e-12$ &  {$6.4572e-07$}\\
					\midrule
					\multirow{3}{*}{\it IBD}& IT& $22$& $22$& $21$& $21$& $21$&  {$27$}\\
					& CPU&$\textbf{0.0702}$ &$1.1711$ & $9.9850$& $44.2197$& $141.1108$&  {$3064.8487$}\\
					&{\tt RES}&$4.0221e-07$ & $4.6059e-07$ & $9.1711e-07$ & $7.8518e-07$ & $6.8357e-07$&  {$9.8667e-07$}\\
      \midrule
					\multirow{3}{*}{ {\it MAPSS}}&  {IT}&  {$5$} & {$5$} & {$6$} &  {$6$} &  {$6$}&  {$7$}  \\
					&  {CPU}&  {$0.21299$}&
 {$1.2038$}&
 {$8.0537$}&
 {$36.5074$}&
 {$116.7495$ }&  { {$514.1132$}}\\
					&{ {\tt RES}}& {$4.6434e-07$}&
 {$9.0780e-07$}&
 {$3.8651e-07$}&
 {$3.8350e-07$}&
 {$4.3826e-07$}  &  {$2.4983e-07$} \\
 \midrule
					\multirow{3}{*}{ {\textit{SL}}}&  {IT}&  {$6$} & {$6$} & {$5$} &  {$5$} &  {$5$}&  {$4$}  \\
					&  {CPU}&  {$0.19624$}&
 {$1.0819$}&
 {$6.3741$}&
 {$30.6976$}&
 {$102.0981$}&  {$481.7656$}\\
					&{ {\tt RES}}& {$2.5612e-08$}&
 {$7.4194e-08$}&
 {$9.4761e-08$}&
 {$3.7303e-09$}&
 {$1.7982e-09$}  &  {$6.1568e-08$} \\
					\midrule
					& & & &\textbf{ Case I}& & &\\
					\midrule
					\multirow{3}{*}{\it SS}& IT& $4$& $4$& $4$& $4$& $4$&  {$4$}\\
					& CPU&$0.2415$&$1.3375$&$6.7059$&$33.6902$&$111.3118$&  {$439.1221$} \\
					&{\tt RES}&$7.7528e-08$&
$5.5120e-08$&
$4.5134e-08$&
$3.9157e-08$&
$3.5073e-08$&  {$2.7836e-08$} \\
					\midrule
					\multirow{3}{*}{\it RSS}& IT& $4$& $4$& $4$& $4$& $4$&  {$4$} \\
					& CPU&$0.2776$&
$1.0467$&
$6.5207$&
$32.4428$&
$110.4587$&  {$432.6755$} \\
					&{\tt RES}&$8.0898e-08$&
$6.0033e-08$&
$4.9839e-08$&
$4.3510e-08$&
$3.9097e-08$ &  {$3.1111e-08$} \\
					\midrule
					\multirow{3}{*}{\it EGSS}& IT& $4$& $4$& $4$& $4$& $4$ &  {$4$} \\
					& CPU&$0.3061$&
$1.2771$&
$6.9552$&
$32.8960$&
$119.7763$&  {$435.0198$} \\
					&{\tt RES}&$5.9583e-10$&
$4.4128e-10$&
$3.6626e-10$&
$3.2009e-10$&
$2.8807e-10$ &  {$1.9145e-08$} \\
\midrule
					\multirow{3}{*}{ {\it RPGSS}}&  {IT}&  {$4$} & {$4$} & {$4$} &  {$4$} &  {$4$} &  {$3$} \\
					&  {CPU}&  {$0.2249$}&
 {$1.1312$}&
 {$7.0322$}&
 {$35.0140$}&
 {$130.2593$} &  {$371.4522$} \\
					&{ {\tt RES}}& {$5.9497e-10$}&
 {$4.4042e-10$}&
 {$3.6494e-10$}&
 {$3.1838e-10$}&
 {$2.8603e-10$} &  {$9.9326e-07$} \\
					\midrule
					\multirow{3}{*}{\textit{PESS}$^{\dagger}$}& IT& $\textbf{2}$& $\textbf{2}$& $\textbf{2}$& $\textbf{2}$& $\textbf{2}$&  {$\bf{2}$}\\
					& CPU&$0.2285$&
$\bf 0.9854$&
$\bf 4.8795$&
$\bf 23.3687$&
$\bf 79.9355$ &  {$\bf 283.8561$} \\
					$s=12$ &\tt RES&$3.1630e-07$&
$2.3239e-07$&
$1.9486e-07$&
$1.7260e-07$&
$1.5754e-07$  &  {$1.3135e-07$}\\
					\midrule
					\multirow{3}{*}{\textit{LPESS}$^{\dagger}$}& IT& $\textbf{2}$& $\textbf{2}$& $\textbf{2}$& $\textbf{2}$& $\textbf{2}$ &  {$\bf 2$}\\
					& CPU&$0.2834$&
$\bf 0.7989$&
$\bf 4.5599$&
$\bf 21.1664$&
$\bf 72.8376$&  {$\bf 268.3114$} \\
					$s=12$ &{\tt RES}&$3.1180e-07$&
$2.2463e-07$&
$1.8481e-07$&
$1.6071e-07$&
$1.4411e-07$ &  {$1.1441e-07$} \\
					\midrule
					& & & & \textbf{Case II}&&  \\
					\midrule
					\multirow{3}{*}{\it SS}& IT& $7$& $7$& $7$& $7$& $7$ &  {$7$}\\
					& CPU&0.2780&
$1.6189$&
$10.6707$&
$54.4493$&
$184.6025$ &  {$703.1082$} \\
					&{\tt RES}& $6.7967e-07$&
$4.9738e-07$&
$4.1383e-07$&
$3.6189e-07$&
$3.2556e-07$ &  {$2.5956e-07$} \\
					\midrule
					\multirow{3}{*}{\it RSS}& IT& $7$& $7$& $7$& $7$& $7$&  {$7$}\\
					& CPU&$0.3248$&
$1.4644$&
$10.4985$&
$52.4559$&
$177.1895$&  {$690.1611$}\\
					&{\tt RES}&$5.2397e-07$&
$3.6832e-07$&
$2.7532e-07$&
$2.1869e-07$&
$1.8240e-07$ &  {$1.2720e-07$} \\
					\midrule
					\multirow{3}{*}{\it EGSS}& IT& $5$& $5$& $4$& $4$& $4$ &  {$4$}\\
					& CPU&$0.2588$&
$1.3792$&
$6.7546$&
$34.8358$&
$118.2298$ &  {$439.7907$}\\
					&{\tt RES}&$6.8491e-08$&
$1.5239e-07$&
$9.7099e-07$&
$6.3466e-07$&
$4.8159e-07$&  {$4.1376e-07$} \\
\midrule
					\multirow{3}{*}{ {\it RPGSS}}& IT&  {$4$} & {$4$} & {$4$} &  {$4$} &  {$4$} &  {$3$} \\
					&  {CPU}&  {$0.2445$}&
 {$1.2987$}&
 {$8.1982$}&
 {$28.5129$}&
 {$104.9127$} &  {$397.4766$} \\
					&{ {\tt RES}}& {$5.2642e-08$}&
 {$1.1366e-07$}&
 {$7.1116e-08$}&
 {$8.2780e-07$}&
 {$5.6396e-07$} &  {$2.5131e-07$} \\
					\midrule
					\multirow{3}{*}{\textit{PESS}$^{\dagger}$}& IT& $3$ & $3$& $3$& $3$& $3$ &  {$3$}\\
					& CPU &$0.2971$&
$1.1597$&
$6.6682$&
$31.2670$&
$103.8616$ &  {$371.1214$} \\
					$s=12$ &{\tt RES} & $7.4100e-08$&
$7.4642e-08$&
$7.3327e-08$&
$6.9803e-08$&
$6.1920e-08$ &  {$4.0967e-08$} \\
					\midrule
					\multirow{3}{*}{\textit{LPESS}$^{\dagger}$}& IT& $3$& $3$& $3$& $3$& $3$ &  {$3$}\\
					& CPU&$0.2226$&$0.9905$&$6.1339$&$30.0975$&$98.6443$&  {$359.8141$} \\
					$s=12$ &{\tt RES}& $1.1683e-09$&$2.3215e-09$&$3.1540e-09$&$3.3907e-09$&$3.0271e-09$ &  {$1.6141e-09$} \\
					\bottomrule
					\multicolumn{8}{l}{Here ${\dagger}$ represents the proposed preconditioners. The boldface represents the top two results.}\\
					\multicolumn{8}{l}{ $\bf{--}$ indicates that the iteration process does not converge within the prescribed IT.}
				\end{tabular}
			}
		\end{table} 
 {\textbf{Problem formulation:}}  We consider the three-by-three block {\it SPP} (\ref{eq11}) taken from \cite{HuangNA} with
  
			\noin $A= \bmatrix{I\otimes G+G\otimes I &0\\ 0&I\otimes G+ G\otimes I}\in \R^{2l^2\times2l^2}, \,\, B=\bmatrix{I\otimes F& F\otimes I}\in \R^{l^2\times 2l^2}, $
	   $C= E\otimes F\in \R^{l^2\times l^2},$
		where $G=\frac{1}{(l+1)^2}\, \mathrm{tridiag}(-1,2,-1)\in \R^{l\times l},\quad F=\frac{1}{l+1}\,  \mathrm{tridiag}(0,1,-1)\in \R^{l\times l}$  and  $ E=\diag(1, l+1, \ldots, l^2-l+ 1)\in \R^{l\times l}.$ Here, $E\otimes F$ denotes the Kronecker product of two matrices $E$ and $F$ and $ \mathrm{tridiag}(a,b,c)\in \R^{l\times l}$ denotes the 
 $l \times l$ tridiagonal matrix with  {diagonal entries $b$, subdiagonal entries $a$, and superdiagonal entries $c$.} For this problem, the size of the matrix $\A$ is $4l^2.$ 
 
 \vspace{2mm}
\noindent {\textbf{Parameter selection:}} Following \cite{HuangNA}, selection of $\widehat{A}$ and $\widehat{S}$  ({\it SPD} approximations of $A$ and $S,$ respectively) for {\it IBD} preconditioner   are done as follows:
		\begin{equation*}
			\widehat{A}=LL^T, \quad \widehat{S}=\diag(B\widehat{A}^{-1}B^T),
		\end{equation*}
		where  $L$ is the incomplete Cholesky factor of $A$ produced by the Matlab function: 
		$$ \textit{\textbf{ichol(A, struct(`type', `ict', `droptol', \text{1e-8}, `michol', `off'))}}.$$
   {For the preconditioner $\P_{\textit{MAPSS}},$ we take $\alpha=\sqrt[4]{\frac{tr(BB^TC^TC)}{m}}$ and $\beta= 10^{-4}$ as in \cite{MAPSS2023}.}
  We consider the parameter choices for the preconditioners $\P_{SS},\P_{RSS},\P_{EGSS},$ $  {\P_{RPGSS}},$ $ \P_{PESS}$ and $\P_{LPESS}$ in  {the following} two cases.
		\begin{itemize}
			\item In Case I: $\alpha=0.1$  for $\P_{SS}$  and $\P_{RSS};$  $\alpha=0.1, \beta=1,\gamma=0.001$ and  $P=I,Q=I,W=I$ for $\P_{EGSS}$ and  {$\P_{RPGSS}$}; and  $\Lambda_1= I,\Lambda_2= I,\Lambda_3= 0.001 I$ for $\P_{PESS}$ and $\P_{LPESS}.$
			\item In Case II:  $\alpha=1$ for $\P_{SS}$  and $\P_{RSS};$ $ \alpha=1, \beta=1, \gamma=0.001$ and $P=A,Q=I,W=CC^T$ for $\P_{EGSS}$  and  {$\P_{RPGSS}$}; and $\Lambda_1=A,\Lambda_2= I,\Lambda_3=0.001 CC^T$  for $\P_{PESS}$ and $\P_{LPESS}.$ 
		\end{itemize} 
	The parameter selection in Case II is made as in \cite{EGSS23}.	
 
 \vspace{2mm}
\noindent  {\textbf{Results for experimentally found optimal parameter:}} The optimal value (experimentally) for  $s$ in the range  $[10,20]$ is determined to be $12$   for minimal CPU times. 
  The test problems generated for the values of $l=16,32,48,64,80,  {128},$ and their numerical results are reported in  Table \ref{tab2}. 

   \begin{table}[ht!]
       \centering
         \caption{  Numerical results of \textit{PESS-I}, \textit{LPESS-I}, \textit{PESS-II} and \textit{LPESS-II} \textit{PGMRES} processes for Example \ref{ex1}.}
          {\resizebox{15cm}{!}{
         \begin{tabular}{cccccccc}
						\toprule
						Process& $l$	& $16$ &$32$& $48$ & $64$ &  {$80$}& $128$\\
						\midrule 
						& size($\A$)	& $1024$ &$4096$& $9216$ & $16384$& {$25600$}& $65536$\\
						\midrule
						\multirow{1}{*}{{\it PESS-I}}& IT& $2$&$2$&$2$&$2$& $2$& $2$\\
				($s=1,$ $\Lambda_1=0.01I,$		& CPU&  $0.2517$&$0.8373$&$4.6626$&$21.8133$ & {$77.9420$}&$308.1418$\\
		$\Lambda_2=0.1I,$ $\Lambda_3=0.001I$)				&{\tt RES}&$4.4970e-07$&$3.8763e-07$&$3.6261e-07$&$3.4880e-07$&$3.3983e-07$&$7.7830e-07$\\	
						\midrule
      \multirow{1}{*}{{\it LPESS-I}}& IT& $2$&$2$&$2$&$2$& $2$& $2$\\
			($s=1,$ $\Lambda_2=0.1I,$			& CPU&  $0.3070$&$0.8184$&$4.7244$&$21.6534$ & {$82.3687$}&$288.9701$\\
   $\Lambda_3=0.001I$ )  &{\tt RES}&$3.1116e-07$&$2.2410e-07$&$1.8435e-07$&$1.6030e-07$&$1.4375e-07$&$1.1441e-07$\\
      \midrule
      \multirow{1}{*}{{\it PESS-II}}& IT& $3$&$3$&$3$&$3$& $3$& $3$\\
				($s=s_{est},$ $\Lambda_1=A,$		& CPU& $0.20373$&$1.05327$&$5.84150$ & $29.3914$& $104.9663$&$424.2004336$\\ 
$\Lambda_2=\beta_{est}I,$  $\Lambda_3= 10^{-4}CC^T$)   &{\tt RES}&  $2.2013e-08$&$3.4850e-08$&$4.3689e-08$&$5.0705e-08$&$5.8005e-08$& $7.2906e-08$\\
      \midrule
      \multirow{1}{*}{{\it LPESS-II}}& IT& $3$&$3$&$3$&$3$& $3$& $3$\\
				($s=s_{est},$ $\Lambda_2=\beta_{est}I,$		& CPU&  $0.19215$&$1.04964$&$4.63143$&$24.20924$ & {$102.84651$}&$389.10761$\\ 
  $\Lambda_3= 10^{-4}CC^T$)   &{\tt RES}&   $1.8126e-09$& $1.0207e-13$&$8.0620e-15$&$9.1601e-14$&$5.7292e-13$&$1.1950e-07$\\
      \midrule
       \end{tabular}}}
       \label{tab:rev1}
   \end{table}
  \vspace{2mm}
\noindent   {\textbf{Results using parameters selection strategy in Section \ref{Sec:parameter}:} To demonstrate the effectiveness of the proposed preconditioners \textit{PESS} and \textit{LPESS} using the parameters discussed in Section \ref{Sec:parameter}, we present the numerical  results by choosing $s=1,$ $\Lambda_1=0.01I, \Lambda_2=0.1I, \Lambda_3=0.001I$ (denoted by \textit{PESS-I} and \textit{LPESS-I}) and $s=s_{est},$ $\Lambda_1=A,$ $\Lambda_2=\beta_{est}I$ and $\Lambda_3=10^{-4}C^TC$ (denoted by \textit{PESS-II} and \textit{LPESS-II}) for the proposed preconditioner. These results are summarized in Table  \ref{tab:rev1}. } 
   
 \vspace{2mm}
\noindent  {\textbf{Convergence curves:}}  Figure \ref{fig1} illustrates convergence curves pertaining to preconditioners {\it BD}, {\it IBD},  {{\it MAPSS}, \textit{SL},} {\it SS}, {\it RSS}, {\it EGSS},  {\textit{RPGSS},} \textit{PESS} and \textit{LPESS} ($s= 12$) for Case II  with $l=16, 32, 48,64, 80$ and  {$128$}. These curves depict the relationship between the relative residue at each iteration step (RES) and IT counts.
  
 \vspace{2mm}
\noindent  {\textbf{Spectral distributions:}} To further illustrate the superiority of the \textit{PESS}  preconditioner, spectral distributions of  $\A$ and  the preconditioned matrices $\P_{BD}^{-1}\A, \P_{IBD}^{-1}\A,$ $  {\P_{MAPSS}^{-1}\A},$ $ {\P_{SL}^{-1}\A},$ $\P_{SS}^{-1}\A,$ $ \P_{RSS}^{-1}\A,$ $ \P_{EGSS}^{-1}\A,$ $ {\P_{RPGSS}^{-1}\A }, \P_{PESS}^{-1}\A$  and $\P_{LPESS}^{-1}\A$ for the Case II with $l=16$ and $s=13$ are displayed in Figure \ref{fig2}.
			\begin{figure}[ht!]
				\centering
				\begin{subfigure}[b]{0.3\textwidth}
					\centering
					\includegraphics[width=\textwidth]{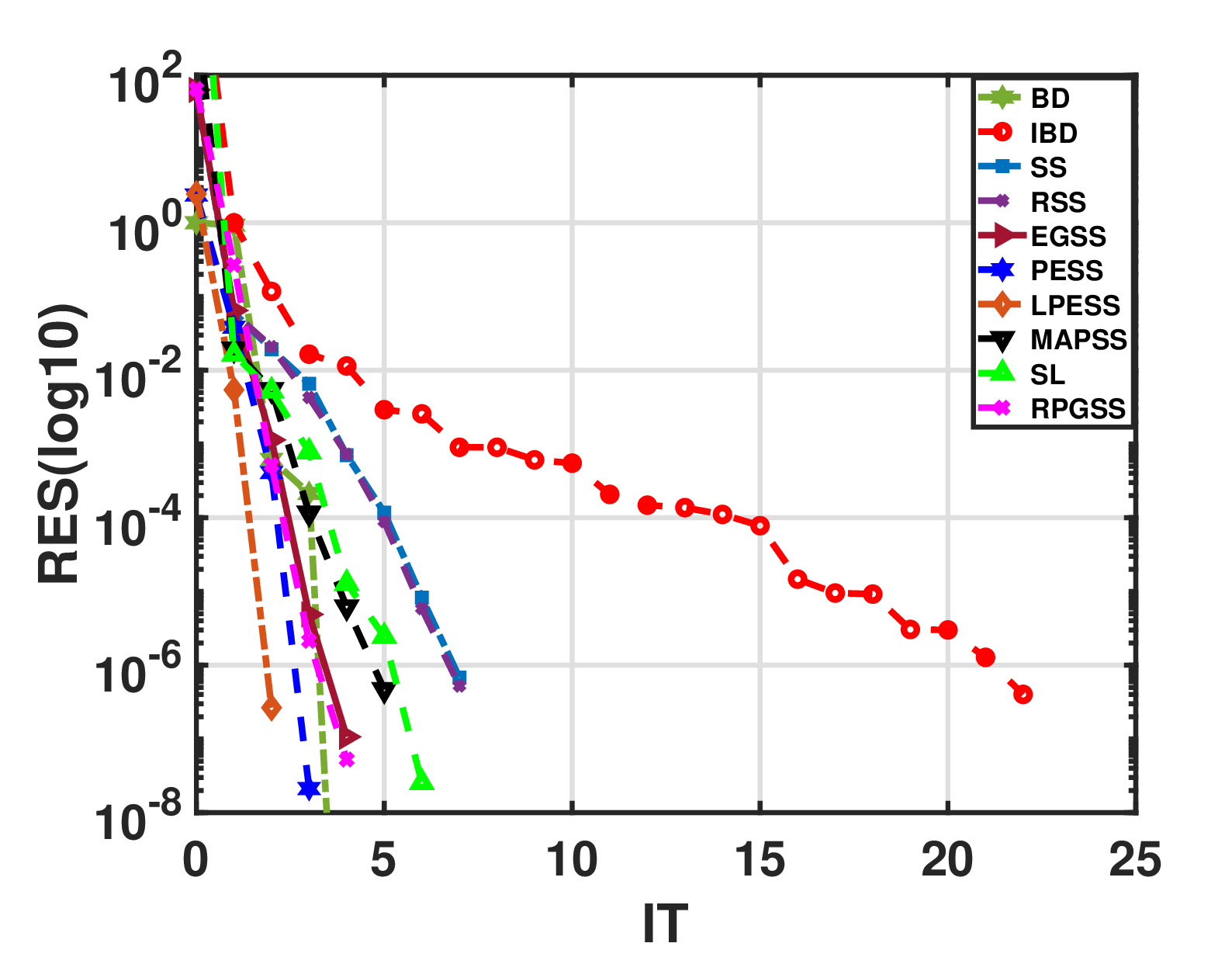}
					\caption{ $l=16$}
					\label{fig:RES16}
				\end{subfigure}
				\begin{subfigure}[b]{0.3\textwidth}
					\centering
					\includegraphics[width=\textwidth]{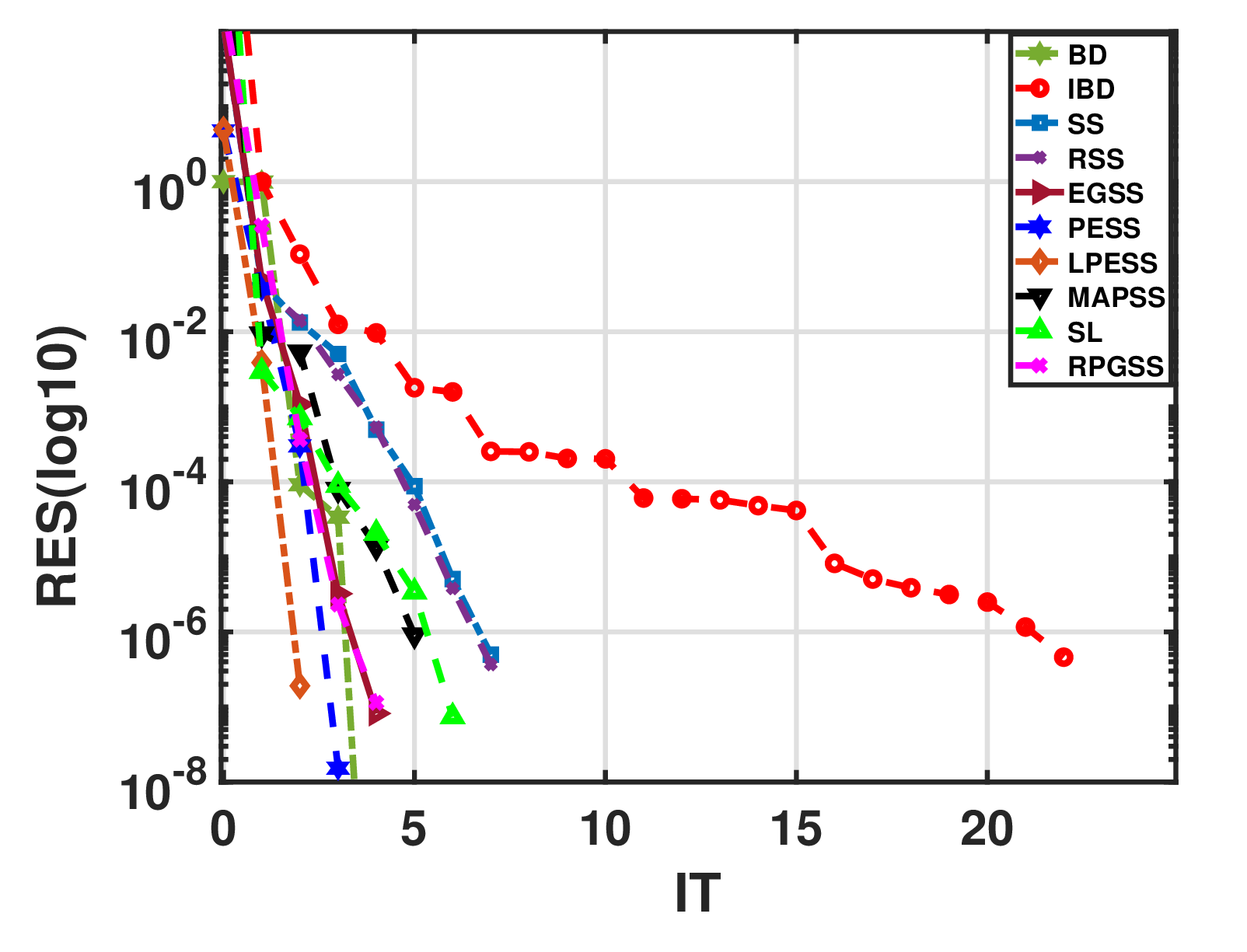}
					\caption{ $l=32$}
					\label{fig:RES32}
				\end{subfigure}
				\begin{subfigure}[b]{0.3\textwidth}
					\centering
					\includegraphics[width=\textwidth]{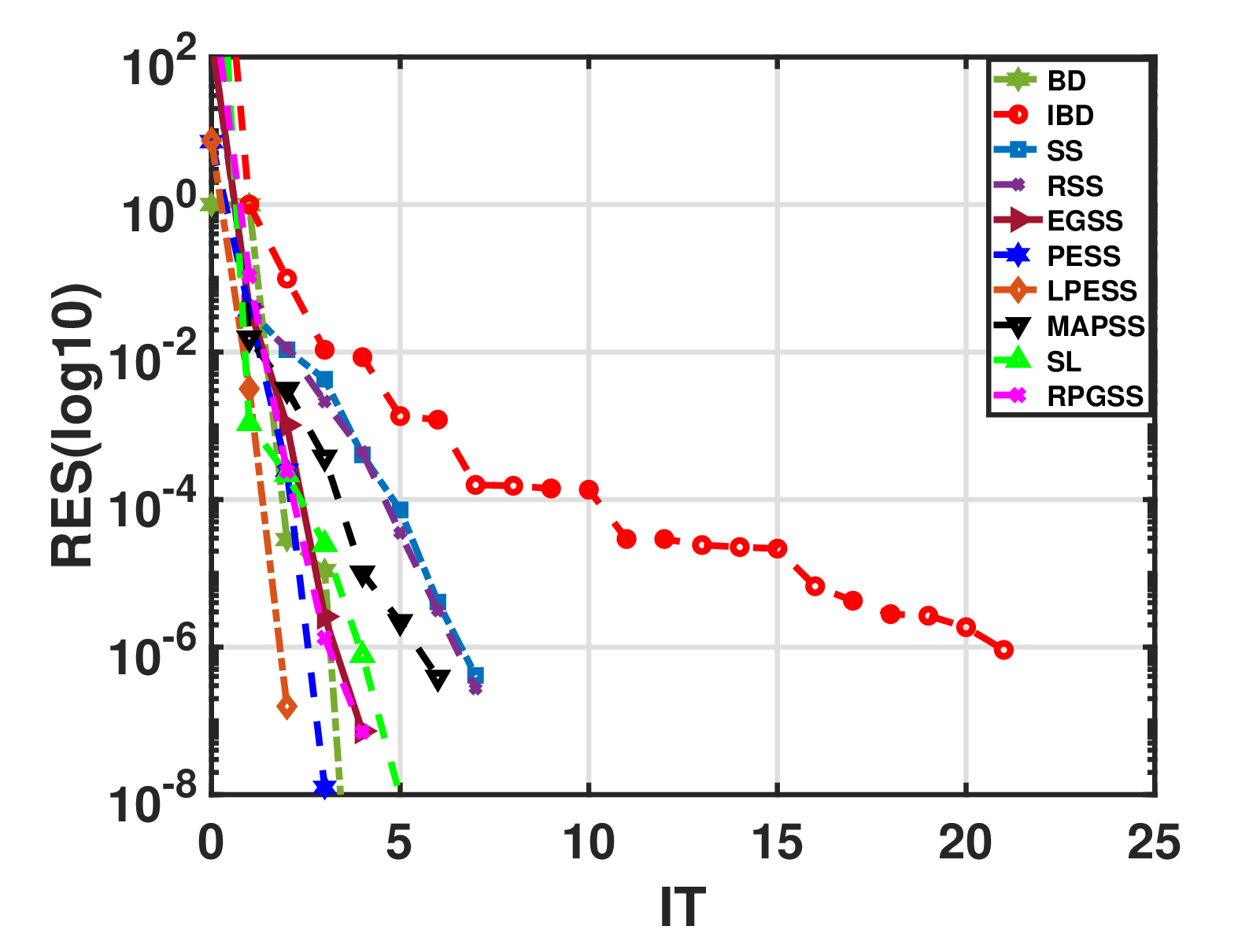}
					\caption{ $l=48$}
					\label{fig:RES48}
				\end{subfigure}\\
				\begin{subfigure}[b]{0.3\textwidth}
					\centering
					\includegraphics[width=\textwidth]{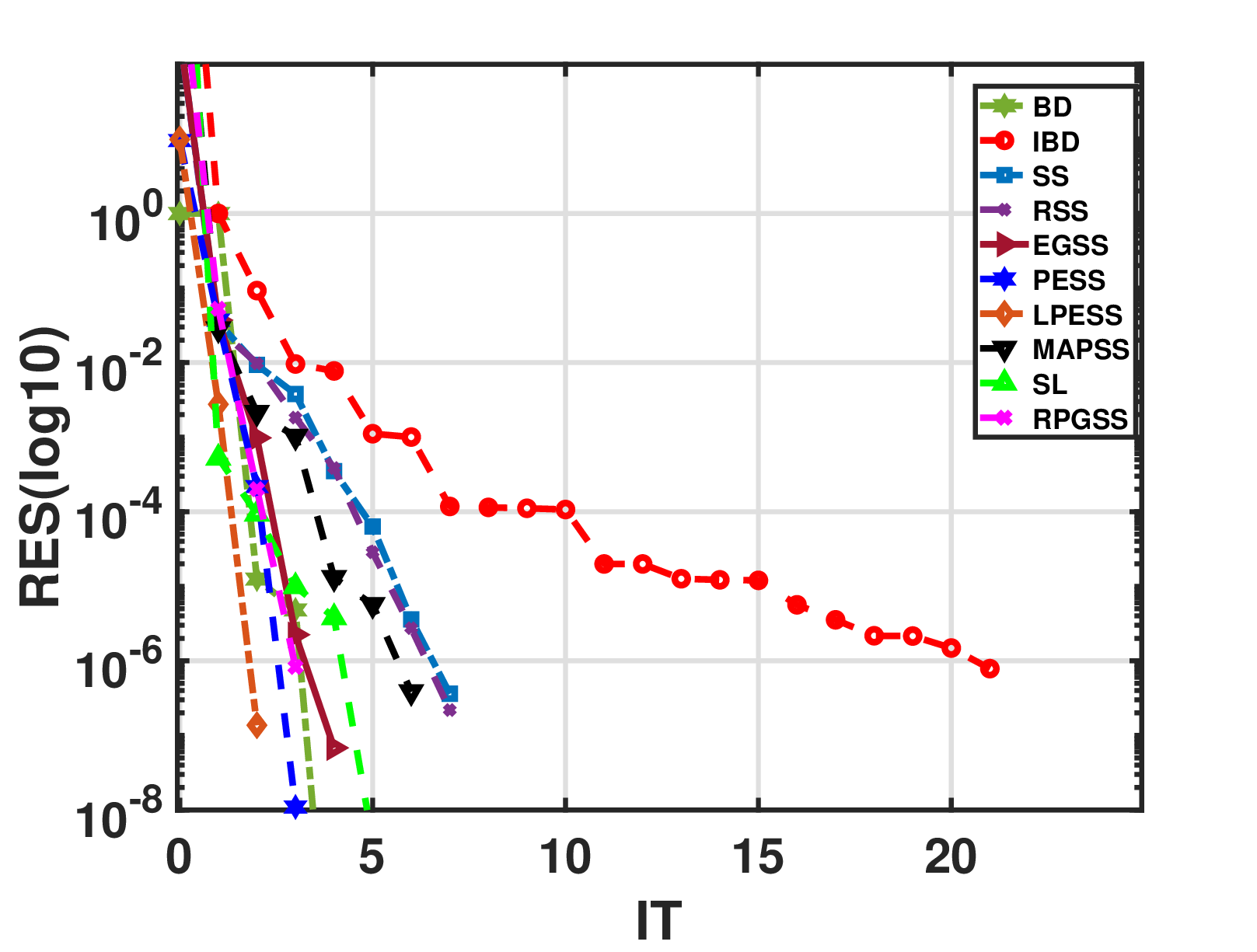}
					\caption{ $l=64$}
					\label{fig:RES64}
				\end{subfigure}
				\begin{subfigure}[b]{0.3\textwidth}
					\centering
					\includegraphics[width=\textwidth]{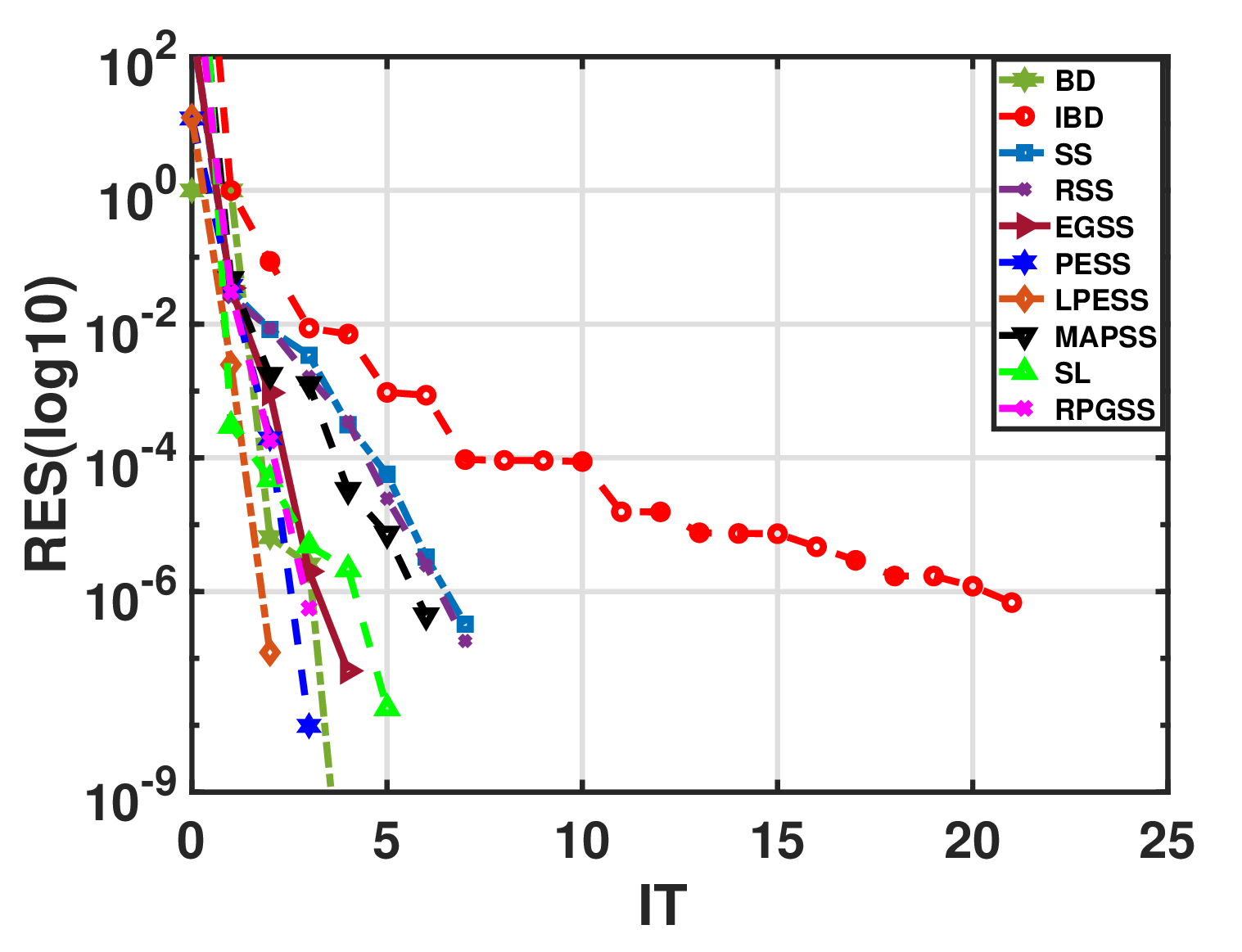}
					\caption{ $l=80$ }
					\label{fig:RES80}
				\end{subfigure}
    \begin{subfigure}[b]{0.34\textwidth}
					\centering
					\includegraphics[width=\textwidth]{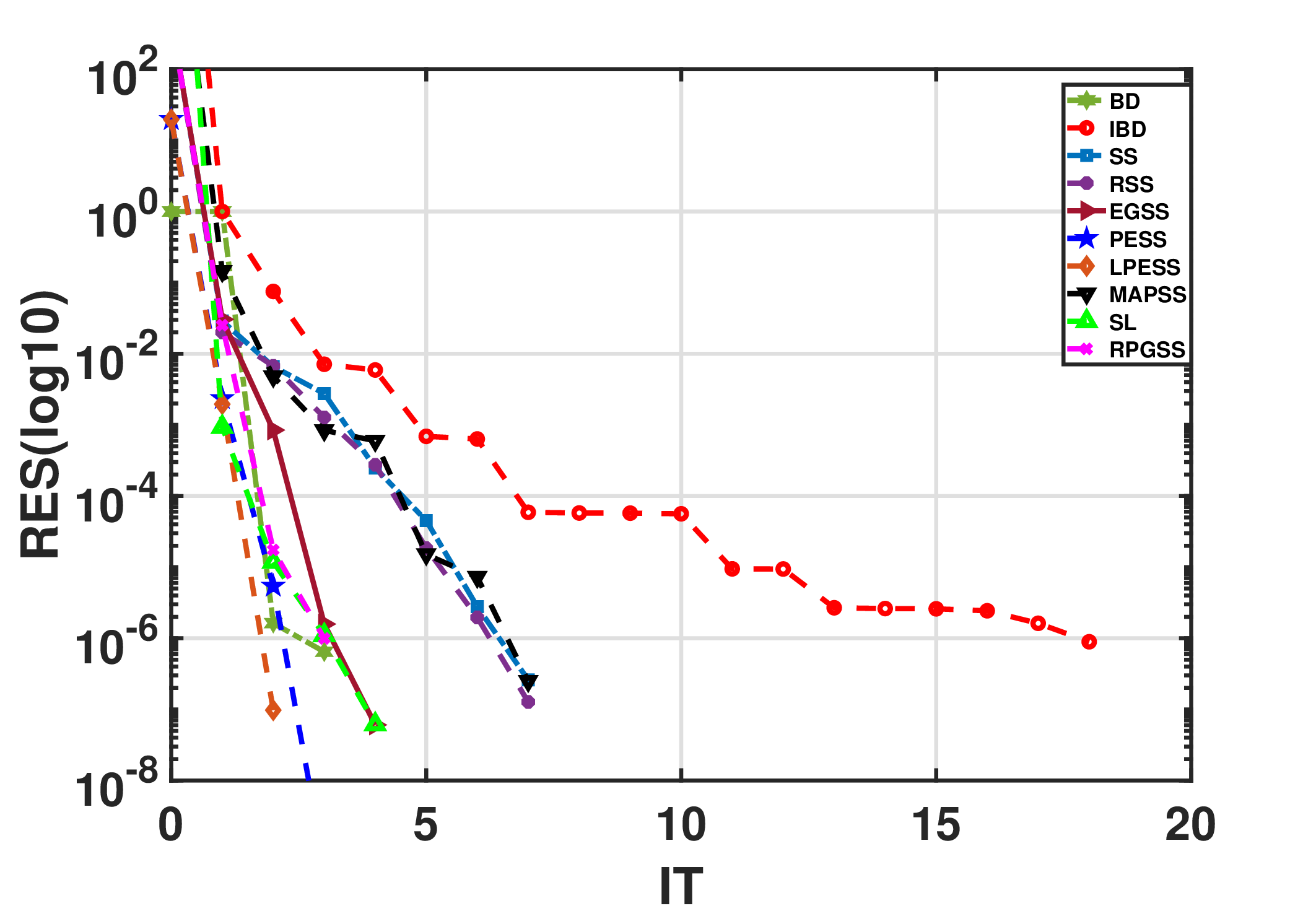}
					\caption{  $l=128$}
					\label{fig:RES128}
				\end{subfigure}
				\caption{Convergence curves for IT versus RES of {\it PGMRES} processes employing {\it BD, IBD,  {MAPSS, SL,} SS, RSS, EGSS,  {RPGSS,}  PESS} and \textit{LPESS} $(s=12)$  preconditioners in Case II for Example \ref{ex1}.}
				\label{fig1}
			\end{figure} 

			\begin{figure}[ht!]
				\centering
				\begin{subfigure}[b]{0.3\textwidth}
					\centering
					\includegraphics[width=\textwidth]{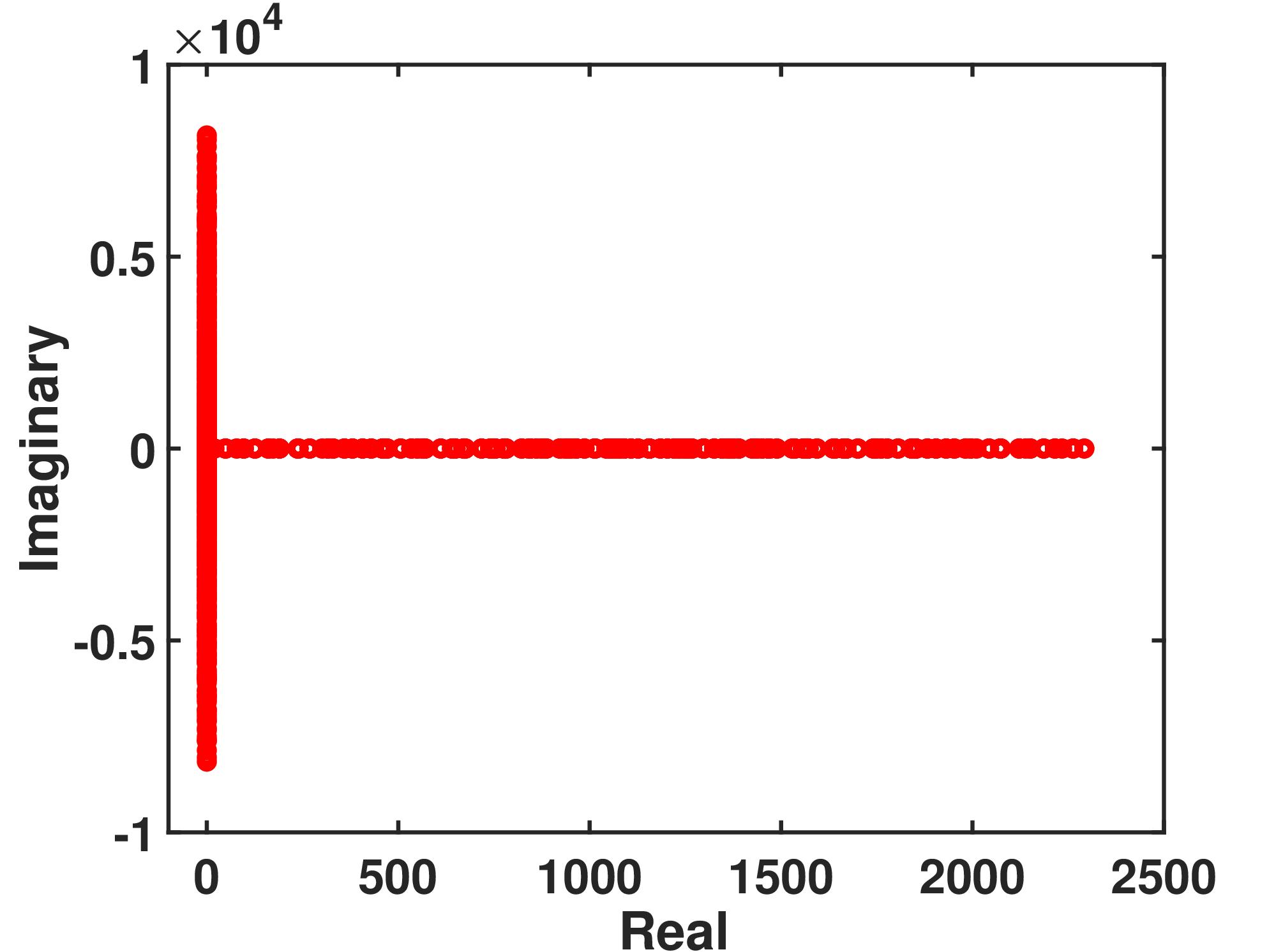}
					\caption{ $\A$ }
					\label{fig:original}
				\end{subfigure}
				\begin{subfigure}[b]{0.3\textwidth}
					\centering
					\includegraphics[width=\textwidth]{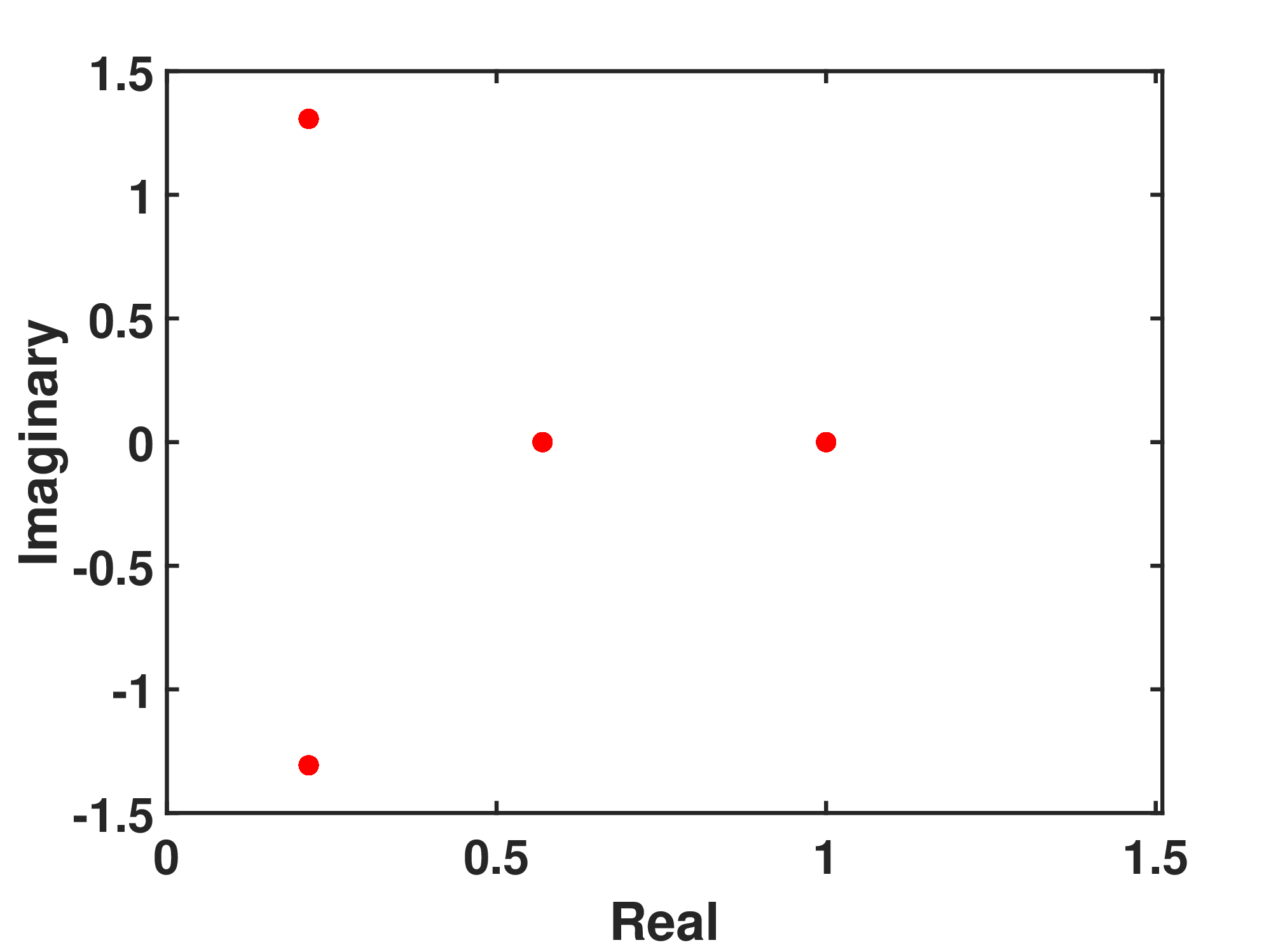}
					\caption{ $\P_{BD}^{-1}\A$}
					\label{fig:BD}
				\end{subfigure}
				\begin{subfigure}[b]{0.3\textwidth}
					\centering
					\includegraphics[width=\textwidth]{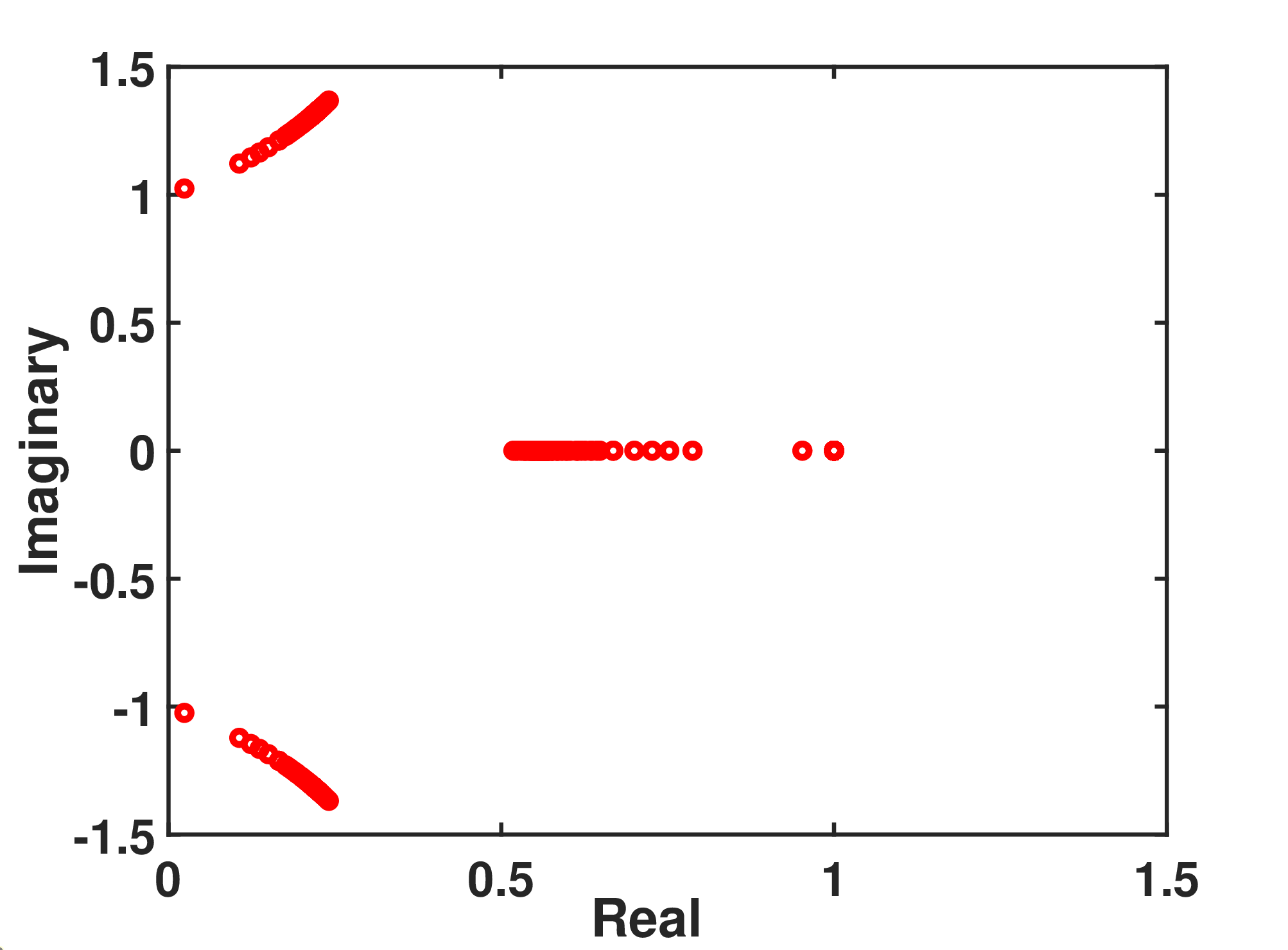}
					\caption{ $\P_{IBD}^{-1}\A$ }
					\label{fig:InBD}
				\end{subfigure}
    \begin{subfigure}[b]{0.3\textwidth}
					\centering
					\includegraphics[width=\textwidth]{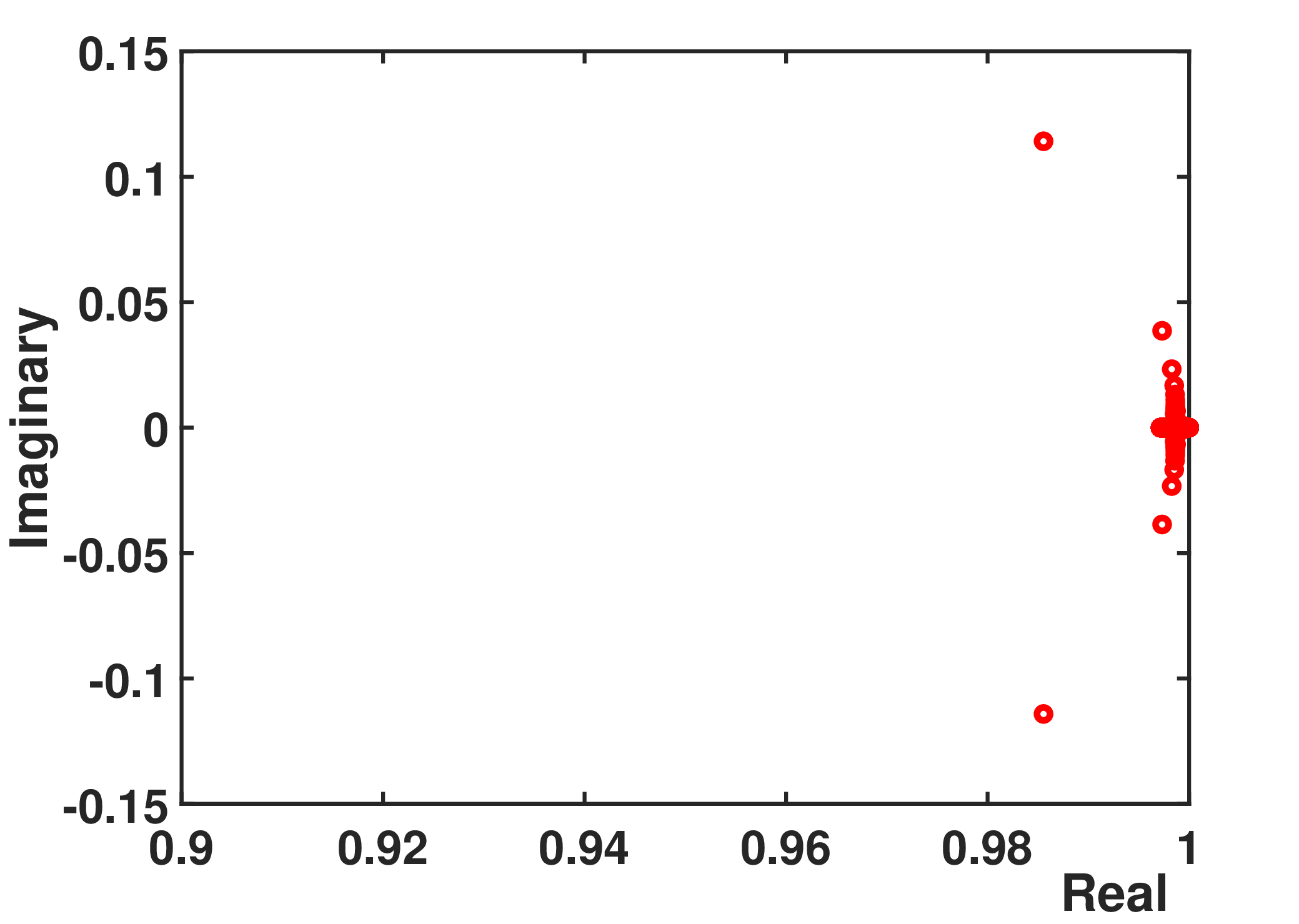}
					\caption{  $\P_{MAPSS}^{-1}\A$ }
					\label{fig:MAPSS}
				\end{subfigure}
     \begin{subfigure}[b]{0.3\textwidth}
					\centering
					\includegraphics[width=\textwidth]{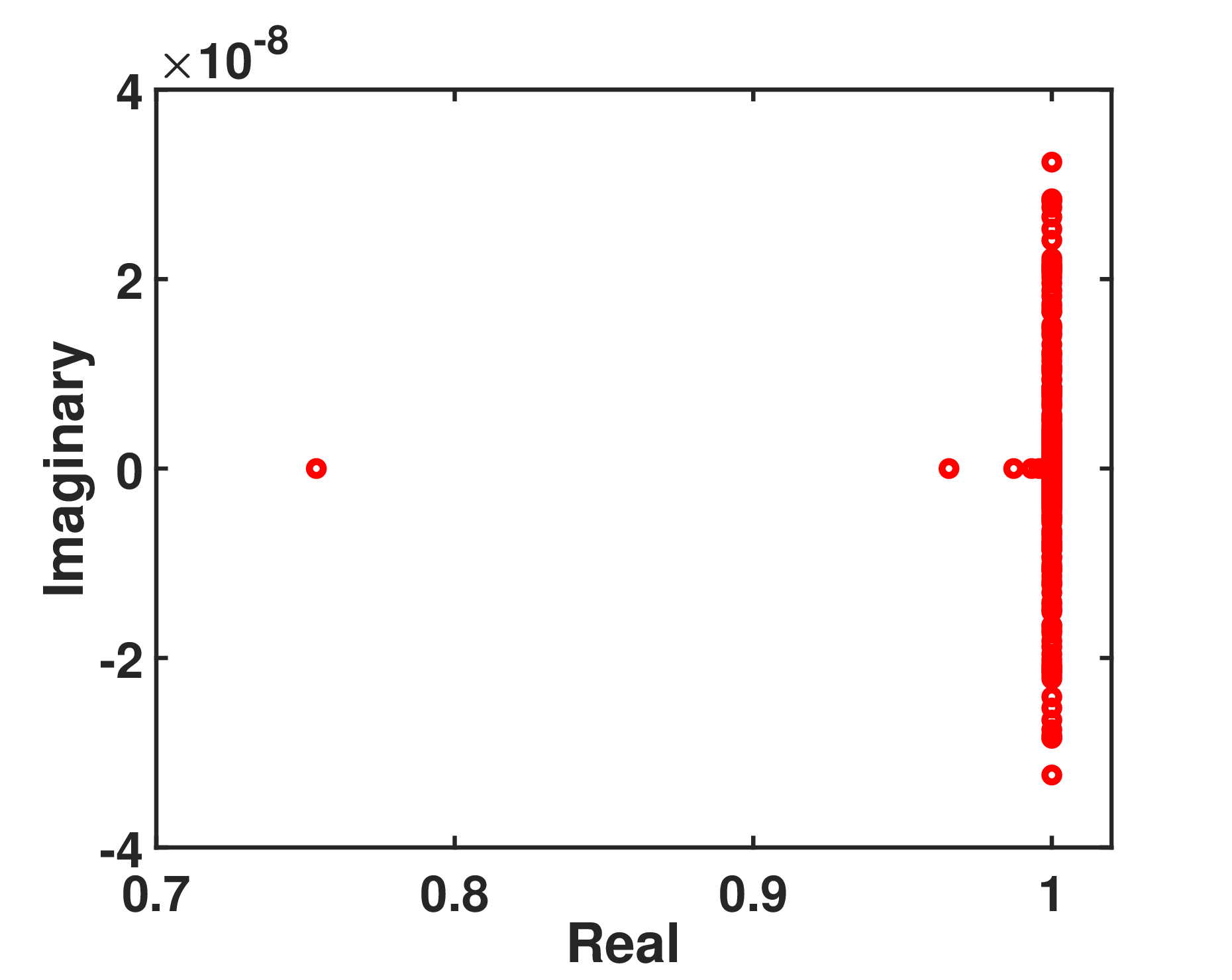}
					\caption{  $\P_{SL}^{-1}\A$ }
					\label{fig:TBDP}
				\end{subfigure}
				\begin{subfigure}[b]{0.31\textwidth}
					\centering
					\includegraphics[width=\textwidth]{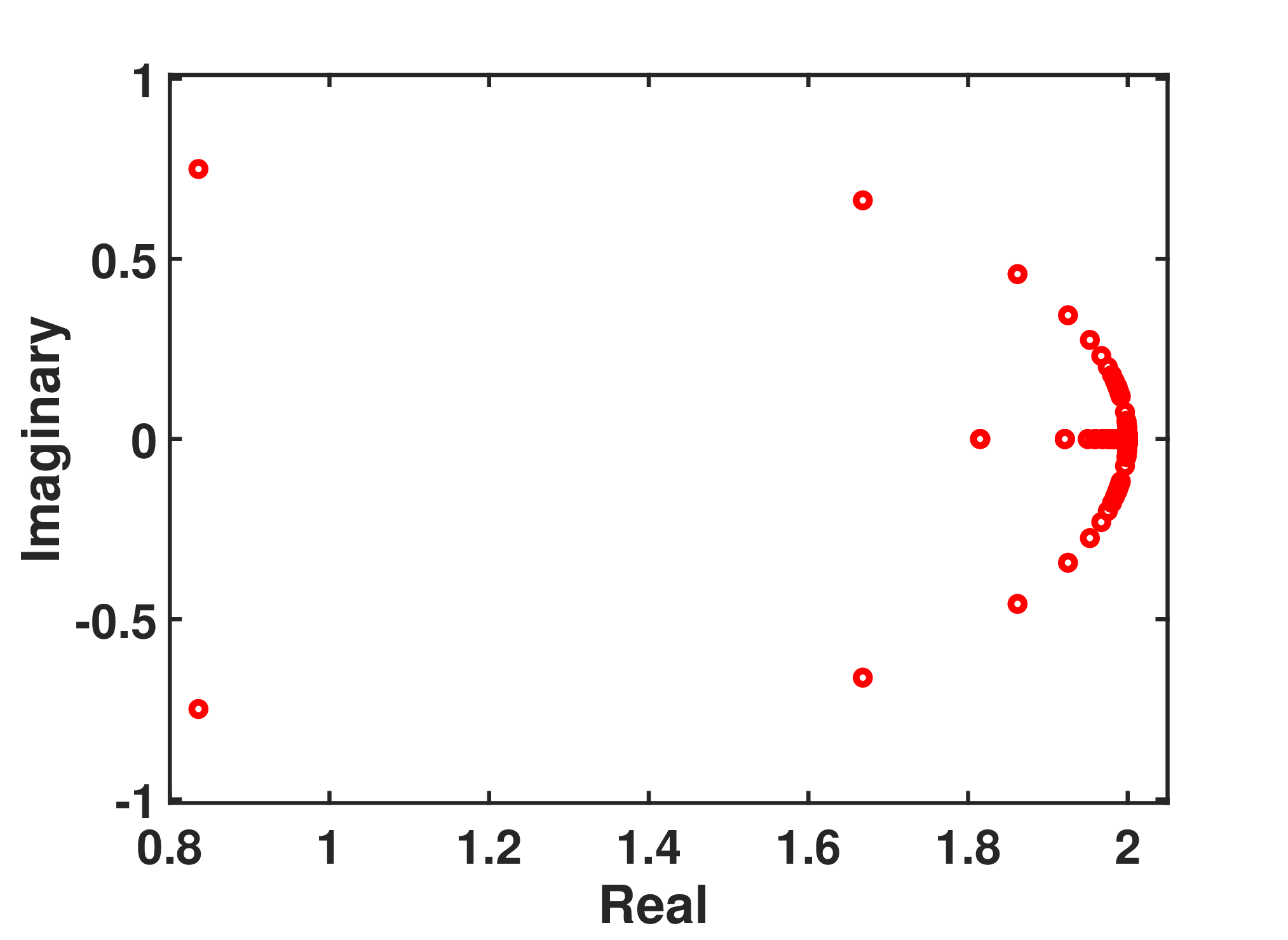}
					\caption{ $\P_{SS}^{-1}\A$ }
					\label{fig:SS}
				\end{subfigure}
				\begin{subfigure}[b]{0.315\textwidth}
					\centering
					\includegraphics[width=\textwidth]{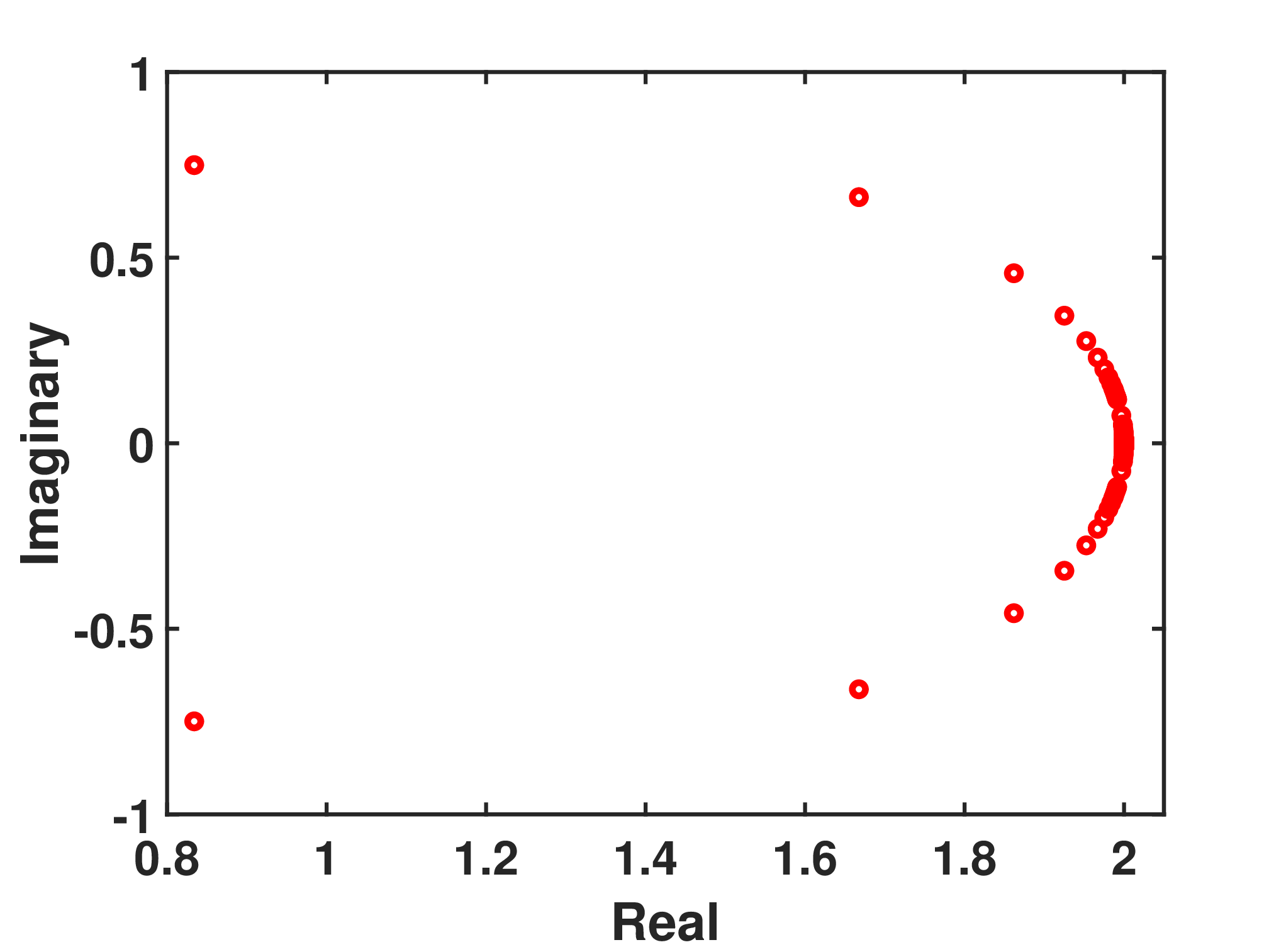}
					\caption{ $\P_{RSS}^{-1}\A$ }
					\label{fig:Rss}
				\end{subfigure}
				\begin{subfigure}[b]{0.3\textwidth}
					\centering
					\includegraphics[width=\textwidth]{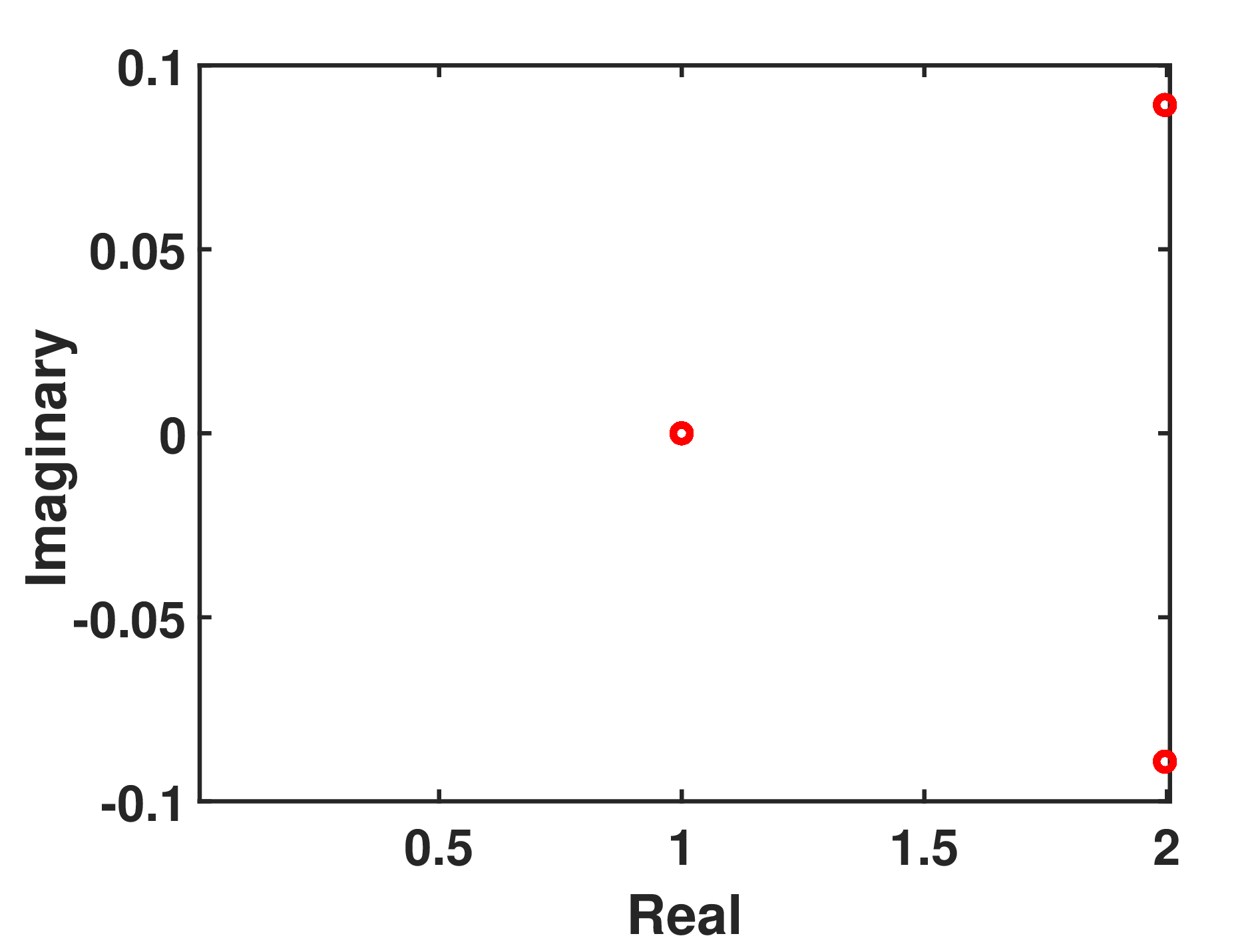}
					\caption{ $\P_{EGSS}^{-1}\A$ }
					\label{fig:EGSS}
				\end{subfigure}
    \begin{subfigure}[b]{0.3\textwidth}
					\centering
					\includegraphics[width=\textwidth]{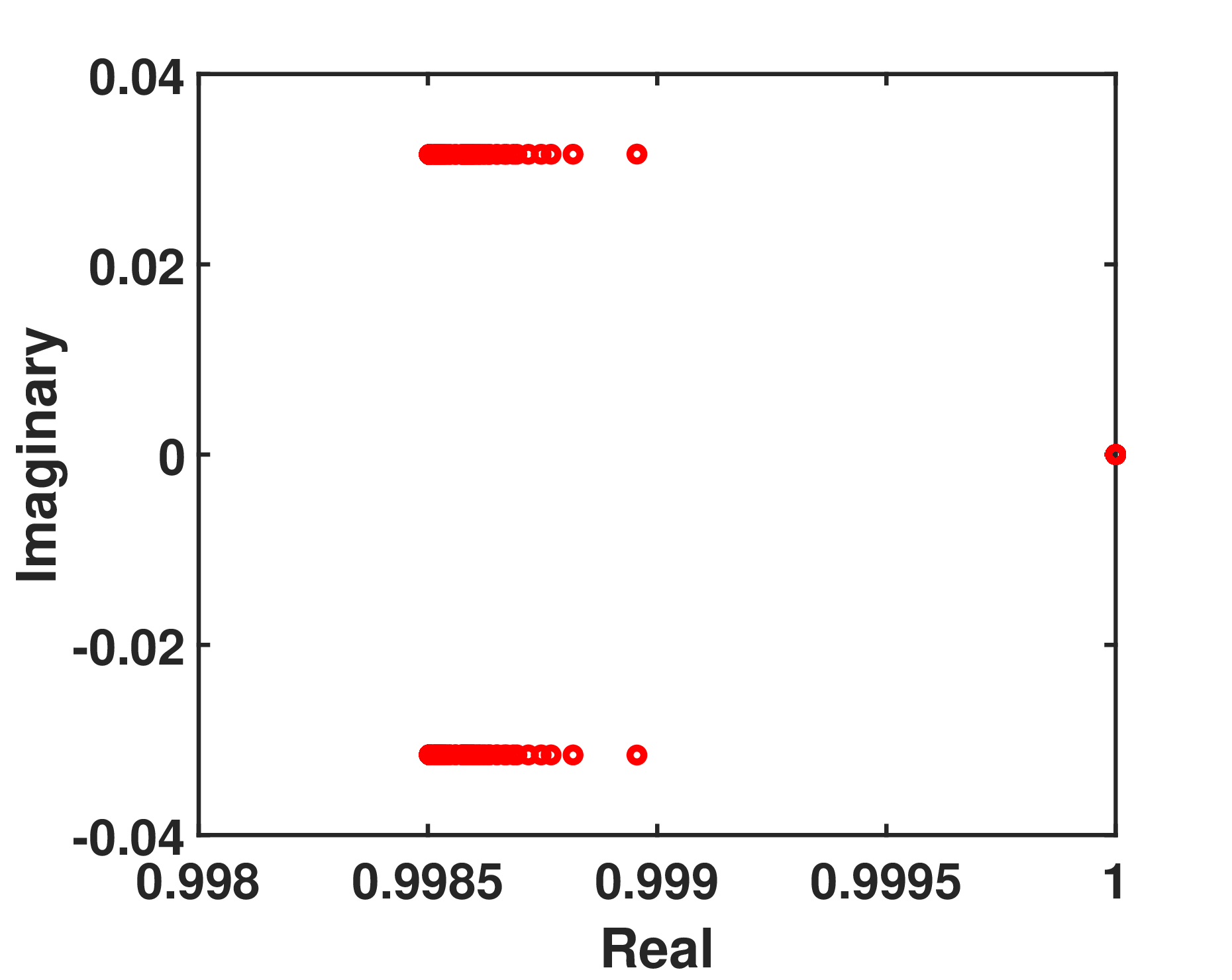}
					\caption{  $\P_{RPGSS}^{-1}\A$ }
					\label{fig:RPGSS1}
				\end{subfigure}
				\begin{subfigure}[b]{0.3\textwidth}
					\centering
					 \includegraphics[width=\textwidth]{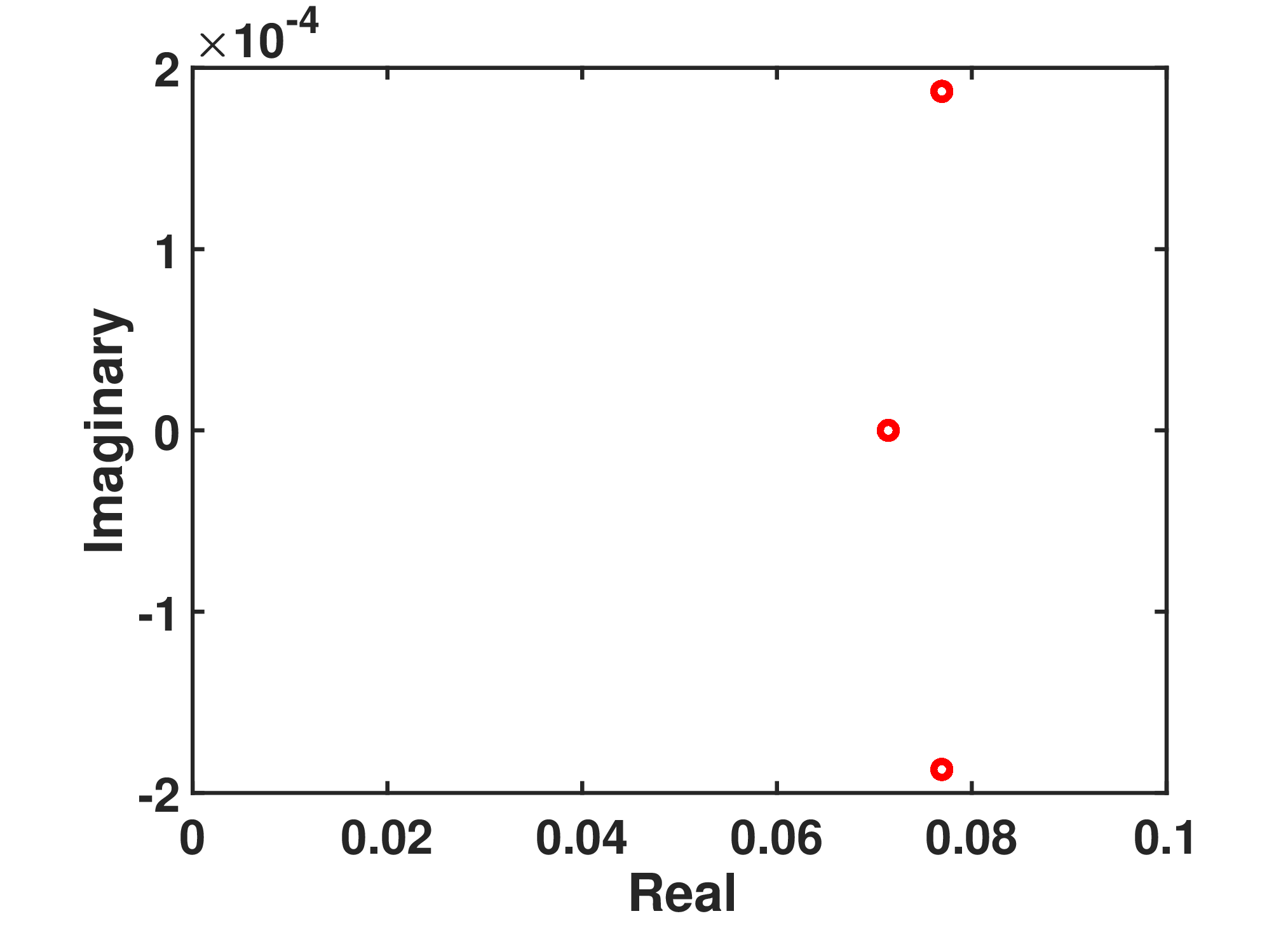}
					\caption{$\P_{PESS}^{-1}\A$ with $s=13$}
					\label{fig:pess}
				\end{subfigure}
    \begin{subfigure}[b]{0.3\textwidth}
					\centering
     \includegraphics[width=\textwidth]{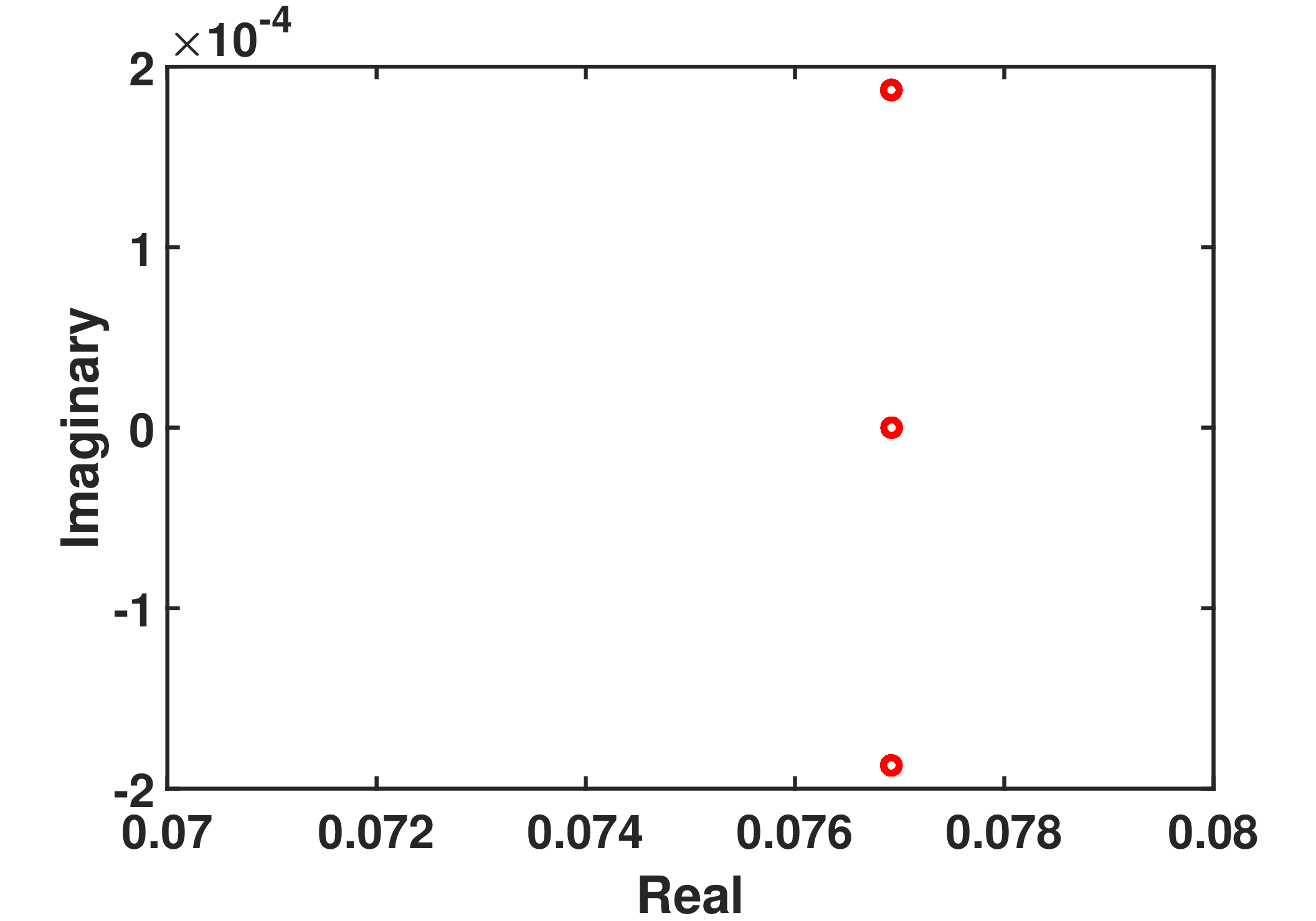}
					\caption{$\P_{LPESS}^{-1}\A$ with $s=13$}
					\label{fig:Lpess}
				\end{subfigure}
				\caption{Spectral distributions of $\A,\P_{BD}^{-1}\A, \P_{IBD}^{-1}\A,  {\P_{MAPSS}^{-1}\A, \P_{SL}^{-1}\A,} \P_{SS}^{-1}\A, \P_{EGSS}^{-1}\A,  {\P_{RPGSS}^{-1}\A,} \P_{PESS}^{-1}\A$ and $\P_{LPESS}^{-1}\A$ for Case II with $l=16$ for Example \ref{ex1}.}
				\label{fig2}
			\end{figure}

			\begin{figure}[ht!]
				\centering
				\begin{subfigure}[b]{0.4\textwidth}
					\centering
     \includegraphics[width=\textwidth]{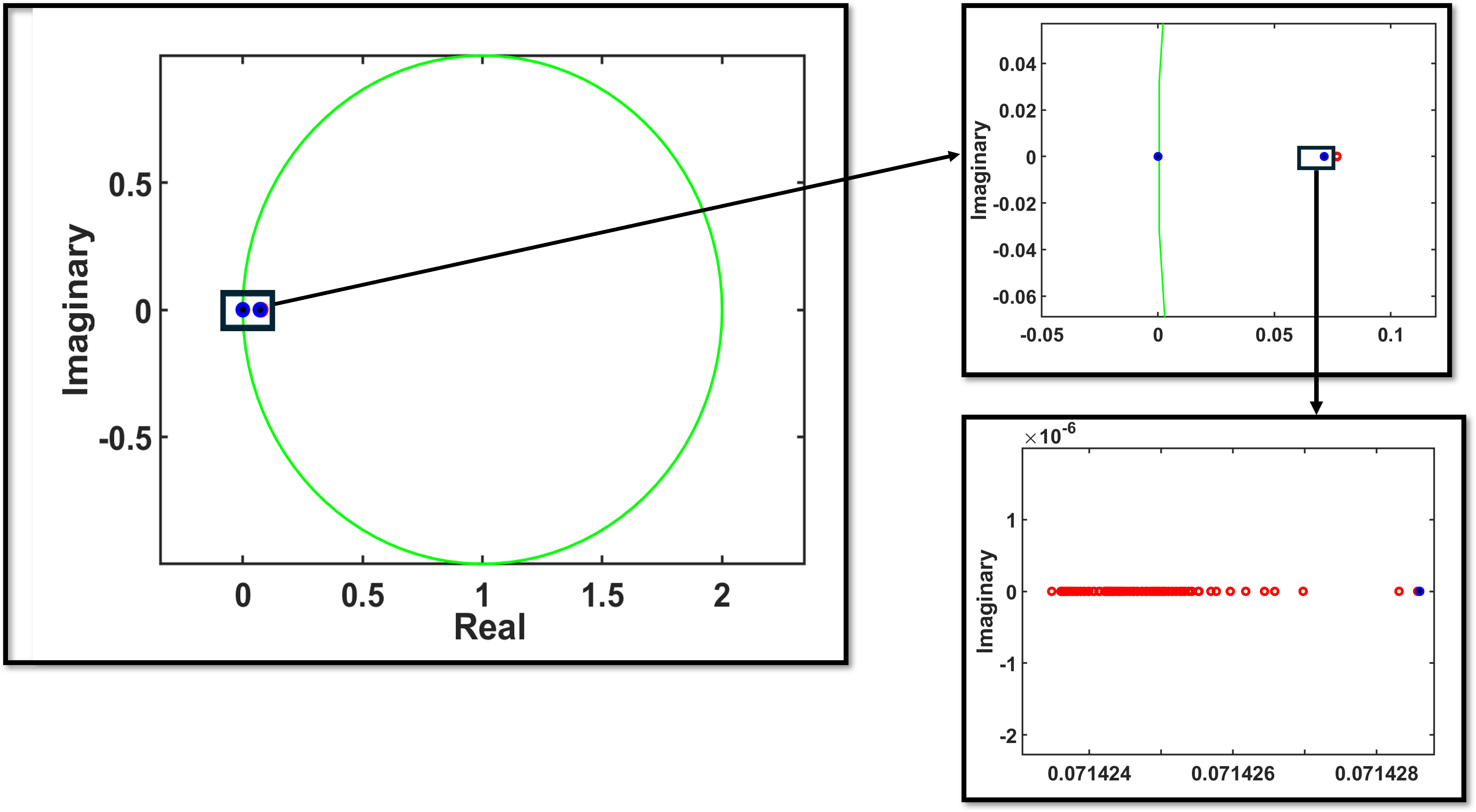}
					\caption{  $\P_{PESS}^{-1}\A$ with $s=13$}
					\label{fig:pess_bounds}
				\end{subfigure}
    \hfil
    \begin{subfigure}[b]{0.53\textwidth}
					\centering
					\includegraphics[width=\textwidth]{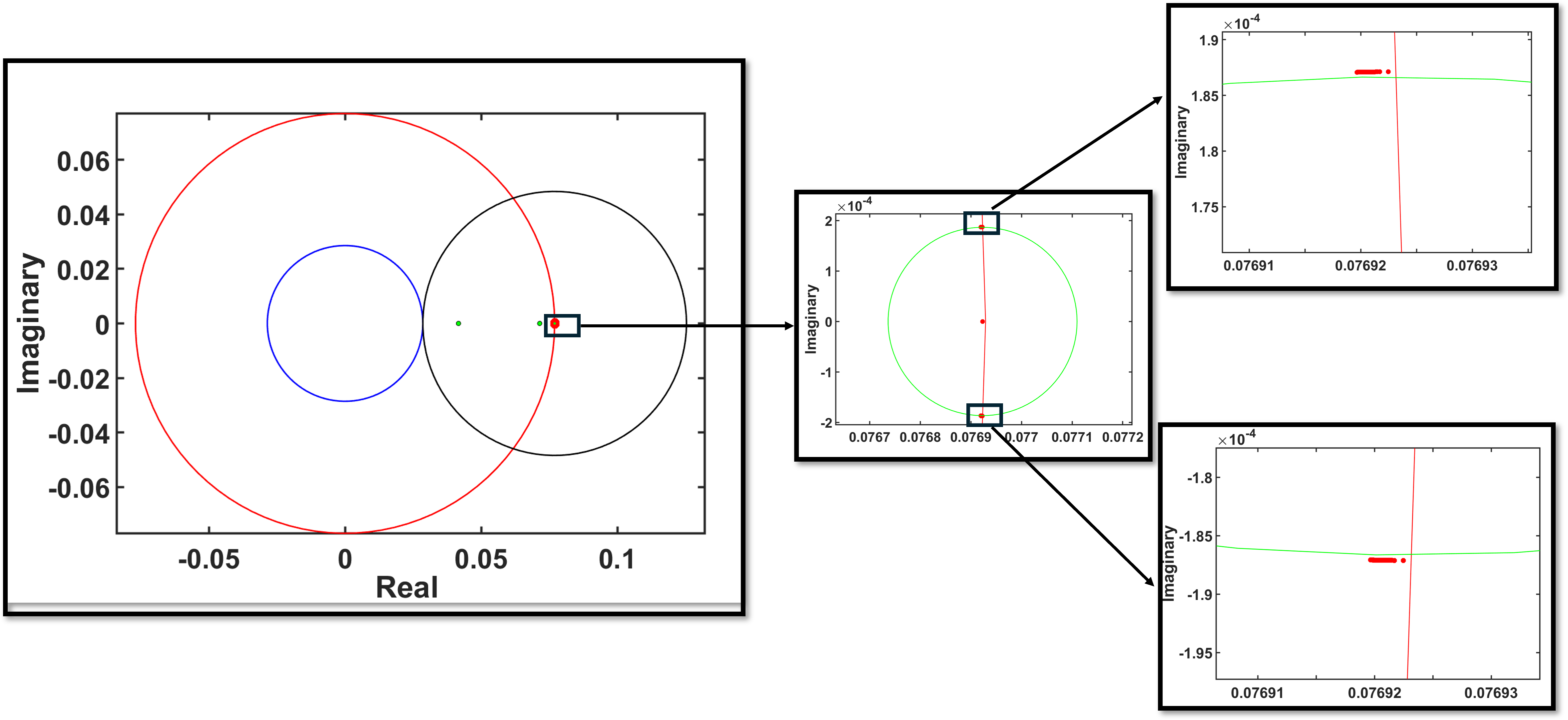}
					\caption{  $\P_{LPESS}^{-1}\A$ with $s=13$}
					\label{fig:Lpess_bounds}
				\end{subfigure}
				\caption{ {Spectral bounds for $\P_{PESS}^{-1}\A$ and $\P_{LPESS}^{-1}\A$ for Case II with $l=16$ for Example \ref{ex1}.}}
				\label{fig:eigenbounds}
			\end{figure}

   %
 \vspace{2mm}
\noindent {\textbf{Spectral bounds:} Furthermore, we draw the eigenvalue bounds of Theorems \ref{th41} and \ref{th42} for the \textit{PESS} preconditioned matrix. The $|\lambda-1|=1$ of Theorem \ref{th41} is drawn by the green unit circle in Figure \ref{fig:eigenbounds}(a) and the points in blue color are the bounds of Theorem \ref{th42}. 
To draw the eigenvalues bounds in Theorem \ref{theorem:RPESS} for \textit{LPESS} preconditioned matrix, we define the following circles:
\begin{align*}
    &C_1:=\left\{\lambda\in \mathbb{C}: |\lambda|=  \sqrt{\frac{\tilde{\theta}_{\max}}{1+s\vartheta_{\min}+s^2\tilde{\theta}_{\max}}}\right\}, ~  C_2:=\left\{\lambda\in \mathbb{C}: |\lambda|= \frac{\vartheta_{\min}}{2+s\vartheta_{\min}} \right\},\\
    & C_3:=\left\{\lambda\in \mathbb{C}: \left|\lambda-\frac{1}{s}\right|= \frac{\vartheta_{\min}}{2+s\vartheta_{\min}} \right\}~\text{and}~ C_4:=\left\{\lambda\in \mathbb{C}: \left|\lambda-\frac{1}{s}\right|= \frac{1}{s(1+s\sqrt{\tilde{\theta}_{\max}})} \right\}.
\end{align*}
 In Figure \ref{fig:eigenbounds}(b), $C_1$ is in red color, $C_2$ is in blue color, $C_3$ is in black color and $C_4$ is in green color. Additionally, the eigenvalue bounds in Theorems \ref{th42}, \ref{th43} and \ref{theorem:RPESS} are also presented in Table \ref{tab:eigbounds_EX1}.}
	  \begin{table}[ht!]
         \centering
        \caption{ Spectral bounds for  $\P_{PESS}^{-1}\A$ and $\P_{LPESS}^{-1}\A$ for case II with $l=16$ for Example \ref{ex1}.}
       \begin{tabular}{c|c|c}
       \hline 
         &  Real eigenvalue $\lambda$& Non-real eigenvalue $\lambda$\\[1ex]
           \hline  
      & Bounds of Theorem \ref{th42} & Bounds of  Theorem \ref{th43} \\[1ex]
      \hline
      \multirow{3}{*}{  $\P_{PESS}^{-1}\A$   }  & contained in  &  $0.0667\leq |\lambda|\leq 0.0739$  \\[0.5ex]
      &  the interval & $4.258\times 10^{-5}\leq \Re(\lambda/(1-s\lambda))\leq 0.5$\\[0.5ex]
      &$(0, 0.071429]$ & $|\Im(\lambda/(1-s\lambda))|\leq 31.6386$\\[1ex]
      \hline 
      & Bounds of Theorem \ref{theorem:RPESS}&  Bounds of Theorem \ref{theorem:RPESS}\\[1ex] 
      \hline
    \multirow{3}{*}{  $\P_{LPESS}^{-1}\A$   }    &contained in & $0.0285\leq |\lambda|\leq 0.07693$ \\[0.5ex]
    & the interval & $ 1.8666\times 10^{-4}\leq |\lambda-\frac{1}{s}|\leq 0.04384$\\[0.5ex]
    &$(0.0416, 0.0769]$ & \\[0.5ex]
    \hline
    
       \end{tabular}
      
       \label{tab:eigbounds_EX1}
   \end{table}
			
		 \vspace{2mm}
\noindent	 { \textbf{Condition number analysis:}} To investigate the robustness of the proposed \textit{PESS} preconditioner, we measure the condition number of the $\P_{PESS}^{-1}\A,$ which is  for any nonsingular matrix $A$ is defined by $\kappa(A):=\|A^{-1}\|_2\|A\|_2.$  In Figure \ref{fig:condition_number}, the influence of the parameter $s$ on the  condition number of  $\P_{PESS}^{-1}\A$ for Case II with $l=32$ is depicted. The parameter  $s$ is chosen from the interval $[5,50]$ with step size one. 

 \vspace{2mm}
\noindent  {\textbf{Discussions:}}			The data presented in Tables \ref{tab2} and \ref{tab:rev1} and Figures \ref{fig1}-\ref{fig:condition_number} allow us to make the following noteworthy observations:
			\begin{itemize}
				\item From Table \ref{tab2}, it is observed that the {\it GMRES} has a very slow convergence speed.  In both Cases I and II, our proposed \textit{PESS} and \textit{LPESS} preconditioners outperform all other compared existing preconditioners from the aspects of IT and CPU times. For example,  in Case I with $l=80,$  our  proposed  \textit{PESS} preconditioner is $76\%$, $43\%$,  {$36\%,$ $22\%,$} $28\%$,  $27\%,$  $33\%$  {and $39\%$} more efficient than the existing {\it BD, IBD,  {MAPSS, SL,} SS, RSS,} {\it EGSS}  and  {\textit{RPGSS}} preconditioner, respectively. Moreover, \textit{LPESS} preconditioners performs approximately $78\%,$ $48\%,$  {$38\%,$ $29\%$,} $35\%,$ $34\%,$  $39\%$ and  { $44\%$} faster than {\it BD, IBD,  {MAPSS, $\P_{SL},$} SS, RSS}, {\it EGSS} and  {\textit{RPGSS}} preconditioner, respectively.  Similar patterns are noticed for $l=16, 32, 48, 64$ and  {$128$}.  
    \item  {Comparing the results of Tables \ref{tab2} and \ref{tab:rev1}, we observe that in both cases, \textit{PESS} and \textit{LPESS} preconditioners outperform the existing baseline preconditioners and IT  are the same as in the experimentally found optimal parameters. This shows that the parameter selection strategy in Section \ref{Sec:parameter} is effective.}
				\item In Figure \ref{fig1}, it is evident that the \textit{PESS} and \textit{LPESS} preconditioners have a faster convergence speed than the other baseline preconditioners when applied to the {\it PGMRES} process for all $l=16,32,48,64, 80$ and  {$128.$}  \begin{figure}[]
				\centering
					\centering
					\includegraphics[width=0.5\textwidth]{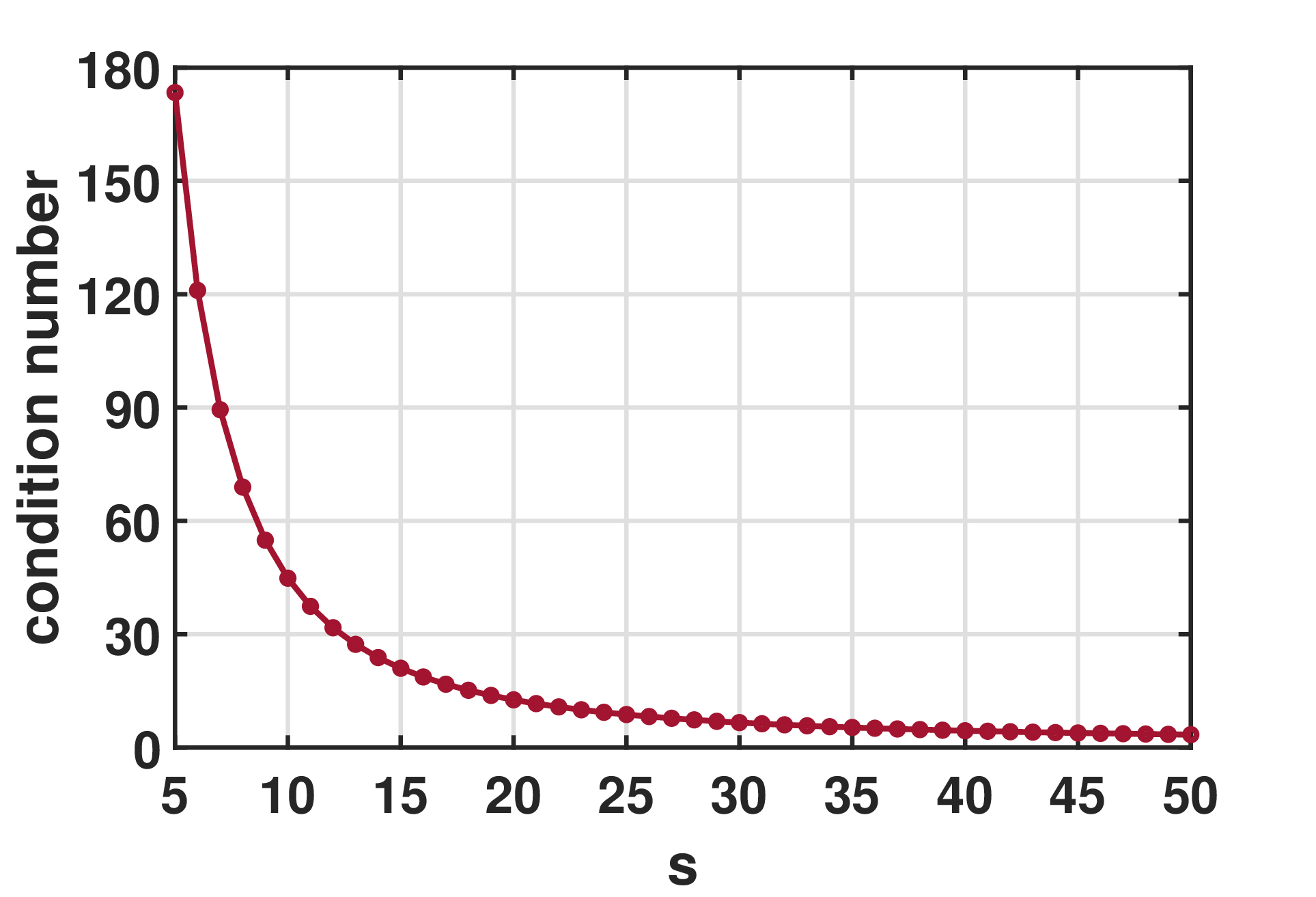}
				\caption{Parameter $s$ vs condition number of the preconditioned matrix $\P_{PESS}^{-1}\A$ for Case II with $l=32$ for Example \ref{ex1}.}
				\label{fig:condition_number}
			\end{figure}
				\item According to Figure \ref{fig2},  the spectrum of $\P_{PESS}^{-1}\A$ and $\P_{LESS}^{-1}\A$ have better clustering properties than the baseline preconditioned matrices, consequently improves the computational efficiency of our proposed \textit{PESS} and \textit{LPESS}  {\it PGMRES} processes.  Moreover, the real eigenvalues of  $\P_{PESS}^{-1}\A$  are contained in the interval $(0,    0.071429]$  and for $\P_{SS}^{-1}\A$ and $\P_{EGSS}^{-1}\A$  are in $(0, 1.9991]$ and $(0,1],$ respectively, which are consistent  with bounds of Theorem \ref{th42}, Corollaries \ref{coro1} and \ref{coro2}, respectively. For non-real eigenvalues of $\P_{PESS}^{-1}\A,$ obtained bounds are consistent with Theorem \ref{th43}.  {Moreover, from Figure \ref{fig:eigenbounds}(b), we observe that eigenvalues of $\P_{LPESS}^{-1}\A$ are contained in $C_2\leq \lambda \leq C_1  \cap C_4\leq \lambda \leq C_4$, which shows the consistency of the bounds in Theorem \ref{theorem:RPESS}.}
				\item 
				For $l=32,$ the computed condition numbers of  $\A, \P_{BD}^{-1}\A$  $ \P_{IBD}^{-1}\A, $  $ {\P_{MAPSS}^{-1}\A}$  {and $\P_{SL}^{-1}\A$} are  $ 5.4289e+ 04,$ $ 9.5567e+ 09,$  $ 9.5124e+ 09,$    { $7.2548e+05$ and $4.2852e+09
,$} respectively, which are very large.  Whereas Figure \ref{fig:condition_number}  {illustrates} a decreasing trend in the condition number of $\P_{PESS}^{-1}\A$ (for instance, when $s=50$, $\kappa(\P_{PESS}^{-1}\A)= 3.4221$) with increasing values of  $s.$ This observation highlights that  $\P_{PESS}^{-1}\A$ emerges as well-conditioned, consequently, the solution obtained by the proposed \textit{PESS} {\it PGMRES} process is more reliable and robust.
			\end{itemize}
   The above discussions affirm that the proposed \textit{PESS} and \textit{LPESS} preconditioner is efficient, robust and better well-conditioned with respect to the baseline preconditioners.
   %
   
		\begin{figure}[ht!]
			\centering
				\centering
				\includegraphics[width=0.45\textwidth]{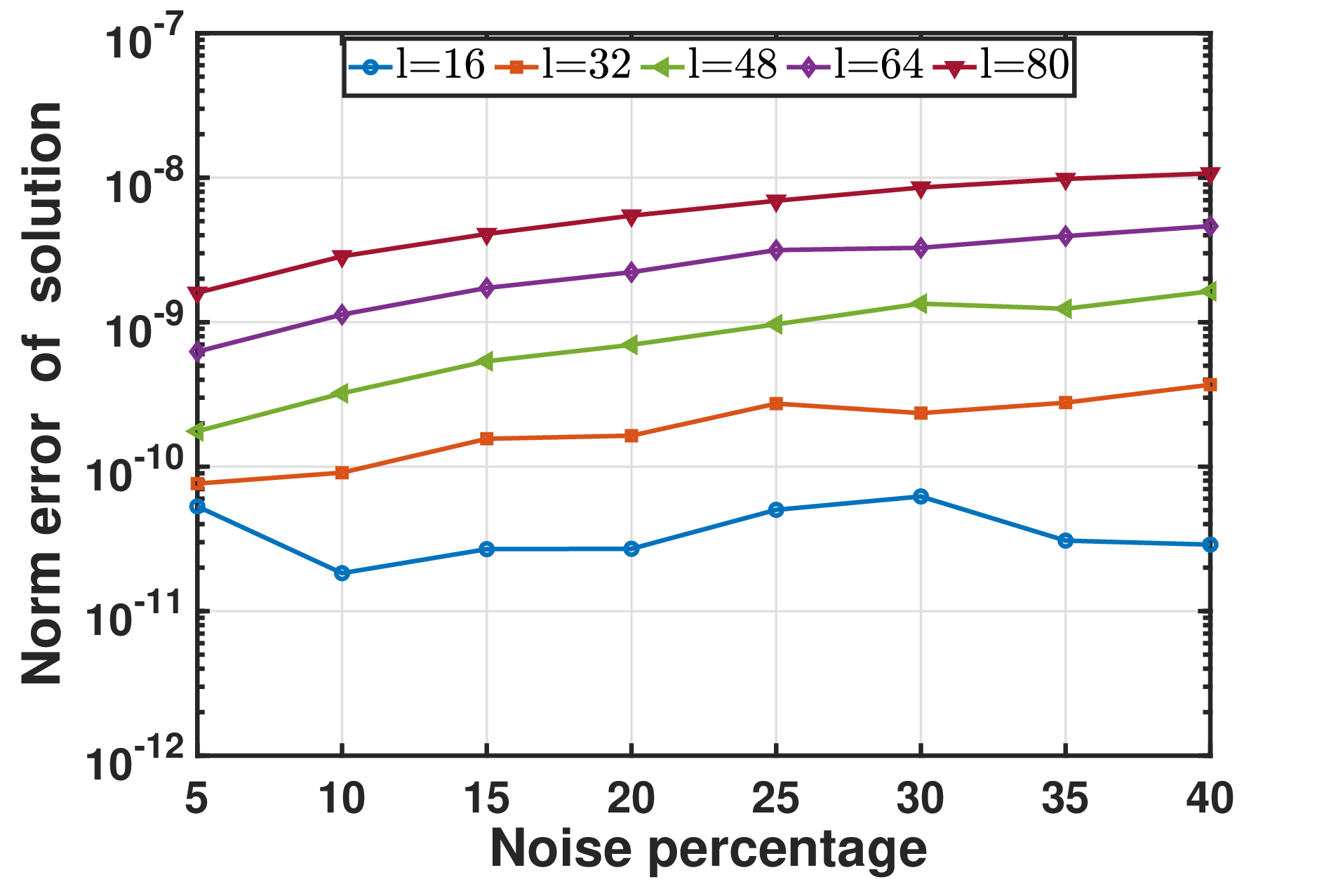}
				\label{fig:noise}
			\caption{Relationship of norm error of solution with increasing 
				noise percentage, employing proposed $\P_{PESS}$ for Case II with $s=12$ for $l=16,32,48,64$ and $80$ for Example \ref{ex1}.}
			\label{fig:noise_error}
		\end{figure}
		
 \vspace{2mm}
\noindent\textbf{Sensitivity analysis of the solution by employing the proposed \textit{PESS} preconditioner:}		To study the sensitivity of the solution obtained by employing the proposed \textit{PESS} {\it PGMRES} process to small perturbations on the input data matrices. In the following, we consider the perturbed counterpart 
		$(\A+\Delta \A )\tilde{\textbf{u}}=\textbf{d}$ of the three-by-three block \textit{SPP} \eqref{eq11}, where the perturbation matrix $\Delta \A $ has same structure as of $\A.$ We construct the perturbation matrix $\Delta \A$ by adding noise to the block matrices $B$ and $C$ of $\A$ as follows: $$\Delta B=10^{-4} * N_P * std(B).*randn(m,n)\quad\text{and} \quad \Delta C=  10^{-4} * N_P * std(C).*randn(p, m),$$
		where  $N_P$ is the noise percentage and $std(B)$ and $std(C)$ are the standard deviation of $B$  and $C,$ respectively. Here, $randn(m, n)$ is an $m$-by-$n$ normally distributed random matrix generated by the MATLAB command $randn$. We perform the numerical test for $l=16,32,48,64$ and $80$ for Example \ref{ex1} using the {\it PESS PGMRES} process (Case II, $s=12$). The calculated norm error $\|\Tilde{\textbf{u}}-\textbf{u}\|_2$ among the solution of the perturbed system and the original system with increasing $N_P$ from  $5\%$ to $40\%$ with step size $5\%$ are displayed in the  {Figure} \ref{fig:noise_error}. We noticed that with growing $N_P,$ the norm error of the solution remains consistently less than $10^{-8}, $ which demonstrates the robustness of the proposed \textit{PESS} preconditioner and the solution $\textbf{u}$ is insensitive to small perturbation on $\A$.
  \end{exam}
		\begin{exam}\label{ex2}
  	\begin{table}[ht!]
				\centering
				\caption{Numerical results of  {\it GMRES}, {\it  IBD,  {\textit{MAPSS}, SL,}  SS, RSS, EGSS,  {RPGSS,} PESS} and {\it LPESS PGMRES} processes for Example \ref{ex2}.}
				\label{tab3}
				\resizebox{10.5cm}{!}{
					\begin{tabular}{ccccccc}
						\toprule
						Process& $\textbf{h}$	& $1/8$ &$1/16$& $1/32$ & $1/64$ &   {$1/128$}\\
						\midrule 
						& size($\A$)	& $288$ &$1088$& $4224$ & $16640$&  {$66048$} \\
						\midrule
						\multirow{3}{*}{{\it GMRES}}& IT& $158$&$548$&$2048$&$\bf{--}$&  {$\bf{--}$} \\
						& CPU&  $0.2325$&$0.6960$&$124.7595$&$4545.5072$ &  {$9396.6958$}\\
						&{\tt RES}&  $8.0113e-07$ & $8.8023e-07$ & $9.0887e-07$&$ 4.0481e-06$&  {$8.3021e-06$}\\	
						\midrule
						\multirow{3}{*}{\it IBD}& IT& $34$& $35$& $25$& $23$ &  {$21$}\\
						& CPU&$\bf 0.0147$&$\bf 0.1166$&$1.9070$&$32.6213$ &  {$1664.6991$}\\
						&{\tt RES}& $9.4111e-07$&$6.9932e-07$&$9.2544e-07$&$9.3340e-07$&  {$7.3550e-07$}\\
       \midrule
					\multirow{3}{*}{ {\it MAPSS}}&  {IT}&  {$11$} & {$7$} & {$4$} &  {$4$} &  {$4$}  \\
					&  {CPU}&  {$0.2061$}&
 {$0.3453$}&
 {$1.8717$}&
 {$26.6600$}&
 {$482.0390$} \\
					&{ {\tt RES}}& {$5.2642e-08$}&
 {$1.1366e-07$}&
 {$7.1116e-08$}&
 {$8.2780e-07$}&
 {$2.1189e-09$}  \\
						\midrule
      \multirow{3}{*}{ {\textit{SL}}}&   {IT}&  {$6$} & {$6$} & {$5$} &  {$5$} &  {$4$}  \\
					&  {CPU}&
 {$0.2130$}&
 {$0.2929$}&
 {$1.8697$}&
 {$27.1124$} &  {$516.3012$} \\
					&{ {\tt RES}}& {$6.6518e-07$}&
 {$1.1953e-08$}&
 {$9.5638e-08$}&
 {$2.1528e-11$}&
 {$2.9479e-10$}  \\
						\midrule
						& & &\textbf{ Case I}&&  \\
						\midrule
						\multirow{3}{*}{\it SS}& IT& $5$&$6$&$9$&$11$&  {$11$}
\\
						& CPU&$0.3047$&
$0.4422$&
$2.4279$&
$51.5825$&  {$1245.2019$}\\     &{\tt RES}&$4.5523e-07$&
$7.9331e-07$&
$2.5307e-07$&
$4.5327e-07$&  {$9.8497e-07$} \\
						\midrule
						\multirow{3}{*}{\it RSS}& IT& $\bf 4$& $\bf 4$& $5$& $5$ &  {$6$} \\
						&CPU&$0.2044$&$0.3436$&$1.6047$&$27.6128$&  {$720.4234$} \\
						&{\tt RES}&$3.4853e-08$&
$9.2155e-07$&
$1.7469e-08$&
$2.7261e-07$ &  {$3.2621e-09$} \\
						\midrule
						\multirow{3}{*}{\it EGSS}& IT& $7$&$9$&$11$&$15$&  {$5$}\\
						& CPU&$0.2389$&$0.4051$&$4.4353$&$52.7804$ &  {$744.5983$} \\
						&{\tt RES}&$1.6861e-07$&$4.0517e-08$&$1.3223e-07$&$1.774e-07$&  {$9.1207e-07$}
 \\
   \midrule
					\multirow{3}{*}{ {\it RPGSS}}& { IT}&  {$4$} & {$4$} & {$4$} &  {$5$} &  {$5$}  \\
					&  {CPU}&  {$0.2105$}&
 {$0.3883$}&
 {$2.2981$}&
 {$28.1748$}&
 {$544.8828$} \\
					&{ {\tt RES}}& {$2.8700e-09$}&
 {$5.1131e-08$}&
 {$3.3646e-07$}&
 {$8.8663e-09$}&
 {$1.4981e-07$}  \\
						\midrule
						\multirow{3}{*}{\textit{PESS}$^{\dagger}$}& IT& $4$& $4$& $5$& $5$ & {$4$}\\
						& CPU&$0.2068$&
$0.3480$&
$1.7633$&
$25.2920$&  {$484.2796$}
 \\
						$s=30$ &{\tt RES}&$2.3499e-08$&
$4.2761e-07$&
$2.1247e-08$&
$2.0987e-07 $&  {$4.3859e-07$} \\
      \midrule
						\multirow{3}{*}{\textit{LPESS}$^{\dagger}$}& IT& $\bf 3$& $\bf{3}$& $\textbf{4}$& $\textbf{3}$ & $ {4}$\\
						& CPU&$ \bf 0.1990$&
$\bf 0.2578$&
$\bf 1.4778$&
$\bf 17.9730 $&  {$457.7446$}\\
						$s=30$ &{\tt RES}&$1.6606e-07$&
$8.3803e-07$&
$4.9289e-10$&
$5.2658e-07$&  {$3.3191e-08$}\\
						\midrule
						& &&  \textbf{Case II}& & \\
						\midrule
						\multirow{3}{*}{\textit{SS}}& IT& $29$&
$38$&
$60$&
$52$&  {$55$}\\
						& CPU&$0.2999$&
$1.1144$&
$13.4580$&
$213.2381$&  {$4980.7387$} \\
						&{\tt RES}&$ 7.3560e-07$&
$9.5790e-07$&
$9.2515e-07$&
$9.6425e-07$&  {$9.6948e-07$}\\
						\midrule
						\multirow{3}{*}{\textit{RSS}}& IT& $12$& $13$& $14$& $12$&  {$12$}\\
						& CPU&$0.3797$&
$0.6019$&
$3.4491$&
$50.3476$&  {$1361.8902$}\\
						&{\tt RES}&$3.2456e-07$&
$5.3086e-07$&
$6.6600e-07$&
$4.2834e-07$&  {$1.2733e-07$}\\
						\midrule
						\multirow{3}{*}{\textit{EGSS}}& IT& $9$& $9$& $9$& $7$&  {$4$} \\
						& CPU&$0.2459$&$0.4180$&$2.4320$&$27.9359 $&  {$419.1294$}\\
						&{\tt RES}&$7.9001e-09$&$8.6623e-09$&$7.4644e-09$&$1.2457e-07$&  {$7.6469e-08$}\\
      \midrule
					\multirow{3}{*}{ {\it RPGSS}}& IT&  {$6$} & {$7$} & {$7$} &  {$5$} &  {$4$}  \\
					&  {CPU}&  {$0.2154$}&
 {$0.4096$}&
 {$1.8893$}&
 {$23.2802$}&
 {$423.0121$} \\
					&{ {\tt RES}}& {$5.2642e-08$}&
 {$1.1366e-07$}&
 {$7.1116e-08$}&
 {$8.2780e-07$}&
 {$2.1189e-09$}  \\
						\midrule
						\multirow{3}{*}{\textit{PESS}$^{\dagger}$}& IT& $5$& $5$& $5$& $5$&  { $\bf 3$}\\
						& CPU&$0.2091$&
$0.4184$&
$1.8118$&
$25.8048$&  { $\bf 365.3128$}
 \\
						$s=26$ &{\tt RES}&$2.3801e-08$&
$1.1965e-08$&
$8.4923e-09$&
$5.9356e-08 $&  {$4.2039e-07$} \\
      \midrule
						\multirow{3}{*}{\textit{LPESS}$^{\dagger}$}& IT& $4$& $5$& $\bf 5$& $\bf 5$&  { $\bf 3$}\\
						& CPU&$0.2070$&
$0.3662$&
$\bf 1.5330$&
$\bf 23.5997$&  { $ \bf 359.1330$}\\
						$s=26$ &{\tt RES}&$2.4250e-07$&
$8.6182e-12$&
$5.3930e-11$&
$8.1253e-10$& {$5.8345e-07$} \\
						\bottomrule
						\multicolumn{7}{l}{ Here ${\dagger}$ represents the proposed preconditioners. The boldface represents the top two results.}\\
						\multicolumn{7}{l}{ $\bf{--}$ indicates that the iteration process does not converge within the prescribed IT.}
					\end{tabular}
				}
			\end{table} 
		 {\textbf{Problem formulation:}}	In this example, we consider the three-by-three block {\it SPP} \eqref{eq11}, where block matrices $A$ and $B$ originate from the two dimensional Stokes equation namely ``leaky" lid-driven cavity problem \cite{ Matlab2007}, in a square domain $\Xi=\{(x,y) ~|~ 0\leq x\leq 1, 0 \leq y\leq 1\}$, which is defined as follows:
			\begin{align}\label{eq51}
				& -\Delta {\bf u}+\nabla p=0, \quad \text{in} \quad \Xi,\\ \nonumber
				&\quad \quad \quad\nabla \cdot {\bf u}=0,  \quad \text{in} \quad \Xi.
   			\end{align}
			A Dirichlet no-flow condition is applied on the side and bottom boundaries, and the nonzero horizontal velocity on the lid is $\{y=1;-1\leq x\leq 1| {\bf u}_x= 1\}. $  In this context, $\Delta$ symbolizes the Laplacian operator in $\mathbb{R}^2$, $\nabla$ represents the gradient, $\nabla \cdot$ denotes the divergence. ${\bf u}$ and $p$ refer to the velocity vector field and the pressure scalar field, respectively. 
					  
\begin{table}[ht!]
       \centering
         \caption{  Numerical results of \textit{PESS-I}, \textit{LPESS-I}, \textit{PESS-II} and \textit{LPESS-II} \textit{PGMRES} processes for Example \ref{ex2}.}
          {\resizebox{13cm}{!}{
         \begin{tabular}{ccccccc}
						\toprule
						Process& $\textbf{h}$	& $1/8$ &$1/16$& $1/32$ & $1/64$ &   {$1/128$}\\
						\midrule 
						& size($\A$)	& $288$ &$1088$& $4224$ & $16640$&  {$66048$} \\
                            \midrule
						\multirow{1}{*}{{\it PESS-I}}& IT& $5$&$5$&$6$&$6$& $5$\\
					($s=1,$ $\Lambda_1=0.001I,$	& CPU&  $0.30867$&
$0.46688$&$2.48254$&$46.52887$ & {$577.9420$}\\
				$\Lambda_2=0.1I,$ $\Lambda_3=0.001I$)		&{\tt RES}& $1.4125e-08$&
$1.3379e-07$&
$3.8821e-08$&
$2.6061e-07$&
$9.7383e-07$\\		 
      \midrule
      \multirow{1}{*}{{\it LPESS-I}}& IT& $4$&$4$&$4$&$5$& $3$\\
      			($s=1,$ $\Lambda_2=0.1I,$	& CPU&  $0.1829$&$0.3319$&$1.2774$&$19.1090$ & {$349.6728$}\\
       	  $\Lambda_3=0.001I$)   &{\tt RES}&  $3.5507e-09$ & $3.4257e-08$ & $7.0375e-07$&$  2.3073e-09$& $ 7.1453e-07$\\
      \midrule
      \multirow{1}{*}{{\it PESS-II}}& IT& $3$&$4$&$4$&$4$& $4$\\
				($s=s_{est},$ $\Lambda_1=A,$		& CPU&  $0.18470$&$0.39536$&$1.32148$&$17.22840$&$485.01767$\\
$\Lambda_2=\beta_{est}I,$  $\Lambda_3= 10^{-4}CC^T$) &{\tt RES}& $ 9.6100e-07$&$5.6520e-07$&$7.8750e-08$&$3.9903e-08$&$6.9694e-09$\\
      \midrule
      \multirow{1}{*}{{\it LPESS-II}}& IT& $6$&$6$&$5$&$5$& $3$\\
					($s=s_{est},$ $\Lambda_2=\beta_{est}I,$	& CPU&  $0.1656$&
$0.3777$&
$1.4341$&
$19.0658$&
$377.2001$\\ 
  $\Lambda_3= 10^{-4}CC^T$)    &{\tt RES}& $2.5227e-07$ &$5.1622e-07$&$4.8924e-08$&$3.2751e-08$&$7.9950e-07$\\
      \midrule
       \end{tabular}}}
       \label{tab:rev2}
   \end{table}
To generate the matrices $A\in \R^{n\times n}$ and $B\in \R^{m\times n}$ of the system \eqref{eq11}, the discretization task of the Stokes equation \eqref{eq51} is accomplished by  the IFISS software developed by \citet{Matlab2007}. Following \cite{Matlab2007} 
    with grid parameters $\textbf{h}=1/8, 1/16, 1/32, 1/64,  {1/128}$,  {we get} $n= 2(1+ 1/\bf{h})^2$ and $m=1/\textbf{h}^2.$
			To make the system of equations in \eqref{eq11} not too ill-conditioned and not too sparse, we construct the block matrix $C= \bmatrix{\Pi & randn(p, m- p)}$,
			where $\Pi= \diag(1, 3, 5,\ldots, 2p- 1)$  and $p= m- 2.$ Additionally, to ensure the positive definiteness of the matrix $A$, we  add $0.001I$ with $A.$ 

	\begin{figure}[ht!]
	\centering
	\begin{subfigure}[b]{0.3\textwidth}
		\centering
		\includegraphics[width=\textwidth]{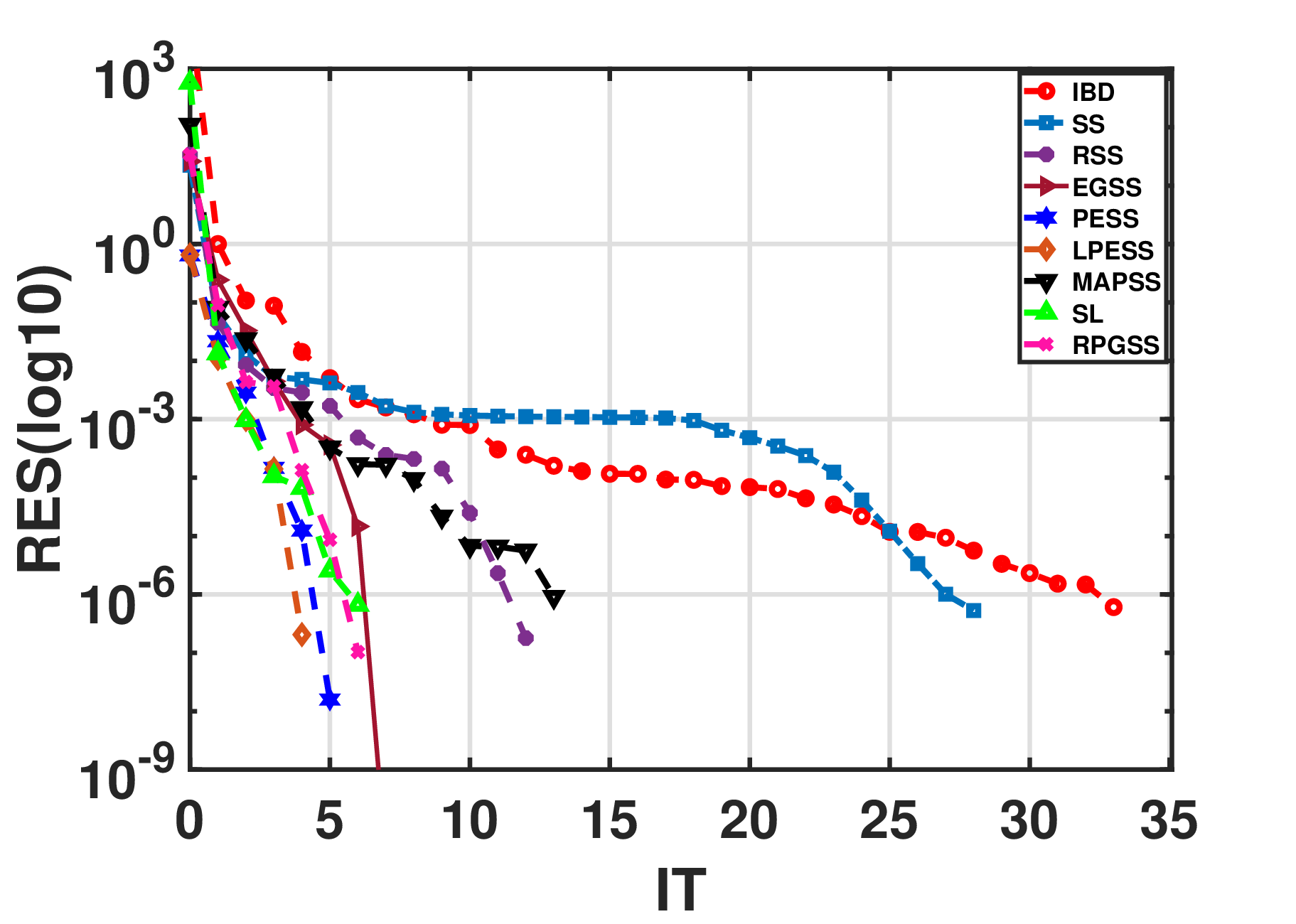}
		\caption{\footnotesize  $\textbf{h}=1/8$ }
		\label{figstokes8}
	\end{subfigure}
	\begin{subfigure}[b]{0.3\textwidth}
		\centering
		\includegraphics[width=\textwidth]{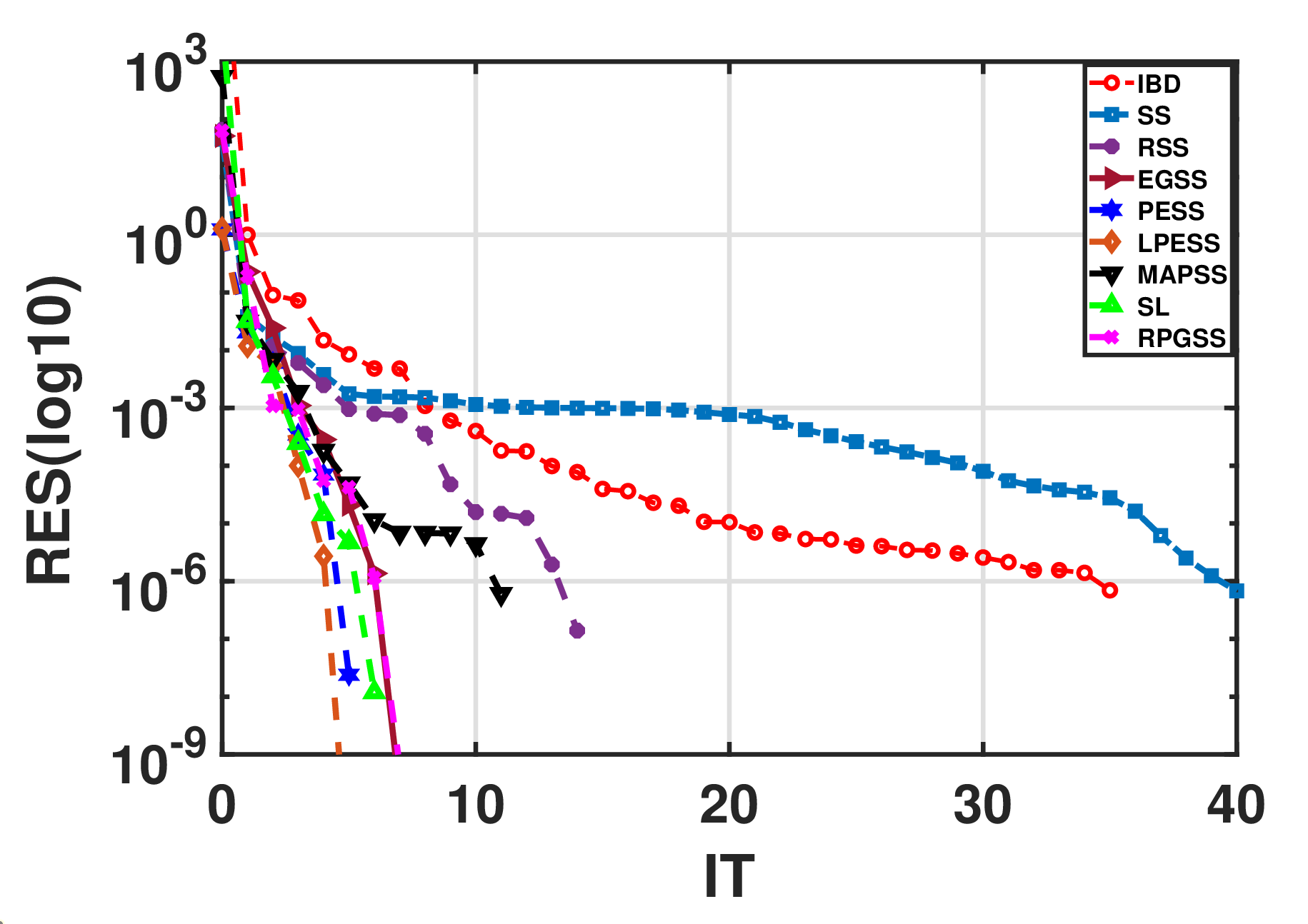}
		\caption{\footnotesize  $\textbf{h}=1/16$}
		\label{figstokes16}
	\end{subfigure}
	\begin{subfigure}[b]{0.3\textwidth}
		\centering
		\includegraphics[width=\textwidth]{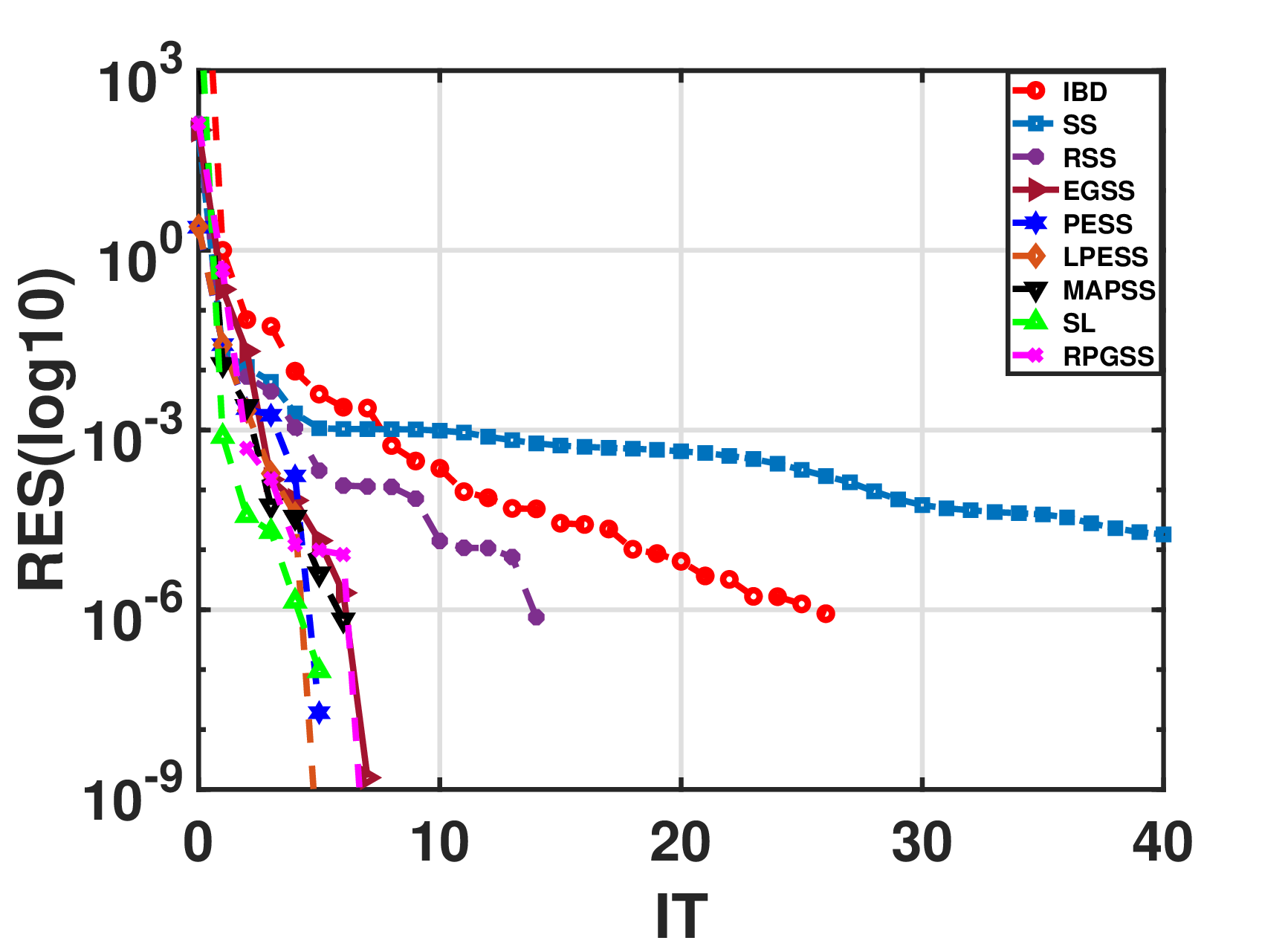}
		\caption{\footnotesize  $\textbf{h}= 1/32$}
		\label{figstokes32}
	\end{subfigure}
	\begin{subfigure}[b]{0.3\textwidth}
		\centering
		\includegraphics[width=\textwidth]{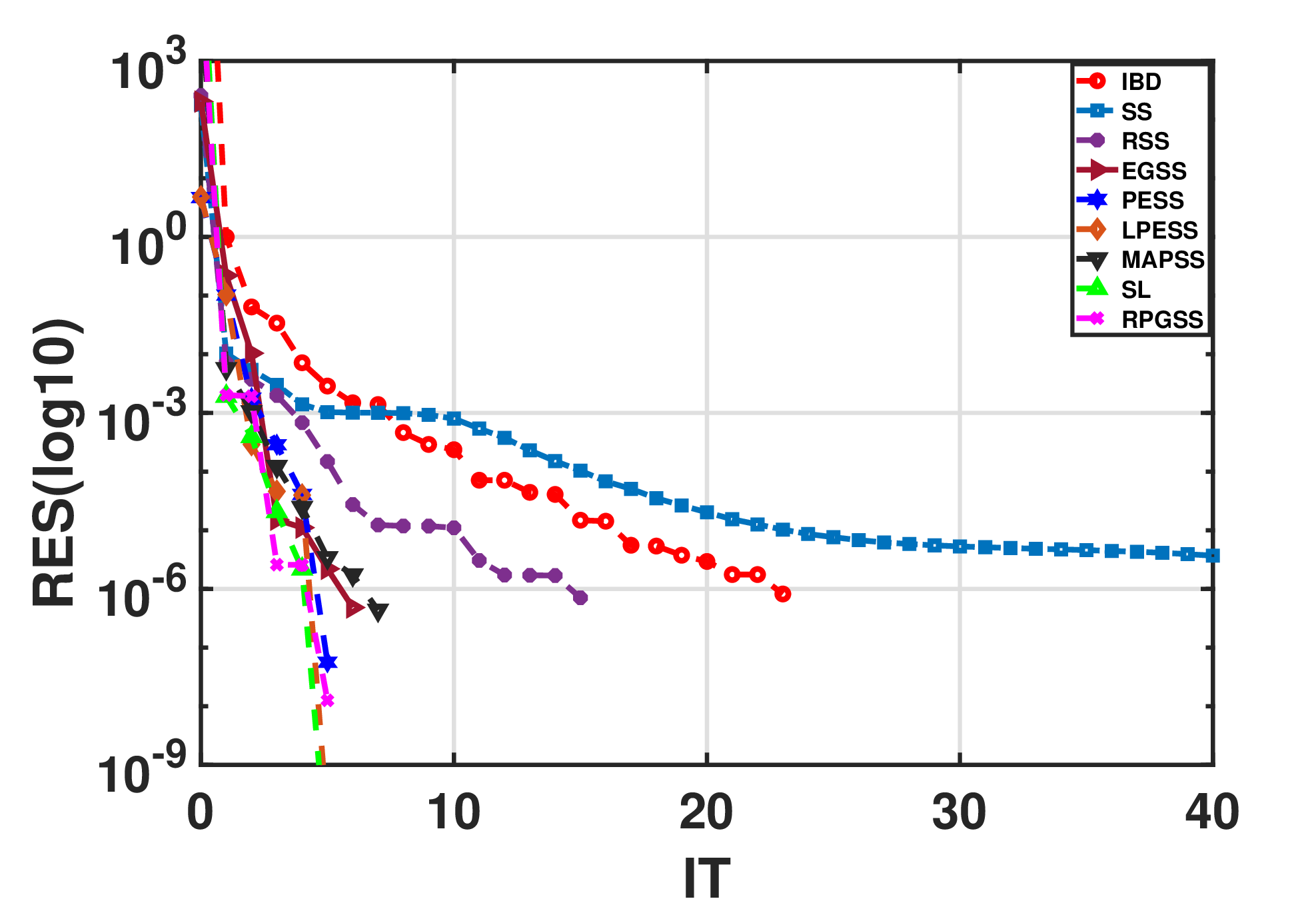}
		\caption{\footnotesize  $\textbf{h}= 1/64$}
		\label{figstokes64}
	\end{subfigure}
 \begin{subfigure}[b]{0.3\textwidth}
		\centering
		\includegraphics[width=\textwidth]{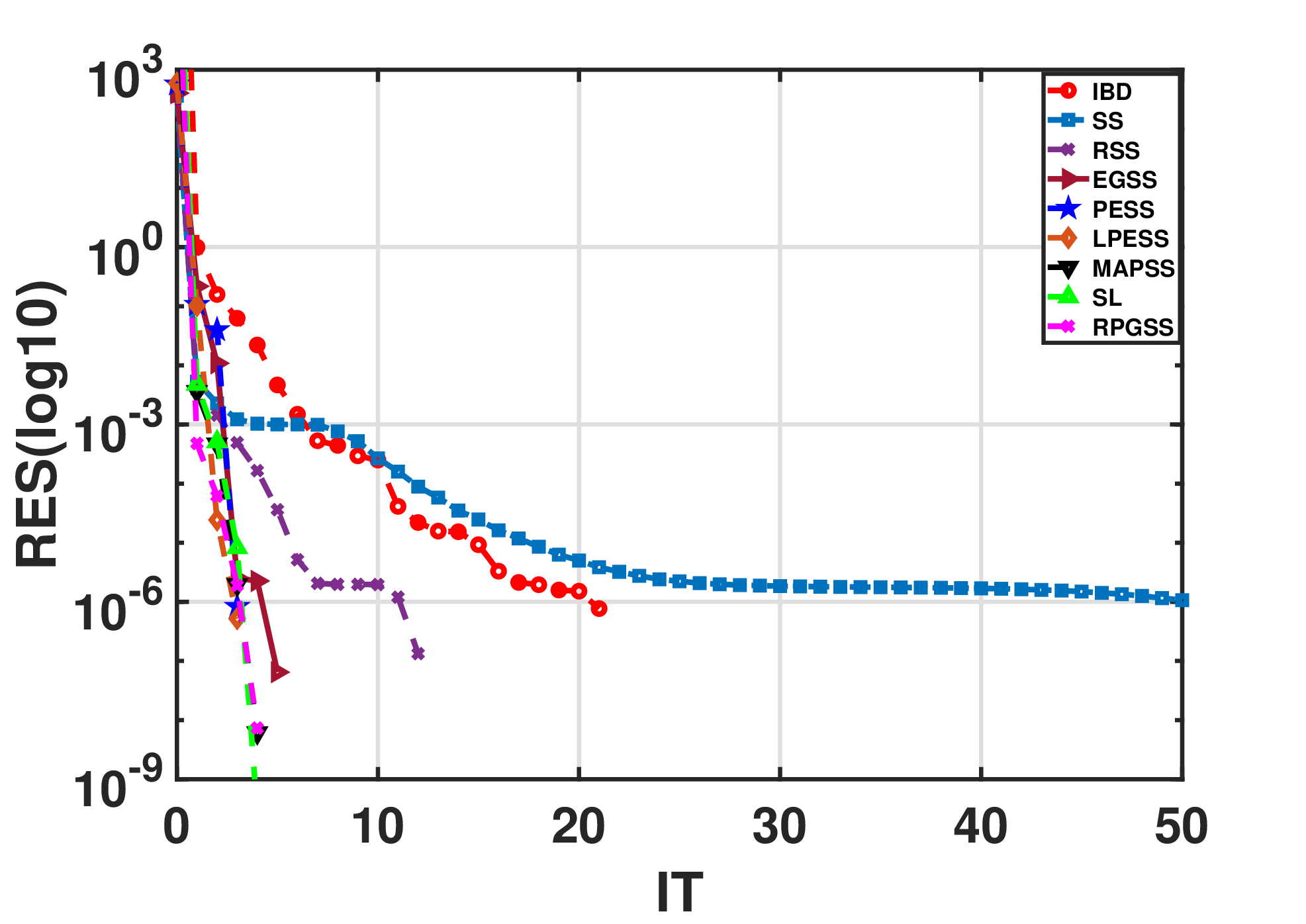}
		\caption{\footnotesize    $\textbf{h}= 1/128$}
		\label{figstokes128}
	\end{subfigure}
	\caption{Convergence curves for IT versus RES of the {\it PGMRES} processes by employing  {\it IBD,  {MAPSS, SL,} SS, RSS,  EGSS,  {RPGSS,} PESS} \text{and} \textit{LPESS}  preconditioners in Case II for Example \ref{ex2}.}
	\label{fig4}
\end{figure} 

   %
				\begin{figure}[ht!]
				\centering
				\begin{subfigure}[b]{0.3\textwidth}
					\centering    
                        \includegraphics[width=\textwidth]{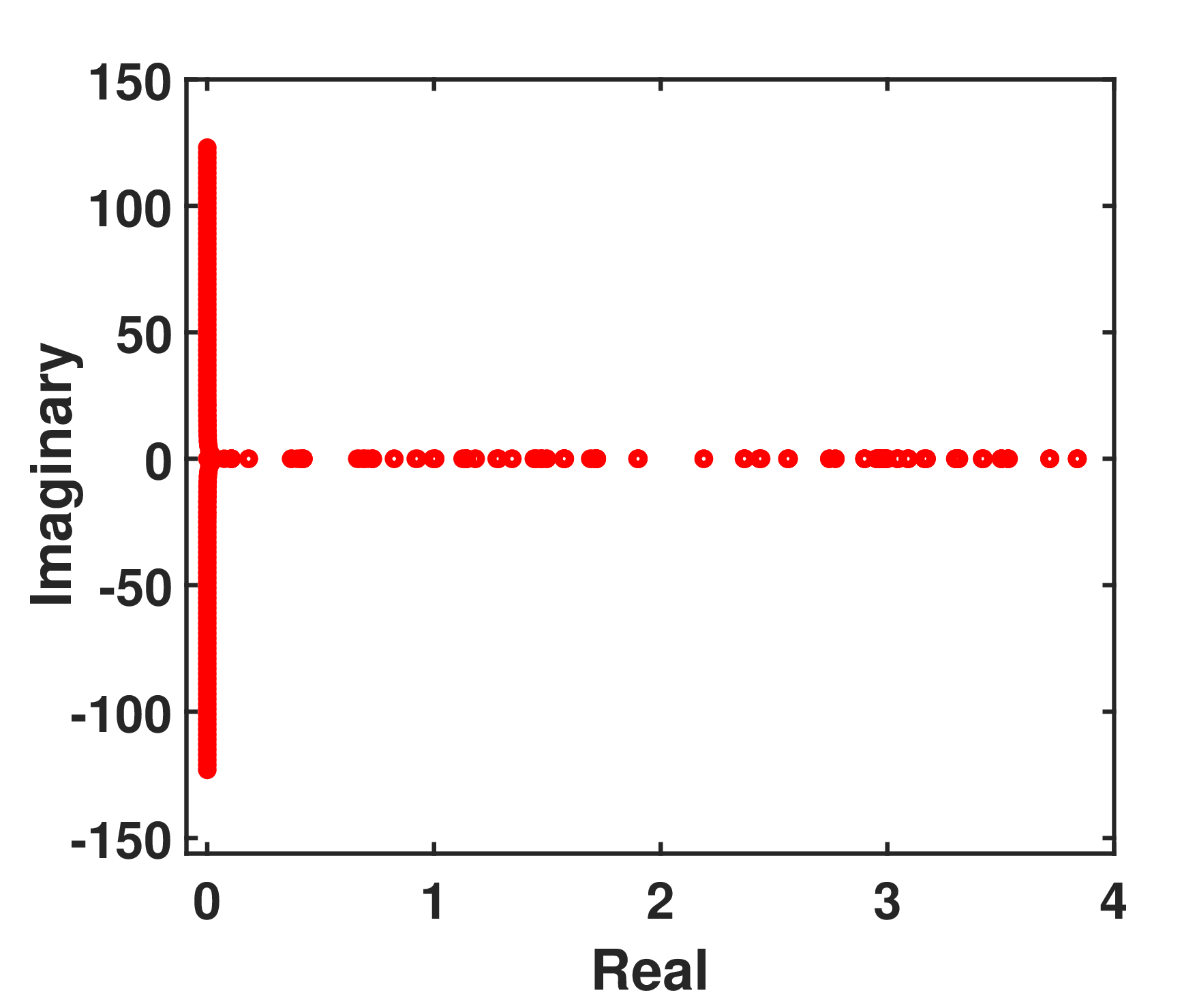}
					\caption{\footnotesize  $\A$}
					\label{original_stokes}
				\end{subfigure}
				\begin{subfigure}[b]{0.3\textwidth}
					\centering
					\includegraphics[width=\textwidth]{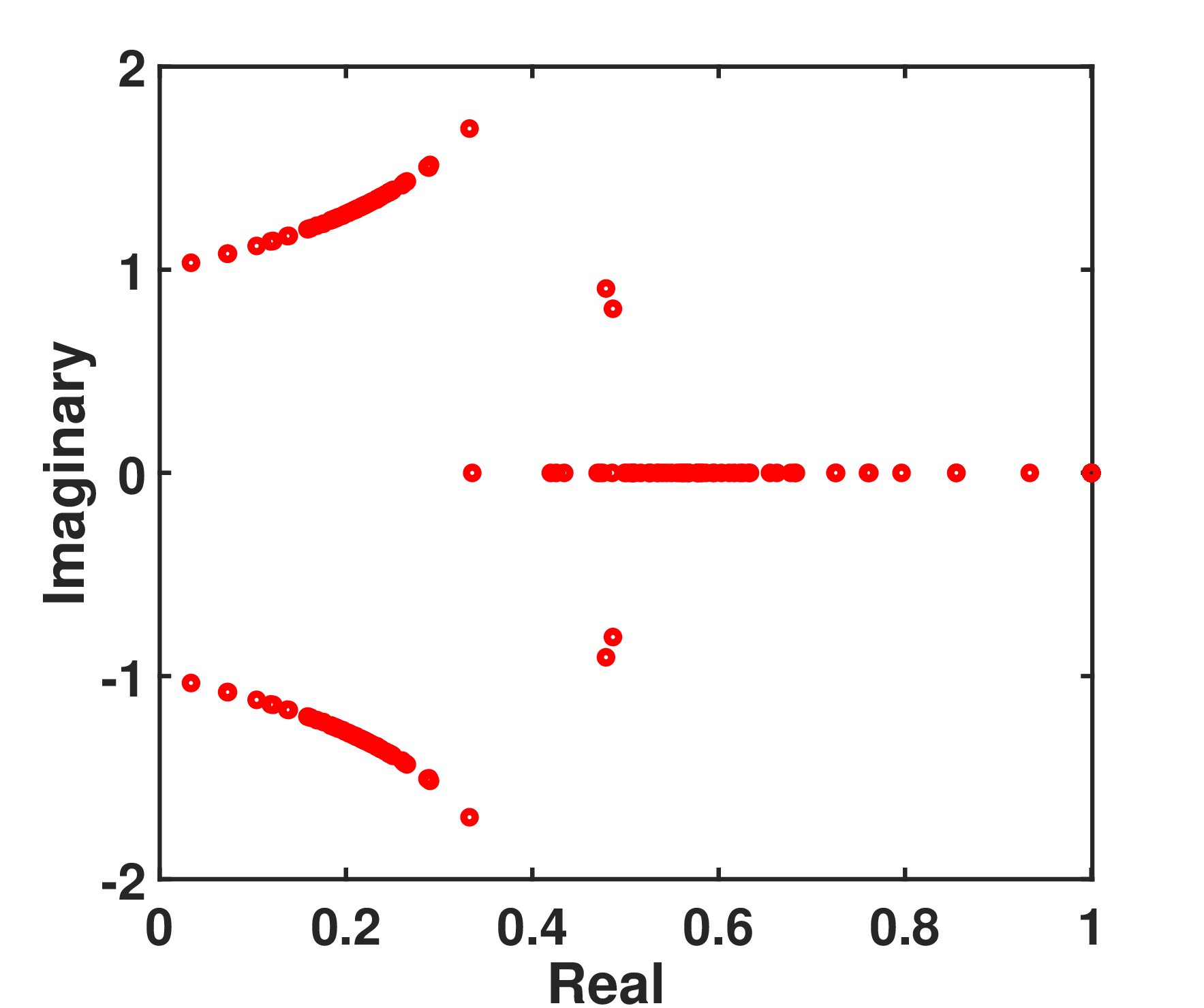}
					\caption{\footnotesize $\P_{IBD}^{-1}\A.$ }
					\label{IBD_STOKES}
				\end{subfigure}
    \begin{subfigure}[b]{0.3\textwidth}
					\centering
					\includegraphics[width=\textwidth]{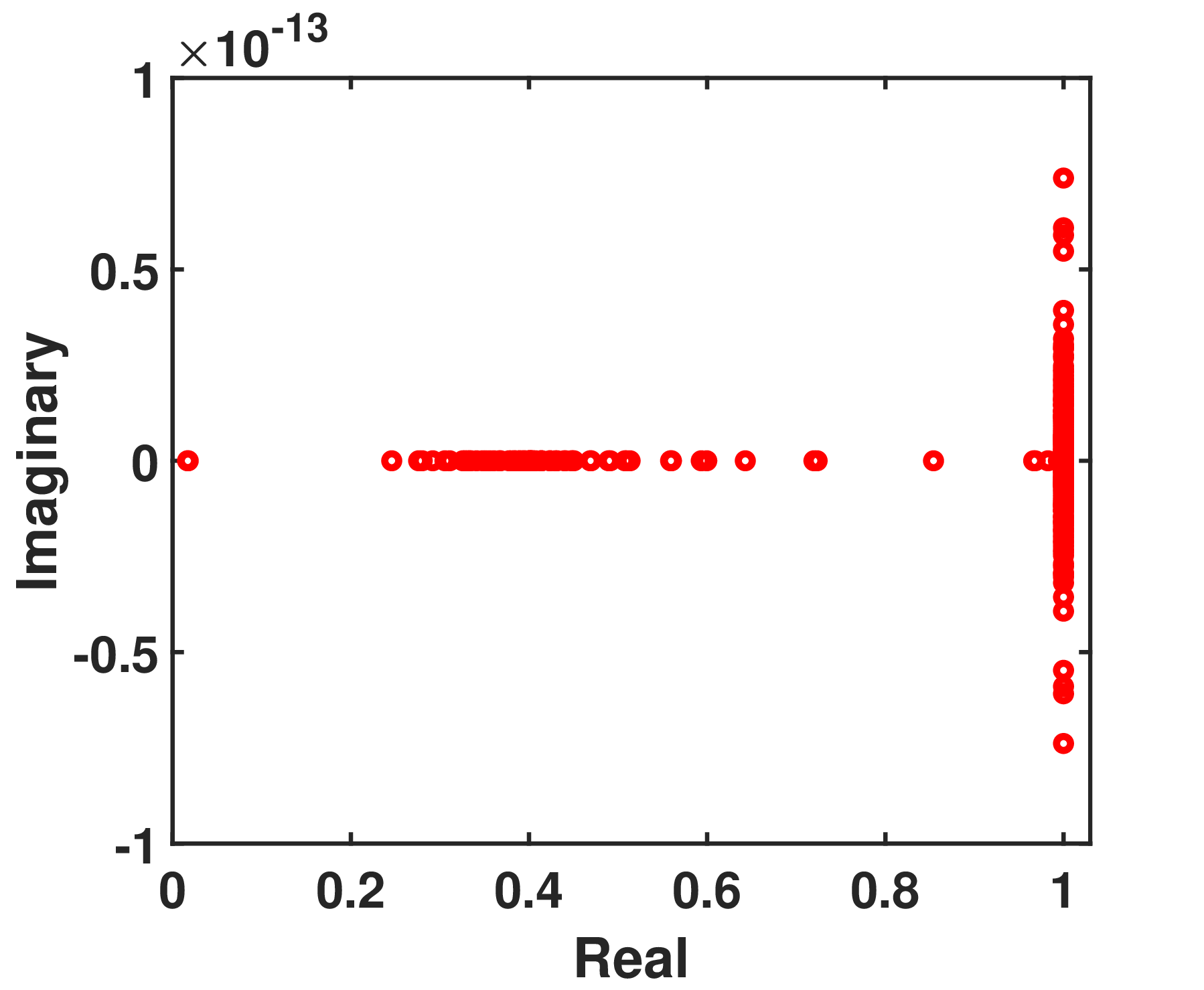}
					\caption{\footnotesize 
   $\P_{MAPSS}^{-1}\A$ }
					\label{MAPSS_STOKES}
				\end{subfigure}
     \begin{subfigure}[b]{0.3\textwidth}
					\centering
					\includegraphics[width=\textwidth]{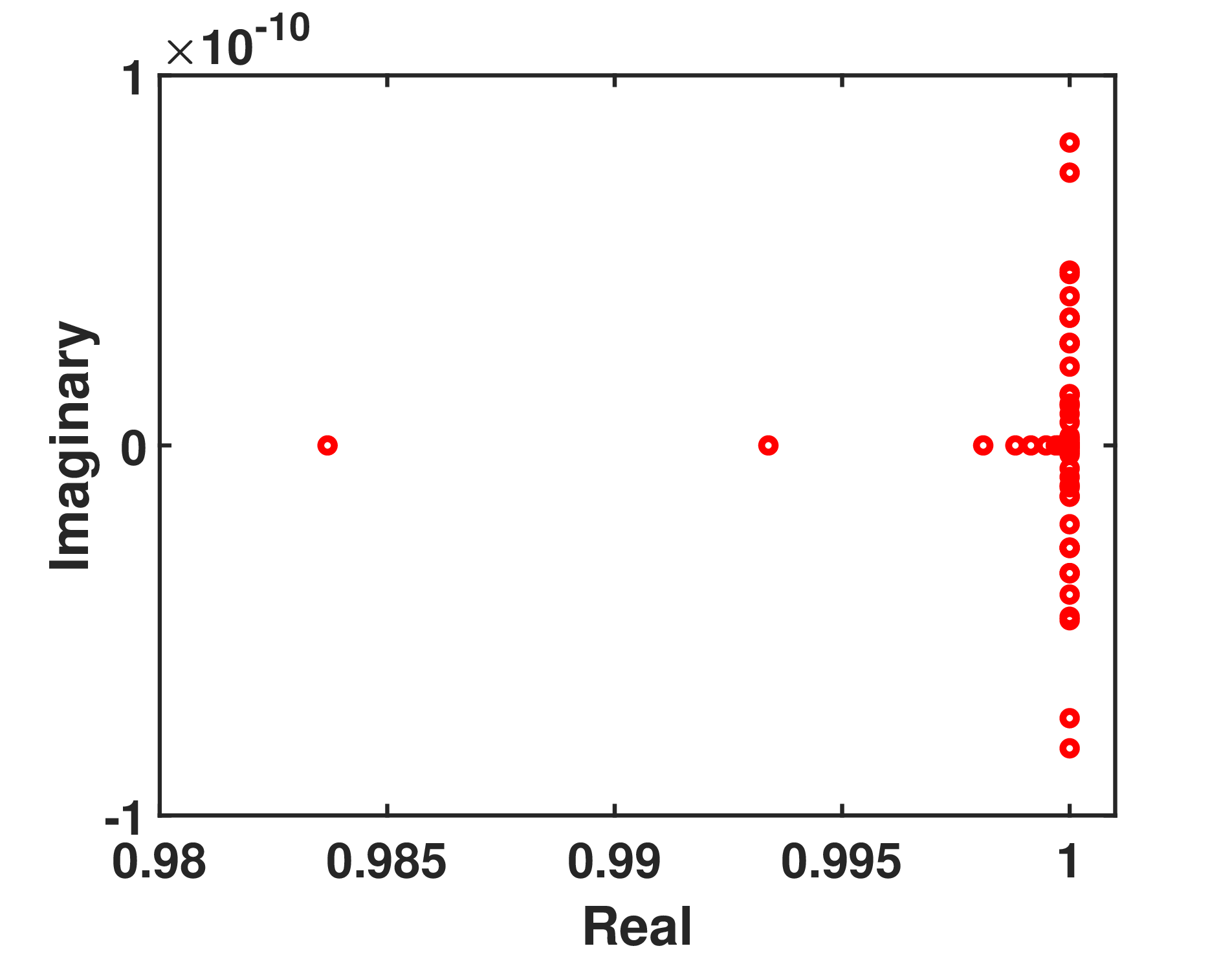}
					\caption{\footnotesize  $\P_{SL}^{-1}\A$ }
					\label{TBDP_STOKES}
				\end{subfigure}
				\begin{subfigure}[b]{0.3\textwidth}
					\centering
					\includegraphics[width=\textwidth]
					{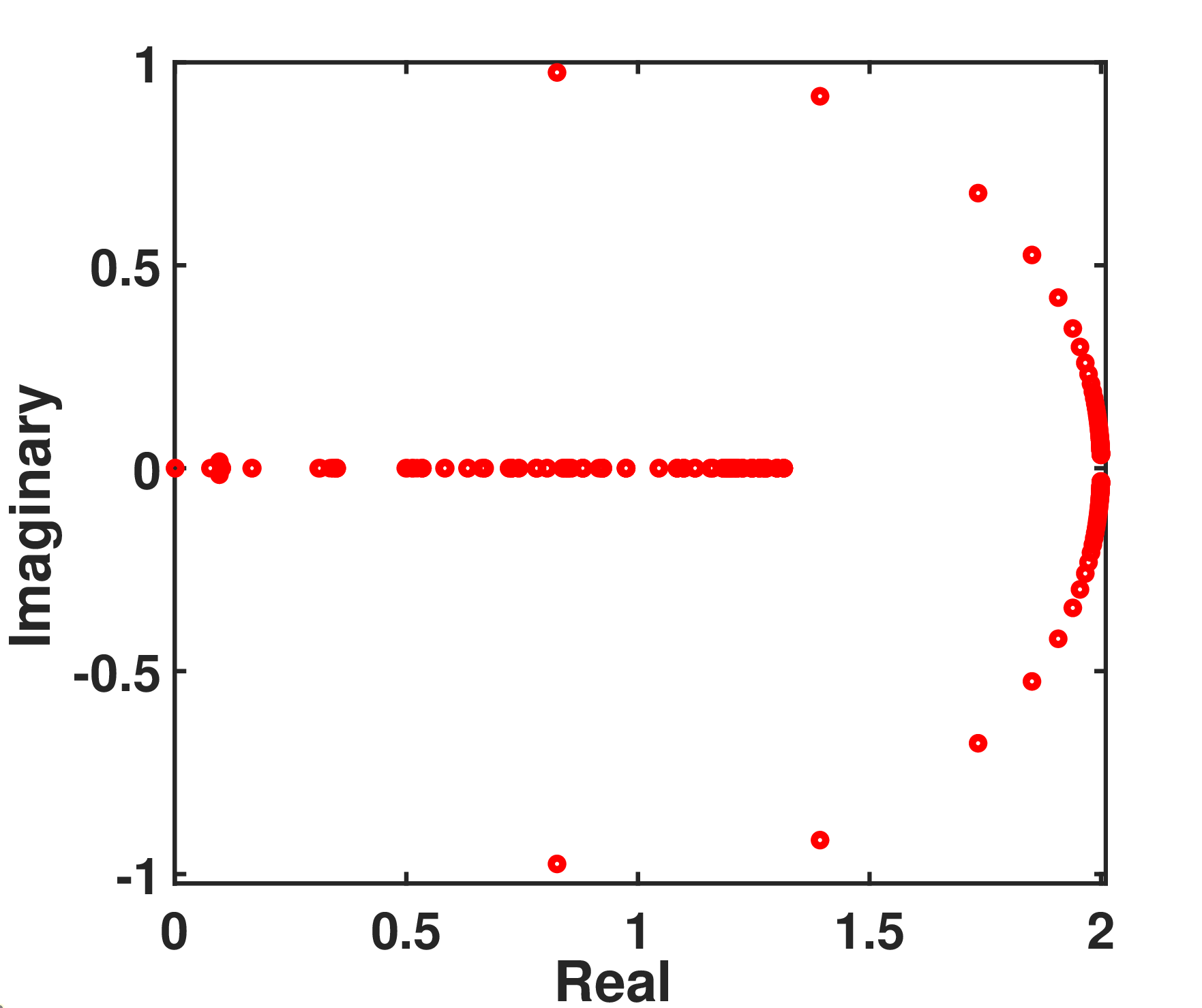}
					\caption{ \footnotesize $\P_{SS}^{-1}\A$ }
					\label{SS_STOKES}
				\end{subfigure}%
				\begin{subfigure}[b]{0.3\textwidth}
					\centering
					\includegraphics[width=\textwidth]{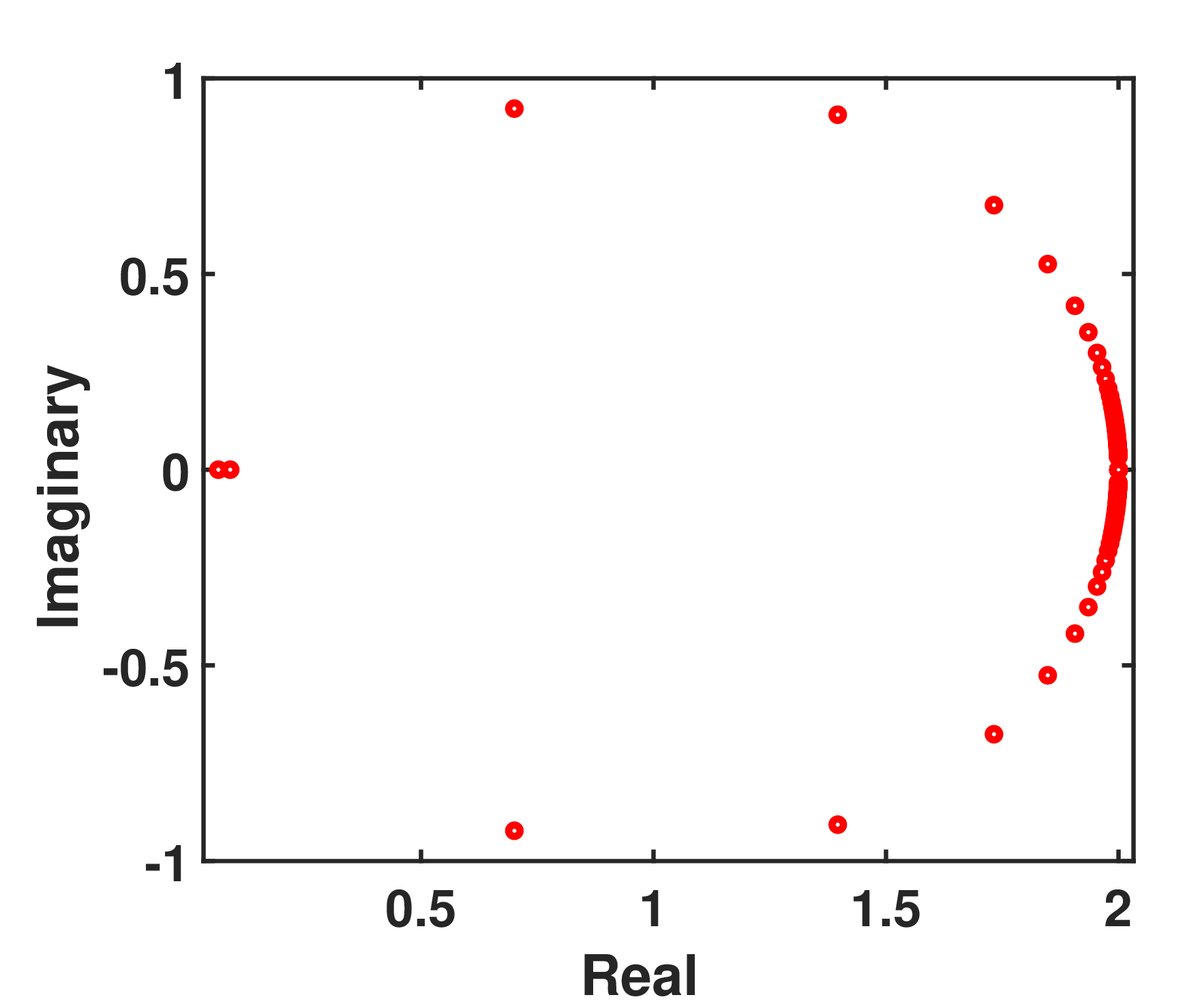}
					\caption{ \footnotesize  $\P_{RSS}^{-1}\A$}
					\label{RSS_STOKES}
				\end{subfigure}
				\begin{subfigure}[b]{0.3\textwidth}
					\centering
					\includegraphics[width=\textwidth]{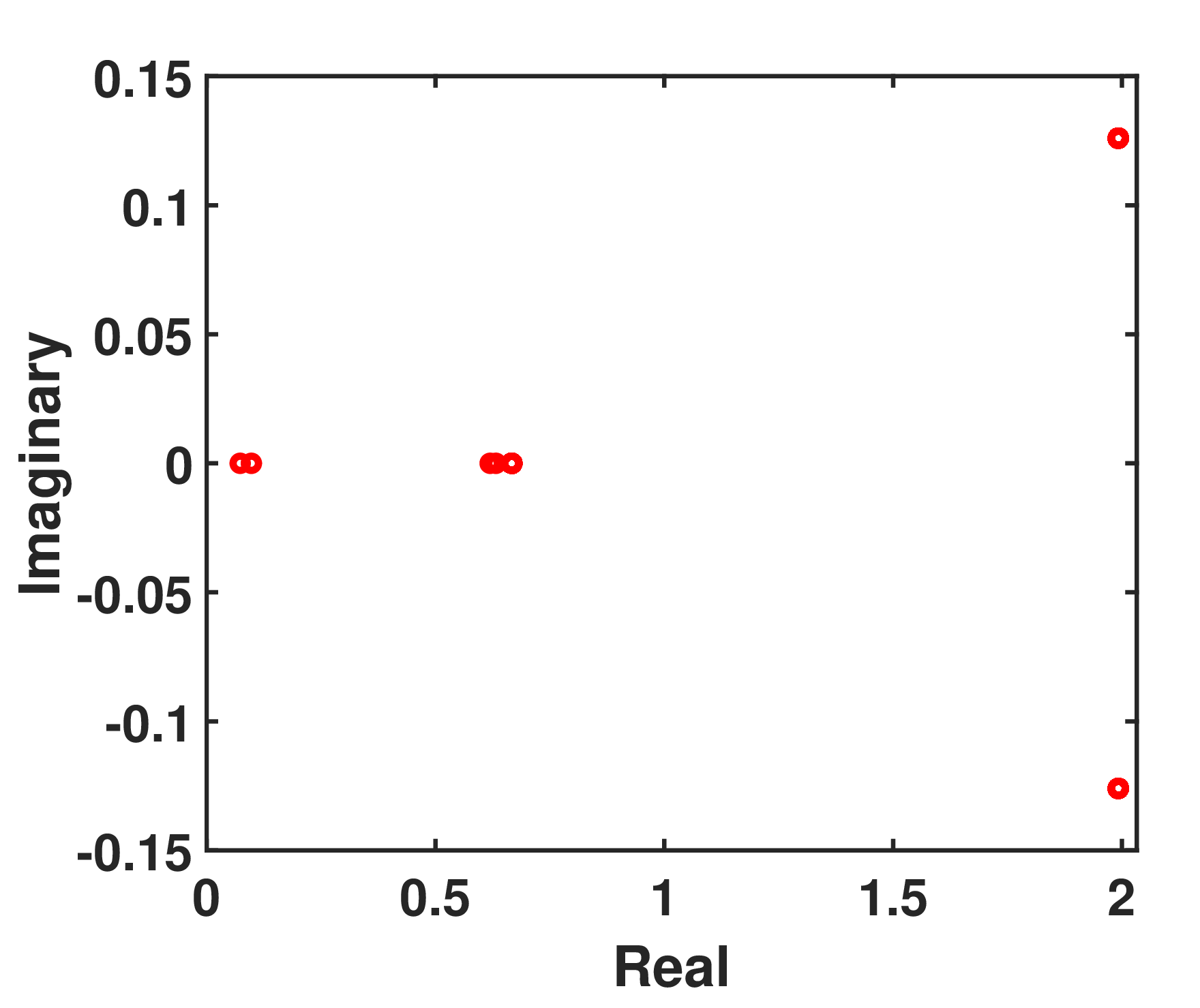}
					\caption{ \footnotesize $\P_{EGSS}^{-1}\A$}
					\label{EGSS_STOKES}
				\end{subfigure}
    \begin{subfigure}[b]{0.32\textwidth}
					\centering
					\includegraphics[width=\textwidth]{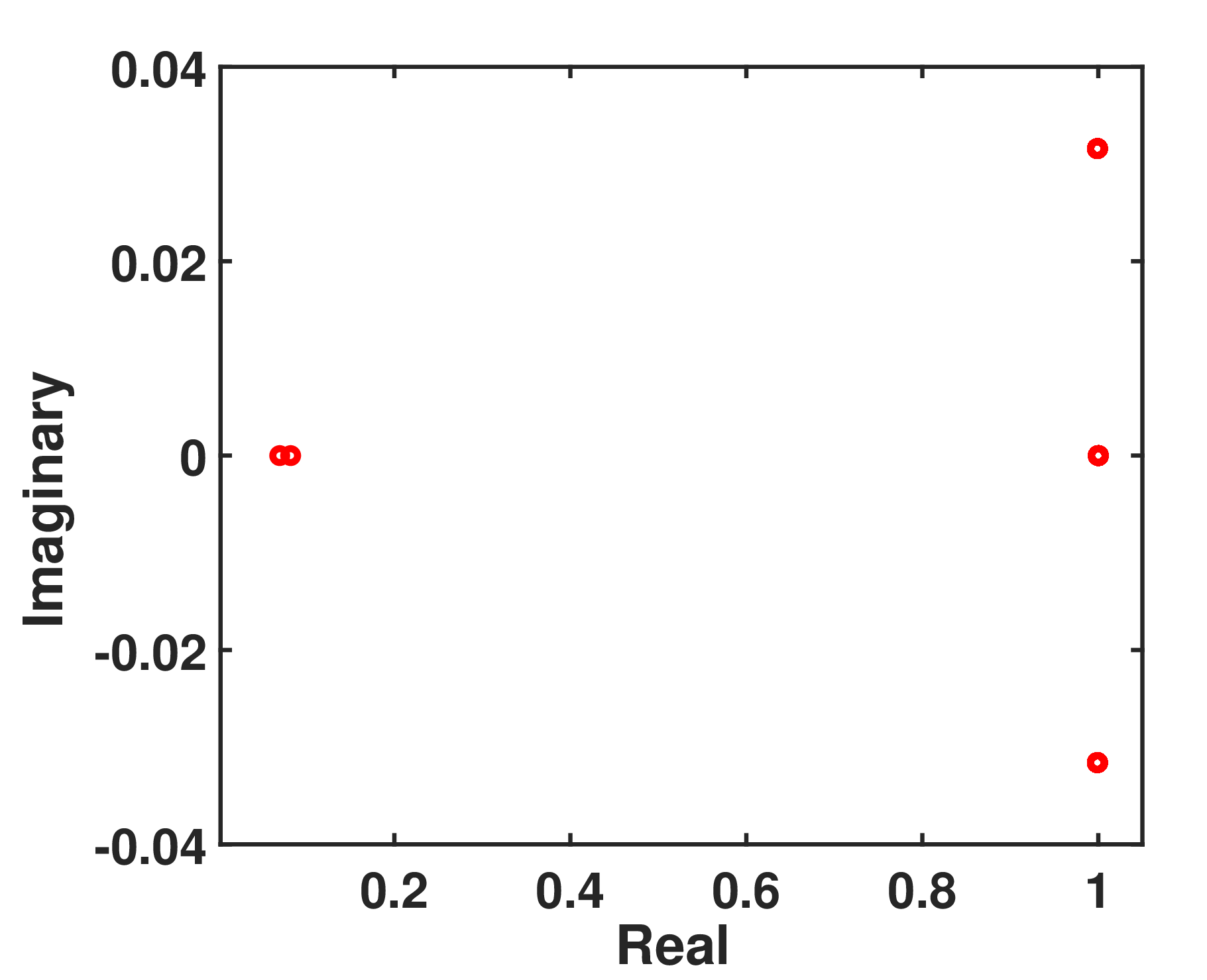}
					\caption{   $\P_{RPGSS}^{-1}\A$ }
					\label{fig:RPGSS}
				\end{subfigure}
				\begin{subfigure}[b]{0.3\textwidth}
					\centering
					 \includegraphics[width=\textwidth]{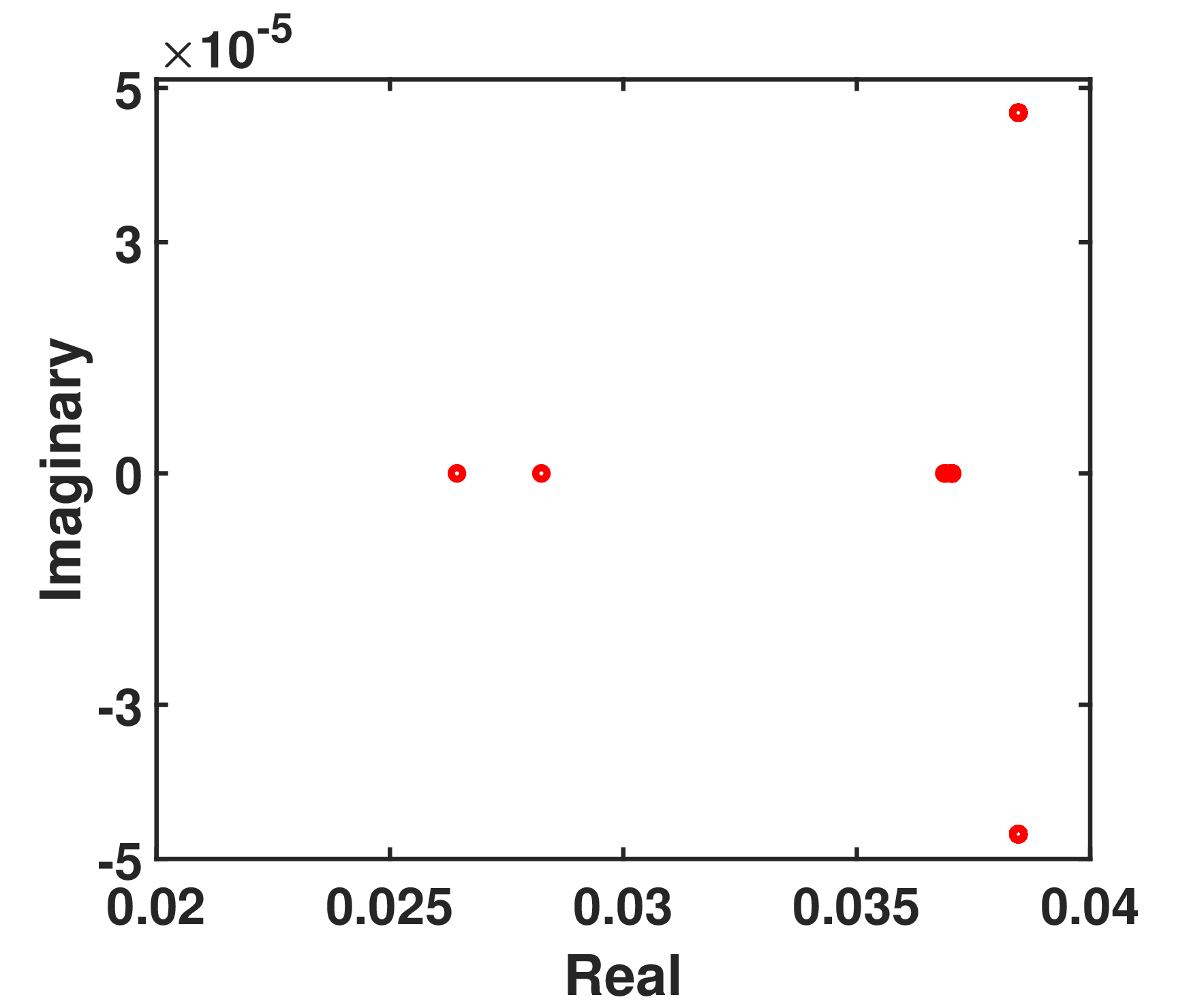}
					\caption{ \footnotesize   $\P_{PESS}^{-1}\A$ with $s=26$}
					\label{PESS_STOKES}
				\end{subfigure}
				\begin{subfigure}[b]{0.35\textwidth}
					\centering
					\includegraphics[width=\textwidth]{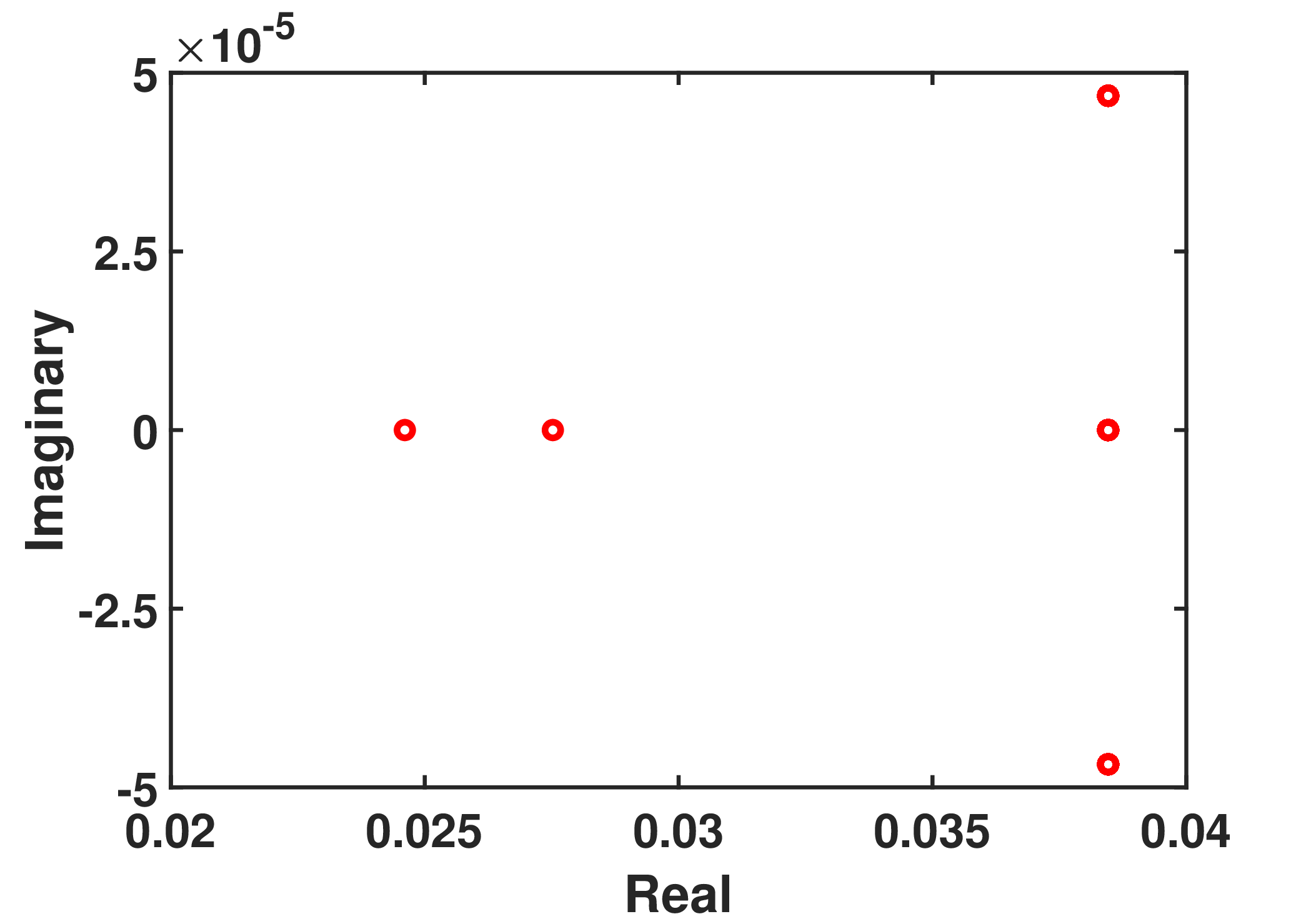}
					\caption{ \footnotesize   $\P_{LPESS}^{-1}\A$ with $s=26$}
					\label{LPESS_STOKES}
				\end{subfigure}
				\caption{Spectral distributions of $\A, \P_{IBD}^{-1}\A,  {\P_{MAPSS}^{-1}\A,  \P_{SL}^{-1}\A,} \P_{SS}^{-1}\A, \P_{EGSS}^{-1}\A,   {\P_{RPGSS}^{-1}\A,} \P_{PESS}^{-1}\A$ and $\P_{LPESS}^{-1}\A$ for Case II with ${\bf h}=1/8.$}
				\label{fig5}
			\end{figure}
  In Table \ref{tab3}, we list numerical  results produced using the {\it GMRES} process and {\it PGMRES} process with the preconditioners $\P_{IBD},  {\P_{MAPSS}, \P_{SL},} \P_{SS}, \P_{RSS}, \P_{EGSS},  {\P_{RPGSS},}$ $ \P_{PESS}$ and $\P_{LPESS}$  for different grid parameter values of $\textbf{h}.$
  
 \vspace{2mm}
\noindent   {\textbf{Parameter selection:}}  Parameter choices  {for the proposed} preconditioners are made in two cases as in Example \ref{ex1}. However, in Case I, we take $\alpha=0.01, \beta=0.1,$  $\Lambda_1=0.01 I$ and $\Lambda_2=0.1I.$
    The parameter $\alpha$ in Case I is taken as in \cite{CAOSS19} and the parameters in Case II are taken as in \cite{EGSS23}.  For the {\it IBD} preconditioner, the matrices $\widehat{A}$ and $\widehat{S}$ and for  {\textit{MAPSS} preconditioner $\alpha$ and $\beta$}  are constructed as in Example \ref{ex1}. 
   
 \vspace{2mm}
\noindent  {\textbf{Results for experimentally found optimal parameter:}} In the interval $[10, 30]$, the empirical optimal choice for  $s$ is found to be $30$ for Case I and $26$ for Case II. Table \ref{tab3} shows that the {\it GMRES} process has a very slow convergence speed and also does not converge for $\textbf{h}=1/64,$  {$1/128$} within $7000$ iterations. The proposed preconditioners require almost five times fewer iterations compared to the  {\it IBD} preconditioner for convergence. Moreover, in both cases, the \textit{PESS} and \textit{LPESS} preconditioners outperform the  {\textit{MAPSS, SL},} {\it SS, RSS, EGSS} and  {\it RPGSS} preconditioners in terms of IT and CPU times. 
 {Notably, for the \textit{PESS} preconditioner, the IT remains constant in both cases, whereas for the \textit{SS} and \textit{EGSS} preconditioners, the IT increases as the size of 
$\A$ increases.
Furthermore, in Case I with $\textbf{h}=1/64$, our proposed \textit{LPESS}  preconditioner is approximately  $45\%$, $33\%$, $34\%,$ $65\%,$ $35\%,$ $66\%$ and $36\%$ more efficient than the existing  {\it IBD, MAPSS, SL, SS, RSS, EGSS} and {\it RPGSS}  preconditioners, respectively.} Similar patterns are observed for other values of $\textbf{h}.$

   \begin{figure}[]
				\centering
				\begin{subfigure}[b]{0.38\textwidth}
					\centering
                         \includegraphics[width=\textwidth]{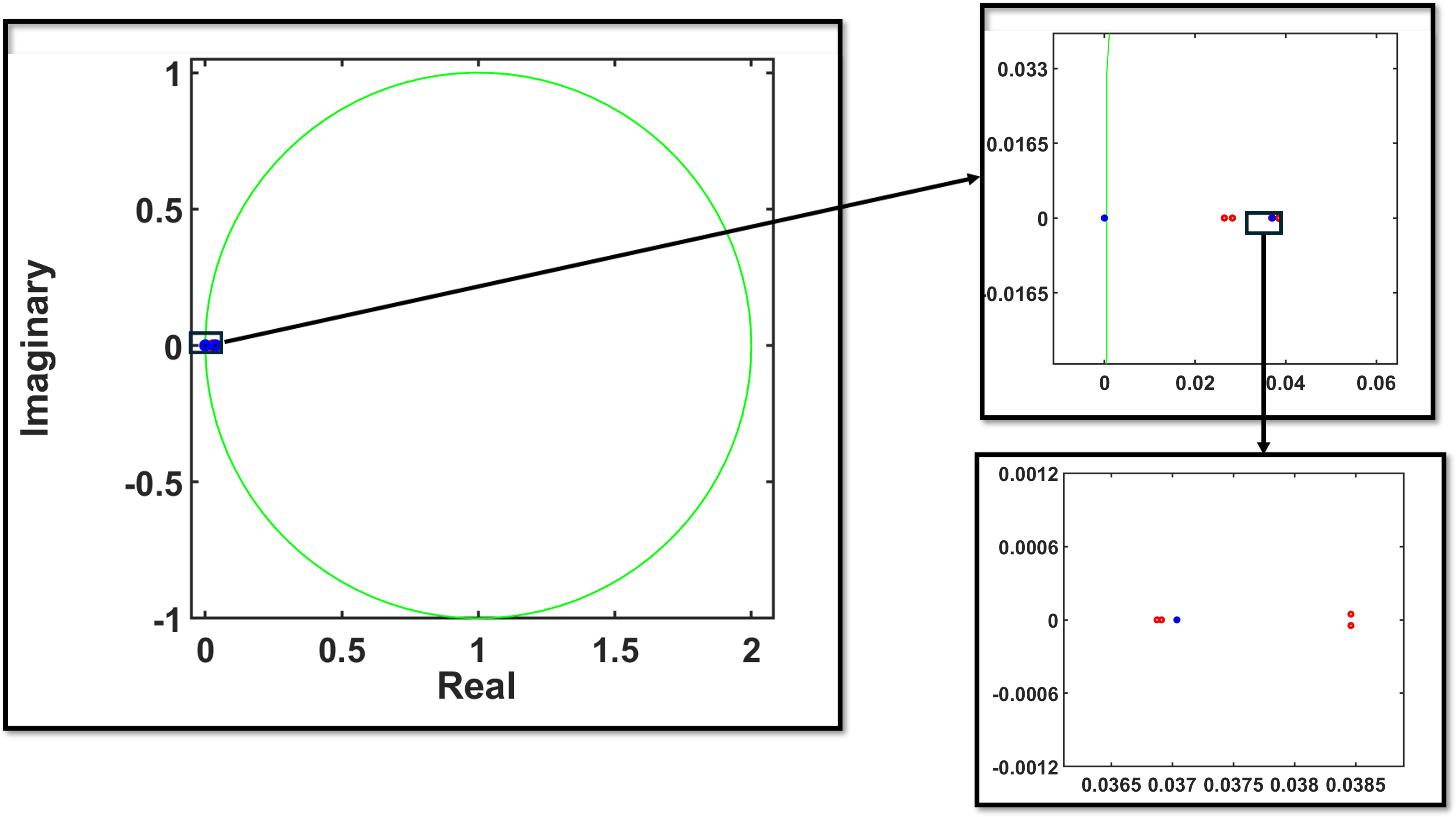}
					\caption{ \footnotesize   $\P_{PESS}^{-1}\A$ with $s=26$}
					\label{PESS_STOKES_bounds}
				\end{subfigure}
   \hfil
				\begin{subfigure}[b]{0.55\textwidth}
					\centering
                         \includegraphics[width=\textwidth]{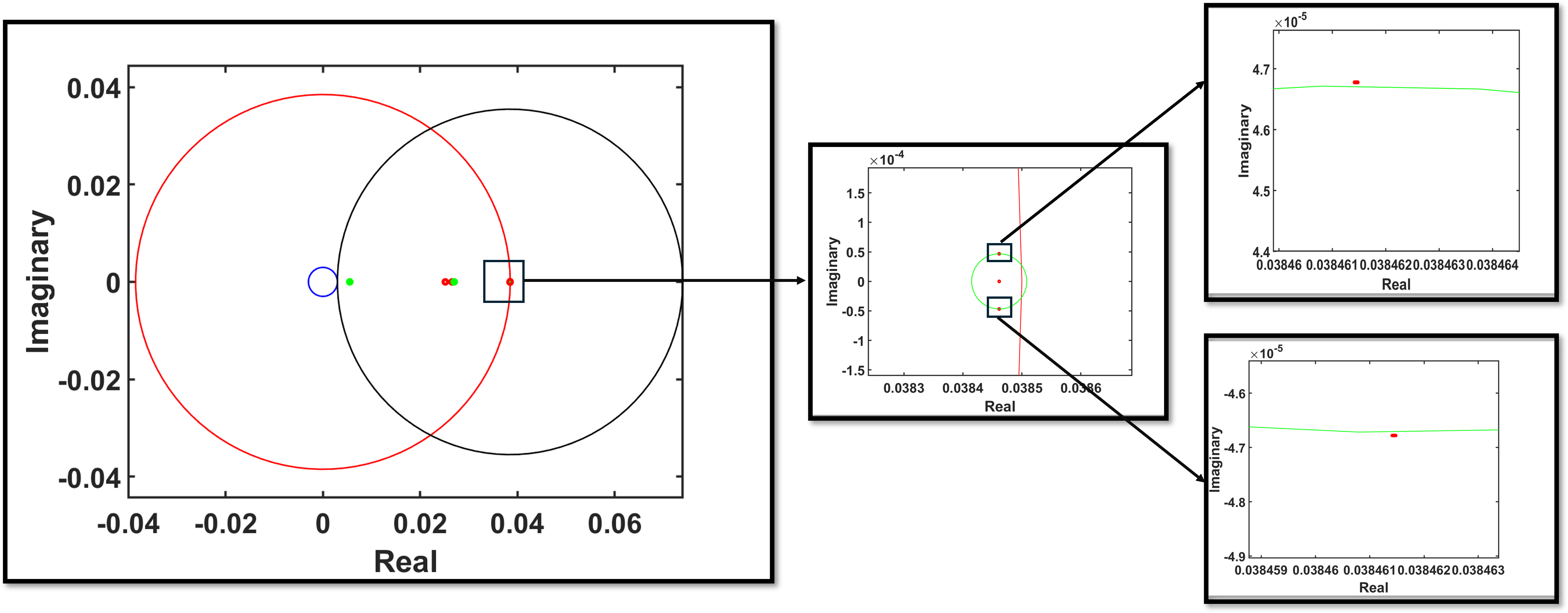}
					\caption{ \footnotesize   $\P_{LPESS}^{-1}\A$ with $s=26$}
					\label{LPESS_STOKES_bounds}
				\end{subfigure}
			\caption{{Spectral bounds for $ \P_{PESS}^{-1}\A$ and $\P_{LPESS}^{-1}\A$ for Case II with ${\bf h}=1/8$  for Example \ref{ex2}.}}
				\label{fig:stokesnew}
			\end{figure}
   \vspace{2mm}
    \noindent  {\textbf{Results using parameters selection strategy in Section \ref{Sec:parameter}:} To demonstrate the effectiveness of the proposed \textit{PESS} and \textit{LPESS} preconditioners using the parameters discussed in Section \ref{Sec:parameter}, we present the numerical test results for \textit{PESS-I}, \textit{LPESS-I}, \textit{PESS-II} and \textit{LPESS-II} in Table \ref{tab:rev2} as in Example \ref{ex1}. Comparing the results in Tables \ref{tab3} and \ref{tab:rev2}, we observe that the parameter selection strategy described in Section \ref{Sec:parameter} is effective.}
    \begin{table}[ht!]
       \centering
        \caption{ {Spectral bounds for  $\P_{PESS}^{-1}\A$ and $\P_{LPESS}^{-1}\A$ for case II with $l=16$ for Example \ref{ex2}}}
    {   \begin{tabular}{c|c|c}
       \hline 
         &  Real eigenvalue $\lambda$& Non-real eigenvalue $\lambda$ \\[1ex]
           \hline & Bounds of Theorem \ref{th42}&  Bounds of Theorem \ref{th43}\\[1ex]
      \hline
      \multirow{3}{*}{  $\P_{PESS}^{-1}\A$  }  & contained in  &  $0.0334 \leq |\lambda|\leq  0.375$  \\[0.5ex]
      &  the interval & $3.3624\times 10^{-6}\leq \Re(\lambda/(1-s\lambda))\leq 0.5$\\[0.5ex]
      &$(0, 0.0370]$ & $|\Im(\lambda/(1-s\lambda))|\leq 31.6386$\\[1ex]
      \hline 
      & Bounds of Theorem \ref{theorem:RPESS}&  Bounds of Theorem \ref{theorem:RPESS}\\[1ex]
      \hline
    \multirow{3}{*}{ $\P_{LPESS}^{-1}\A$  }    &contained in & $0.0030\leq |\lambda|\leq 0.0385$ \\[0.5ex]
    & the interval & $ 4.6722\times 10^{-5}\leq |\lambda-\frac{1}{s}|\leq 0.0355$\\[0.5ex]
    &$(0.00552, 0.0270]$ &\\[0.5ex]
    \hline
       \end{tabular}}
      
       \label{tab:eigenvalue_Ex2}
   \end{table}
   
  \vspace{2mm}
\noindent  {\textbf{Convergence curves:}}  The convergence curves in Figure \ref{fig4} demonstrate the rapid convergence of the proposed \textit{PESS} and \textit{LPESS} \textit{PGMRES}  processes compared to the  {\it IBD,  {MAPSS, SL,} SS, RSS, EGSS} and  {\it RPGSS}  in terms of RES versus IT counts.
 
 \vspace{2mm}
\noindent  {\textbf{Spectral distributions:}}
   Figure \ref{fig5}  illustrates the spectral distributions of the original matrix $\A,$ 
			and  the preconditioned matrices $\P_{IBD}^{-1}\A,$ $  {\P_{MAPSS}^{-1}\A,}$ $ {\P_{SL}^{-1}\A,}$ $ \P_{SS}^{-1}\A, \P_{RSS}^{-1}\A,$ $ \P_{EGSS}^{-1}\A,$ $ {\P_{PRGSS}^{-1}\A,}$ $\P_{PESS}^{-1}\A$  and $\P_{LPESS}^{-1}\A$ for Case II with ${\bf{h}}= 1/8.$ According to Figure \ref{fig5}, eigenvalues of  $\P_{PESS}^{-1}\A$ and $\P_{LPESS}^{-1}\A$ demonstrate a superior clustering compared to the other preconditioned matrices, which leads to a favorable convergence speed for the proposed \textit{PESS} and \textit{LPESS} {\it PGMRES} processes.  
   
\vspace{2mm}  
\noindent {\textbf{Spectral bounds:} In Figure \ref{fig:stokesnew}, we draw the spectral bounds of Theorems \ref{th41}, \ref{th42} and \ref{theorem:RPESS} for  \textit{PESS} and \textit{LPESS} preconditioned matrices. In Figure \ref{fig:stokesnew}(a),  $|\lambda-1|=1$ of Theorem \ref{th41} is drawn by the green unit circle, while the points in blue indicate the bounds from Theorem \ref{th42}. Moreover, we draw the circles  $C_1$ (in red), $C_2$ (in blue), $C_3$ (in black) and $C_4$ (in green) in Figure \ref{fig:stokesnew}(b). We can observe that the eigenvalues of the preconditioned matrix  $\P_{LPESS}^{-1}\A$ lie in the intersection of the annulus of the circles $C_1$ and $C_2,$ and the annulus of the circles $C_3$ and  $C_4.$  Additionally, the eigenvalue bounds derived in Theorems \ref{th42}, \ref{th43} and \ref{theorem:RPESS} are listed in Table \ref{tab:eigenvalue_Ex2}}.

     \begin{figure}[ht!]
				\centering
				\begin{subfigure}[b]{0.4\textwidth}
					\centering					\includegraphics[width=\textwidth]{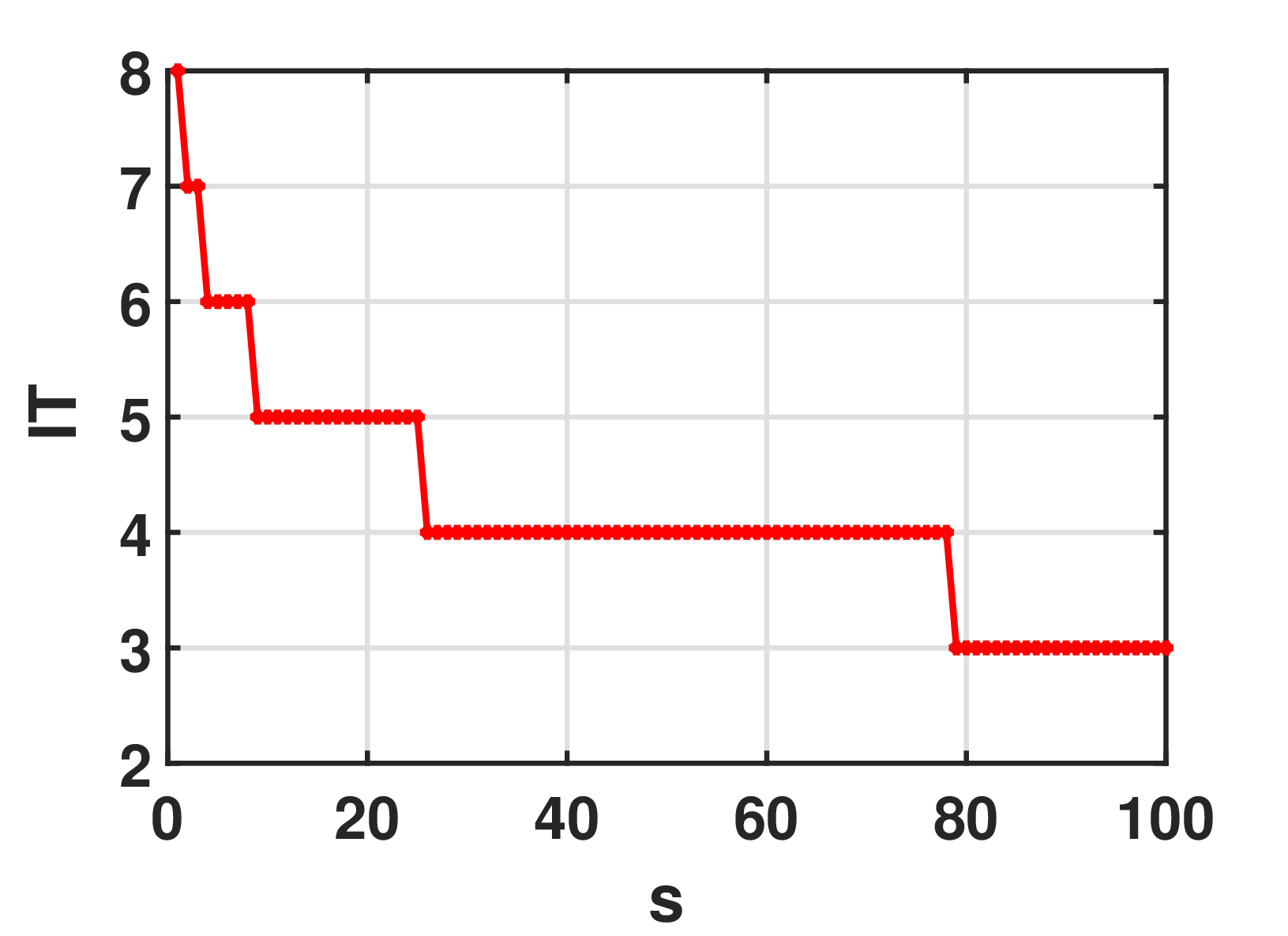}
					\caption{\footnotesize  \textit{PESS} \textit{PGMRES} process}
					\label{cpu1}
				\end{subfigure}
				\begin{subfigure}
					[b]{0.4\textwidth}
					\centering
					\includegraphics[width=\textwidth]{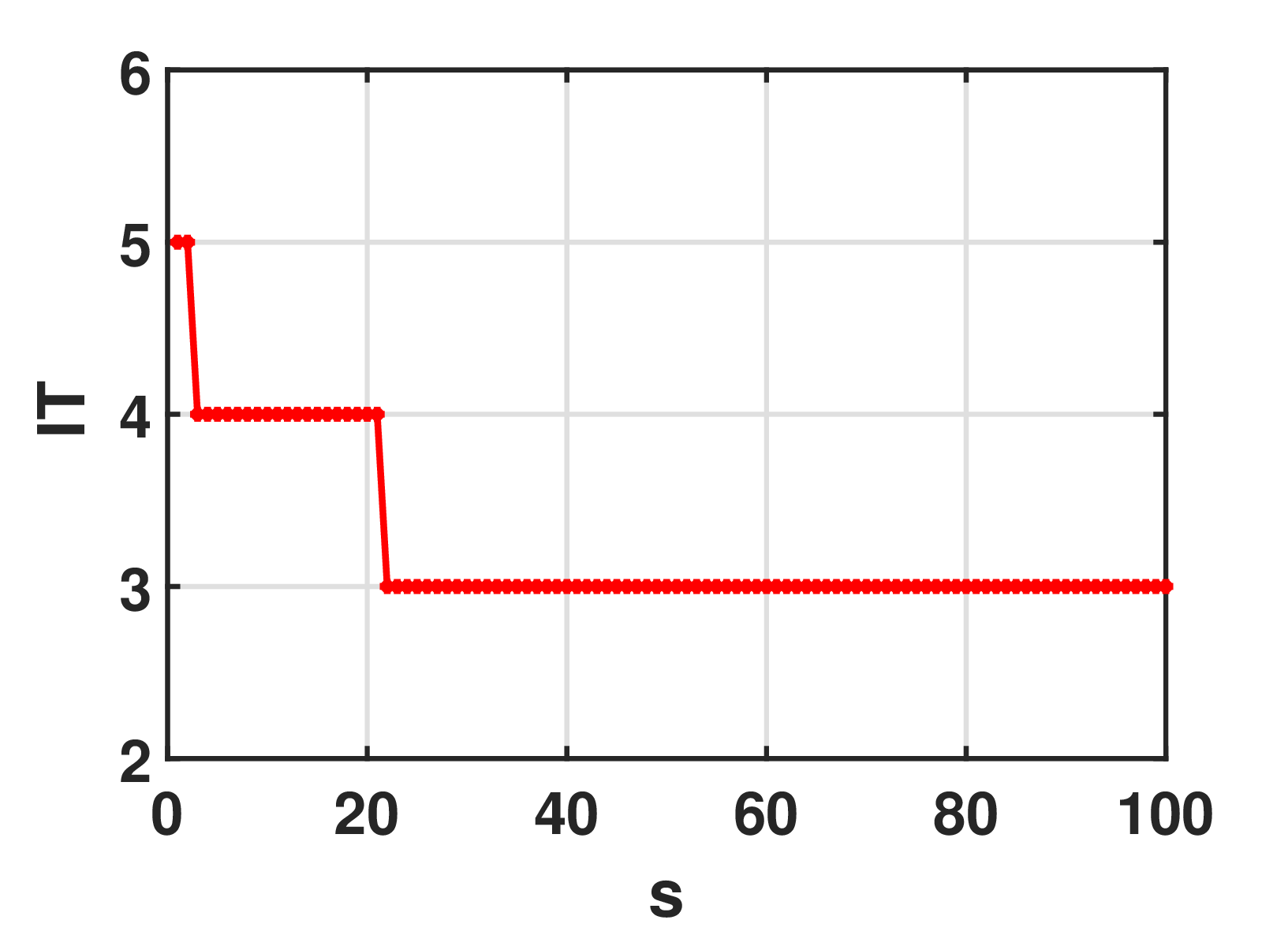}
					\caption{\footnotesize  \textit{LPESS} \textit{PGMRES} process}
					\label{IT2}
				\end{subfigure}
				\caption{Characteristic curves for   IT  of the proposed \textit{PESS} (left) and \textit{LPESS} (right) \textit{PGMRES} processes by varying $s$ in the interval $[1, 100]$ with step size $1$ with $\textbf{h}= 1/16$ in Case I for Example \ref{ex2}.}
				\label{IT_vs_s}
			\end{figure}
 \vspace{2mm}
\noindent  {\textbf{Condition number analysis:}}  Furthermore, the system \eqref{eq11} exhibits ill-conditioning nature with $\kappa(\A)= 5.0701e+05.$  While for Case I with $\mathbf{h}=1/16$, $\kappa(\P_{PESS}^{-1}\A)=1.3353$ and  $\kappa(\P_{LPESS}^{-1}\A)=1.2056,$  indicating that the proposed \textit{PESS} and \textit{LPESS} preconditioned systems are well-conditioned, ensuring an efficient and robust solution. 

 \vspace{2mm}
\noindent  {\textbf{Relationship between $s$ and convergence speed:}} In addition, to demonstrate the relationship of the parameter $s$ and the convergence speed of the \textit{PESS} and \textit{LPESS} preconditioners, we plot graphs of IT counts by varying the parameters $s$ in the interval  $[1,100]$ with step size one in Figure \ref{IT_vs_s}. We consider $\textbf{h}=1/16$ and other choices for $\Lambda_1,\Lambda_2$ and $\Lambda_3$ as in Case I. Figure \ref{IT_vs_s} shows that, with the increasing value of $s,$  decreasing trend in the IT counts for both the \textit{PESS} and \textit{LPESS} preconditioner.  Moreover, for $s>22,$  using \textit{LPESS} preconditioner, we can solve this three-by-three block $SPP$  only in $3$ iterations.
		\end{exam}	
 {  \begin{exam}\label{exam3}\cite{OPTMIZATION1,  HuangNA}
      \begin{table}[ht!]
					\centering
			\caption{ Numerical results of {\it GMRES}, {\it IBD},  {\textit{MAPSS}, \textit{SL},} {\it SS}, {\it RSS}, {\it EGSS},  {\it RPGSS,} \textit{PESS} and {\it LPESS PGMRES} processes for Example \ref{exam3}.}
			\label{tab:exam3}
			\resizebox{13cm}{!}{
			 {	\begin{tabular}{cccccc}
					\toprule
					Process& Problem	&aug2dc& aug3dc & listwet12 & yao   \\
					\midrule 
				& size($\A$)	& $50400$ &$8746$& $30004$ & $6004$   \\
					\midrule
					\multirow{3}{*}{\it GMRES}& IT& $370$& $78$& $2251$&$49$ \\
					& CPU&  $360.1613$ & $14.2486$& $422.8209$& $13.9152$\\
					&{\tt RES}&$9.9905e-07$&$9.2839e-08$&$9.9700e-08$&$9.7483e-08$ \\	
					\midrule
					\multirow{3}{*}{\it IBD}& IT& $5$&$5$ & $28$& $25$ \\
					& CPU&$278.7707$ &$4.80051$ & $278.8961$& $21.9511$\\
					&{\tt RES}&$2.1187e-15$ & $7.5842e-09$ & $8.3961e-07$ & $9.4491e-07$  \\
      \midrule
					\multirow{3}{*}{ {\it MAPSS}}&  {IT}&  {$8$} & {$8$} & {$124$} &  {$36$}   \\
					&  {CPU}&  {$324.3828$}&
 {$7.6662$}&
 {$2078.9607$}&
 {$18.6853$}\\
					&{ {\tt RES}}& {$3.6981e-07$}&
 {$1.7312e-09$}&
 {$9.2628e-07$}&
 {$7.6716e-07$} \\
 \midrule
					\multirow{3}{*}{ {\textit{SL}}}&  {IT}&  {$65$} & {$14$} & {$15$} &  {$37$}    \\
					&  {CPU}&  {$2848.2117$}&
 {$14.4158$}&
 {$270.3901$}&
 {$36.7559$}\\
					&{ {\tt RES}}& {$8.4621e-08$}&
 {$2.0305e-08$}&
 {$8.2851e-08$}&
 {$9.8555e-09$}  \\
					\midrule
					\multirow{3}{*}{\it SS}& IT& $23$& $6$& $7$& $39$\\
					& CPU&$1547.7320$&$8.9176$&$200.8205$&$28.9668$  \\
					&{\tt RES}&$9.9399e-07$&
$7.7120e-07$&
$3.5756e-07$&
$9.5393e-07$ \\
					\midrule
					\multirow{3}{*}{\it RSS}& IT& $23$& $6$& $5$& $16$  \\
					& CPU&$1774.9981$&
$8.6199$&
$161.8405$&
$18.5057$\\
					&{\tt RES}&$5.5282e-07$&
$2.6941e-07$&
$4.1147e-07$&
$8.2434e-07$\\
					\midrule
					\multirow{3}{*}{\it EGSS}& IT& $5$& $5$& $4$& $6$  \\
					& CPU&$350.3081$&
$7.61206$&
$124.0637$&
$6.9551$ \\
					&{\tt RES}&$3.9934e-12$&
$1.1711e-12$&
$1.5556e-07$&
$2.0227e-07$  \\
\midrule
					\multirow{3}{*}{ {\it RPGSS}}&  {IT}&  {$4$} & {$4$} & {$3$} &  {$5$}  \\
					&  {CPU}&  {$297.3382$}&
 {$6.16249$}&
 {$96.5956$}&
 {$5.0133$}  \\
					&{ {\tt RES}}& {$3.2268e-09$}&
 {$2.8160e-10$}&
 {$7.1294e-08$}&
 {$9.2318e-08$} \\
					\midrule
					\multirow{3}{*}{\textit{PESS}$^{\dagger}$}& IT& $3$&$3$ &$2$ & $3$\\
					& CPU&$249.0203$&
$5.5716$&
$68.5225$&
$4.8778$  \\
					$s=30$ &\tt RES&$8.3682e-08$&
$5.7233e-08$&
$8.3159e-07$&
$3.3146e-07$ \\
					\midrule
					\multirow{3}{*}{\textit{LPESS}$^{\dagger}$}& IT& $3$& $3$& $2$& $3$ \\
					& CPU&$237.0945$&
$5.4662$&
$67.0810$&
$3.9688$ \\
					$s=30$ &{\tt RES}&$1.0780e-09$&
$8.4676e-09$&
$4.3037e-07$&
$3.7858e-07$ \\
					\bottomrule
					\multicolumn{6}{l}{Here ${\dagger}$ represents the proposed preconditioners. The boldface represents the top }\\
					\multicolumn{6}{l}{two results. $\bf{--}$ indicates that the iteration process does not converge }\\ \multicolumn{6}{l}{within the prescribed IT.}
				\end{tabular}}
 			}
		\end{table} 
   \textbf{Problem formulation:}     In this example, we consider the three-by-three block \textit{SPP} \eqref{eq11}  arising from the following quadratic programming problem: 
      \begin{align}
          &\min_{x\in \R^n,\, y\in \R^p}\frac{1}{2}x^TAx+f^Tx+h^Ty\\ 
          & \text{such that}~~ Bx+C^Ty=b,
      \end{align}
      where $f\in \R^n$ and $h\in \R^p.$ To solve the minimization problem, we define the Lagrange function: 
      \begin{equation}
          L(x,y,\lambda)=\frac{1}{2}x^TAx+f^Tx+h^Ty+\lambda^T(Bx+C^Ty-b),
      \end{equation}
      where $\lambda\in \R^m.$ Then the Karush-Kuhn-Tucker \cite{KKT} necessary conditions are as follows: 
      \begin{eqnarray*}
          \nabla_{x}L(x,y,\lambda)=0,~ \nabla_{y}L(x,y,\lambda)=0 ~\text{and}~ \nabla_{\lambda}L(x,y,\lambda)=0,
      \end{eqnarray*}
      which can be reformulated as a three-by-three block \textit{SPP} of the form \eqref{eq11}.
      \begin{table}[ht!]
       \centering
         \caption{  Numerical results of \textit{PESS-I}, \textit{LPESS-I}, \textit{PESS-II} and \textit{LPESS-II} \textit{PGMRES} processes for Example \ref{exam3}.}
          {\resizebox{12cm}{!}{
         \begin{tabular}{ccccccc}
						\toprule
					Process& Problem	&aug2dc& aug3dc & listwet12 & yao   \\
					\midrule 
				& size($\A$)	& $50400$ &$8746$& $30004$ & $6004$   \\
    \midrule
						\multirow{1}{*}{{\it PESS-I}}& IT& $5$&$3$&$5$&$6$\\
					($s=1,$ $\Lambda_1=0.001I,$	& CPU&  $520.7330$&$10.1460$&$310.5960$&$5.7449$ \\
					$\Lambda_2=0.1I,$ $\Lambda_3=0.001I$)	&{\tt RES}&$ 5.4476e-08$&
$1.1047e-08$&
$1.5581e-07	$&
$2.9433e-07$\\		 
      \midrule
      \multirow{1}{*}{{\it LPESS-I}}& IT& $4$&$3$&$3$&$3$\\
      			($s=1,$	$\Lambda_2=0.1I,$	& CPU&  $337.9530$&$8.3194$&$118.2540$&$3.5354$ \\
        $\Lambda_3=0.001I$)    &{\tt RES}&  $5.4773e-09$ & $2.8655e-07$ & $2.0805e-08$&$ 1.5576e-07$& \\
      \midrule
      \multirow{1}{*}{{\it PESS-II}}& IT& $3$&$3$&$2$&$3$\\
				($s=s_{est},$ $\Lambda_1=A,$		& CPU&  $236.4489$&$8.0818$&$74.9301$&$3.7349$\\ 	
  $\Lambda_2=\beta_{est}I,$ $\Lambda_3=10^{-4}CC^T$) &{\tt RES}&  $3.0982e-08$ & $5.8697e-09$ & $7.8956e-09$&$ 6.3723e-07$\\
   \midrule 	
      \multirow{1}{*}{{\it LPESS-II}}& IT& $3$&$3$&$2$&$2$\\
			($s=s_{est},$	$\Lambda_2=\beta_{est}I,$		& CPU&  $282.3466$&$8.3194$&$114.7483$&$2.8330$\\ 
   $\Lambda_3=10^{-4}CC^T$)    &{\tt RES}&  $3.7409e-09$&$1.0390e-09$&$1.0465e-08$&$6.8900e-08$\\
      \midrule
       \end{tabular}}}
       \label{tab:rev3}
   \end{table}
 Data matrices $A$, $B$ and $C$ are selected from the CUTEr collection \cite{cuter2003}. The numerical results for the problems ‘aud2dc’, ‘aud3dc’, ‘liswet12’ and ‘yao’ are presented in Table \ref{tab:exam3}.
   
   \vspace{2mm}
\noindent \textbf{Parameter selection:} The parameter selection for the \textit{IBD} and \textit{MAPSS} preconditioners are as in Example \ref{ex1}. For the \textit{SS}, \textit{RSS},\textit{ EGSS}, \textit{RPGSS}, \textit{PESS} and \textit{LPESS} preconditioners, the parameters are chosen the same as in Case II of Example 1 with $\alpha = 0.1$ and $\Lambda_1 = 0.1I$.

 \vspace{2mm}
\noindent\textbf{Results for experimentally found optimal parameter:}  For the proposed preconditioner, the experimentally found parameter $s$ is $30$ in the interval $[10, 30].$  The numerical results presented in Table \ref{tab:exam3} demonstrate that the proposed \textit{PESS} and \textit{LPESS} preconditioners significantly outperform existing baseline preconditioners in terms of IT and CPU time. The proposed preconditioners can solve the there-by-three block \textit{SPP} within $2$ or $3$ iterations. Furthermore, for the problem `liswet12',  our proposed \textit{LPESS}  preconditioner is approximately  $76\%$, $97\%$, $75\%,$ $66\%,$ $59\%,$ $46\%$ and $31\%$ more efficient than the existing  {\it IBD, MAPSS, SL, SS, RSS, EGSS} and {\it RPGSS}  preconditioner, respectively.
  
  \vspace{2mm}
\noindent \textbf{Results using parameters selection strategy in Section \ref{Sec:parameter}: } To further illustrate the effectiveness of the proposed preconditioners under the parameter selection strategy discussed in Section \ref{Sec:parameter}, we reported the numerical results in Table \ref{tab:rev3} for the preconditioners \textit{PESS-I}, \textit{LPESS-I}, \textit{PESS-II} and  \textit{LPESS-II} (as discussed in Example $1$). A comparison of the results in Table \ref{tab:rev3} with the existing baseline preconditioners listed in Table \ref{tab:exam3} indicates that the proposed preconditioners perform well when the parameter $s$ is selected as outlined in Section \ref{Sec:parameter}.
      \end{exam}
      }
      
  \section{Conclusion}\label{sec6}
		In this paper, we proposed the \textit{PESS} iterative process and corresponding \textit{PESS} preconditioner and its relaxed variant \textit{LPESS} preconditioner to solve the three-by-three block {\it SPP} (\ref{eq11}). For the convergence of the proposed \textit{PESS} iterative process, necessary and sufficient criteria are derived. Moreover, we estimated the spectral bounds of the proposed \textit{PESS} and \textit{LPESS} preconditioned matrices.  This empowers us to derive spectral bounds for {\it SS} and {\it EGSS} preconditioned matrices, which have not yet been investigated in the literature to the best of our knowledge.  Numerous experimental analyses are performed to demonstrate the effectiveness of our proposed \textit{PESS} and \textit{LPESS} preconditioners. The key observations are as follows: $(i)$ the proposed \textit{PESS} and \textit{LPESS} preconditioners are found to outperform the existing baseline preconditioners in terms of IT and CPU times. $(ii)$ The proposed preconditioners significantly reduces the condition number of $\A$, consequently showing  {their} proficiency in solving three-by-three block {\it SPPs}. $(iii)$ The proposed \textit{PESS} and \textit{LPESS} preconditioned matrices have  {better} clustered spectral distribution than the baseline preconditioned matrices. $(iv)$ Sensitivity analysis conducted by introducing different percentages of noise on the system \eqref{eq11} showcases the robustness of the proposed \textit{PESS} preconditioner. 
		
  \section*{Declaration of competing interest}
		The authors disclose that they have no conflicting interests.
		
  \section*{Acknowledgments}
		Pinki Khatun sincerely acknowledges the Council of Scientific $\&$ Industrial Research (CSIR) in New Delhi, India, for financial assistance in the form of a fellowship (File no. 09/1022(0098)/2020-EMR-I).
		\bibliography{mybib}
		\bibliographystyle{abbrvnat}
	
 \end{document}